\documentclass[11pt]{amsart}

\usepackage{enumerate, amsmath, amsthm, amsfonts, amssymb, mathrsfs, graphicx, paralist, lscape}
\usepackage[usenames, dvipsnames]{color}
\usepackage[margin=1in]{geometry} 
\usepackage{tikz}
\usetikzlibrary{matrix,arrows}

\numberwithin{equation}{section}
\newtheorem{theorem}[equation]{Theorem}

\newtheorem{proposition}[equation]{Proposition}
\newtheorem{lemma}[equation]{Lemma}
\newtheorem{corollary}[equation]{Corollary}

\theoremstyle{definition}
\newtheorem{rmk}[equation]{Remark}
\newenvironment{remark}[1][]{\begin{rmk}[#1] \pushQED{\qed}}{\popQED \end{rmk}}
\newtheorem{rmks}[equation]{Remarks}

\newtheorem{eg}[equation]{Example}

\newtheorem{defn}[equation]{Definition}
\newenvironment{definition}[1][]{\begin{defn}[#1]\pushQED{\qed}}{\popQED \end{defn}}
\newtheorem{ques}[equation]{Question}

\newtheorem{notn}[equation]{Notation}

\usepackage[colorlinks=true, pdfstartview=FitV, linkcolor=blue, citecolor=blue, urlcolor=blue]{hyperref}
\newcommand{\cB}{\mathcal{B}}

\newcommand{\cG}{\mathcal{G}}
\newcommand{\fG}{\mathfrak{G}}

\newcommand{\cK}{\mathcal{K}}

\newcommand{\cL}{\mathcal{L}}

\newcommand{\cM}{\mathcal{M}}

\newcommand{\cO}{\mathcal{O}}

\newcommand{\cQ}{\mathcal{Q}}

\newcommand{\fS}{\mathfrak{S}}

\newcommand{\ba}{\mathbf{a}}

\newcommand{\bb}{\mathbf{b}}

\newcommand{\bc}{\mathbf{c}}

\newcommand{\bd}{\mathbf{d}}

\newcommand{\be}{\mathbf{e}}

\newcommand{\fp}{\mathfrak{p}}

\newcommand{\bs}{\mathbf{s}}

\newcommand{\bt}{\mathbf{t}}

\newcommand{\bu}{\mathbf{u}}

\newcommand{\bv}{\mathbf{v}}

\newcommand{\bw}{\mathbf{w}}

\newcommand{\bx}{\mathbf{x}}

\newcommand{\bz}{\mathbf{z}}

\newcommand{\zD}{\ensuremath{\Delta}}

\newcommand{\zO}{\ensuremath{\Omega}}
\newcommand{\za}{\ensuremath{\alpha}}
\newcommand{\zb}{\ensuremath{\beta}}

\newcommand{\zg}{\ensuremath{\gamma}}
\newcommand{\zl}{\ensuremath{\lambda}}

\newcommand{\ZZ}{\mathbb{Z}}

\renewcommand{\phi}{\varphi}

\renewcommand{\tilde}[1]{\widetilde{#1}}

\newcommand{\comment}[1]{}
\newcommand{\setst}[2]{\{#1\ | \ #2 \}}

\makeatletter
\def\Ddots{\mathinner{\mkern1mu\raise\p@
\vbox{\kern7\p@\hbox{.}}\mkern2mu
\raise4\p@\hbox{.}\mkern2mu\raise7\p@\hbox{.}\mkern1mu}}
\makeatother

\DeclareMathOperator{\codim}{codim}

\newcommand{\Hom}{\operatorname{Hom}}

\DeclareMathOperator{\rank}{rank}

\DeclareMathOperator{\Spec}{Spec}

\DeclareMathOperator{\Mat}{Mat}

\newcommand{\GL}{\mathbf{GL}}

\newcommand{\into}{\hookrightarrow}

\usetikzlibrary{calc}
\newcommand{\latticepath}[4]{ % #1 = box diameter
  \coordinate (L) at #2;
  \foreach \x/\y/\a in {#4} {
    \coordinate (L1) at ($ (L) + ( #1 * \x , #1 * \y ) $);
    \draw[color=\a, #3] (L) -- (L1);
    \coordinate (L) at (L1);
  }
}

\newcommand{\pipedream}[5]{ % #1 = box diameter
  \coordinate (P) at ($ #2 - (#1, 0 ) $);

  \coordinate (Q) at ($ (P) + ( 1.5 * #1, 0.5 * #1 ) $);
  \foreach \object in {#3} {
    \node at (Q) {\object};
    \coordinate (Q) at ($ (Q) + ( #1, 0 ) $);
  }

  \coordinate (Q) at ($ (P) + ( 0.5*#1, -0.5*#1 ) $);
  \foreach \object in {#4} {
    \node at (Q) {\object};
    \coordinate (Q) at ($ (Q) + ( 0, -#1 ) $);
  }

  \foreach \x/\y/\z/\a/\b in {#5} {
    \coordinate (P1) at ($ (P) + ( #1 * \y , -#1 * \x ) + ( 0      , #1 / 2 ) $);
    \coordinate (P2) at ($ (P) + ( #1 * \y , -#1 * \x ) + ( #1     , #1 / 2 ) $);
    \coordinate (P3) at ($ (P) + ( #1 * \y , -#1 * \x ) + ( #1 / 2 , #1     ) $);
    \coordinate (P4) at ($ (P) + ( #1 * \y , -#1 * \x ) + ( #1 / 2 , 0      ) $);
    \coordinate (P5) at ($ (P) + ( #1 * \y , -#1 * \x ) + ( #1 / 2 , #1 / 2 ) $);
    
    \ifnum \z = 0
      \draw[rounded corners=4, color=\a, thick] (P1) -- (P5) -- (P3);
      \draw[rounded corners=4, color=\b, thick] (P4) -- (P5) -- (P2);
    \else
      \draw[rounded corners=0.2, color=\a, thick] (P3) -- (P5) -- (P4);
      \draw[rounded corners=0.2, color=\b, thick] (P1) -- (P5) -- (P2);
    \fi
  }
}

\def\textcross{
  \begin{minipage}{13pt}
    \begin{tikzpicture}[scale=1,>=latex]
      \pipedream{0.4}{(0,0)}{$$}{$$}{0/0/1/black/black}
    \end{tikzpicture}
  \end{minipage}
}

\def\textturn{
  \begin{minipage}{13pt}
    \begin{tikzpicture}[scale=1]
      \pipedream{0.4}{(0,0)}{$$}{$$}{0/0/0/black/black}
    \end{tikzpicture}
  \end{minipage}
}

\usepackage{blkarray}

\renewcommand{\GL}{{GL}}
\newcommand{\tGL}{\tilde{\GL}}

\newcommand{\tO}{\widetilde{\zO}}

\newcommand{\mcell}{Y_\circ^{v_0}}
\newcommand{\tq}{\tilde{Q}}

\newcommand{\td}{\tilde{\bd}}

\newcommand{\bid}{\mathbf{1}}

\DeclareMathOperator{\Pipes}{\mathcal{P}}
\DeclareMathOperator{\init}{in}
\DeclareMathOperator{\rot}{rot}
\DeclareMathOperator{\row}{row}
\DeclareMathOperator{\col}{col}
\DeclareMathOperator{\minors}{minors}
\DeclareMathOperator{\RedPipes}{\mathcal{RP}}
\newcommand{\rep}{\mathtt{rep}}
\newcommand{\one}{\mathbf{1}}
\renewcommand{\AA}{\mathbb{A}}

\newcommand{\pipedreamalt}[6]{ % #1 = box diameter
  \coordinate (P) at ($ #2 - (#1, 0 ) $);

  \coordinate (Q) at ($ (P) + ( 1.5 * #1, 0.5 * #1 ) $);
  \foreach \object in {#3} {
    \node at (Q) {\object};
    \coordinate (Q) at ($ (Q) + ( #1, 0 ) $);
  }

  \coordinate (Q) at ($ (P) + ( 0.5*#1, -0.5*#1 ) $);
  \foreach \object in {#4} {
    \node at (Q) {\object};
    \coordinate (Q) at ($ (Q) + ( 0, -#1 ) $);
  }

  \foreach \x/\y/\z/\a/\b in {#5} {
    \coordinate (P1) at ($ (P) + ( #1 * \y , -#1 * \x ) + ( 0      , #1 / 2 ) $);
    \coordinate (P2) at ($ (P) + ( #1 * \y , -#1 * \x ) + ( #1     , #1 / 2 ) $);
    \coordinate (P3) at ($ (P) + ( #1 * \y , -#1 * \x ) + ( #1 / 2 , #1     ) $);
    \coordinate (P4) at ($ (P) + ( #1 * \y , -#1 * \x ) + ( #1 / 2 , 0      ) $);
    \coordinate (P5) at ($ (P) + ( #1 * \y , -#1 * \x ) + ( #1 / 2 , #1 / 2 ) $);
    
    \ifnum \z = 0
      \draw[rounded corners=4, color=\a, thick] (P1) -- (P5) -- (P3);
      \draw[rounded corners=4, color=\b, thick] (P4) -- (P5) -- (P2);
    \else \ifnum \z =1
      \draw[rounded corners=0.2, color=\a, thick] (P3) -- (P5) -- (P4);
      \draw[rounded corners=0.2, color=\b, thick] (P1) -- (P5) -- (P2);
    \else
      \draw[rounded corners=0.2, color=\b, thick] (P1) -- (P5) -- (P2);
      \draw[rounded corners=0.2, color=\a, thick] (P3) -- (P5) -- (P4);
    \fi \fi
  }
  {#6}
}

\begin{document}

\title[Three combinatorial formulas]{Three combinatorial formulas for \\type A quiver polynomials and K-polynomials}

\author{Ryan Kinser}
\author{Allen Knutson}
\author{Jenna Rajchgot}
\address{University of Iowa, Department of Mathematics, Iowa City, Iowa 52242, USA}
\email[Ryan Kinser]{ryan-kinser@uiowa.edu}
\address{Cornell University, Department of Mathematics, Ithaca, NY, USA}
\email[Allen Knutson]{allenk@math.cornell.edu}
\address{University of Saskatchewan, Department of Mathematics and Statistics, Saskatoon, SK, Canada}
\email[Jenna Rajchgot]{rajchgot@math.usask.ca}

\begin{abstract}
We provide combinatorial formulas for the multidegree and $K$-polynomial of an arbitrarily oriented type $A$ quiver locus. These formulas are generalizations of three of Knutson-Miller-Shimozono's formulas from the equioriented setting; in particular, we prove the $K$-theoretic component formula conjectured by Buch and Rim\'anyi.
\end{abstract}

%no tex
%We provide combinatorial formulas for the multidegree and K-polynomial of an arbitrarily oriented type A quiver locus. These formulas are generalizations of three of Knutson-Miller-Shimozono's formulas from the equioriented setting; in particular, we prove the K-theoretic component formula conjectured by Buch and Rimányi.



\maketitle

\setcounter{tocdepth}{1}
\tableofcontents

\section{Introduction}
\subsection{Context}
In this article, we study representation spaces of quivers.  These come with a natural action of a product of general linear groups; a quiver locus is, by definition, an orbit closure for this action (see \S \ref{sect:quiverloci}).

Quiver loci have been studied in representation theory of finite-dimensional algebras since at least the early 1980s, with particular interest in orbit closure containment and their singularities.   See the surveys \cite{Bongartzsurvey,Zwarasurvey,HZsurvey} for a detailed account.
They are also important in Lie theory, where they lie at the foundation of Lusztig's geometric realization of Ringel's work on quantum groups \cite{MR1035415,Rhallalgebras}.
From another viewpoint, quiver loci generalize some classically studied varieties such as determinantal varieties and varieties of complexes.  This is because Bongartz's work \cite{MR1402728} implies that Dynkin quiver loci are defined, at least up to radical, by minors of certain block form matrices (see the introduction and \S 3 of \cite{KR}, or \cite{MR3008913}).

The line of approach most directly related to this paper was initiated by Buch and Fulton \cite{BFchernclass}. They produced formulas for equivariant cohomology classes of quiver loci, and interpreted them as universal formulas for degeneracy loci associated to representations of the quiver in the category of vector bundles on an algebraic variety.  More formulas for equivariant cohomology and $K$-classes of quiver loci were subsequently produced in papers such as \cite{MR1932326,FRdegenlocithom,BFR,KMS}.  Rim\'anyi \cite{RimanyiCOHA} has recently shown that these classes are natural structure constants for the Cohomological Hall Algebra of Kontsevich and Soibelman associated to the quiver \cite{KontSoib}.
A more detailed account of the state of the art can be found in recent works such as \cite{MR2492443,allman,MR3239295}.

\subsection{Summary of results and methods}
%We prove three combinatorial formulas for $K$-polynomials and multidegrees of type $A$ quiver loci. 
Our formulas are first proven in the bipartite (i.e. sink-source) orientation; this comprises the bulk of the paper.  We then extend the bipartite results to all orientations using the method of \cite[\S5]{KR}: if $Q$ is a type $A$ quiver of arbitrary orientation, there is an associated bipartite type $A$ quiver $\tq$ and a bijection between orbit closures for $Q$ and a certain subset of orbit closures for $\tq$.  We show that our formulas for the quiver loci of $Q$ can be obtained by a simple substitution into the formulas for the  corresponding quiver loci of $\tq$ (Proposition \ref{prop:alphabetsub}), followed by simplification to make them intrinsic to $Q$ (making minimal reference to the associated bipartite quiver $\tq$).
 
The bipartite orientation is special because of the bipartite Zelevinsky map constructed in \cite{KR}.  It is an analogue of the map constructed by Zelevinsky for equioriented type $A$ quivers in \cite{Zgradednilp}, which was further studied in \cite{LMdegen, KMS}.  
In both cases, the maps give isomorphisms between quiver loci and intersections of Schubert varieties with an opposite Schubert cell in a
partial flag variety (Kazhdan-Lusztig varieties).
The bipartite Zelevinsky map allows us to draw on the large body of knowledge about Schubert varieties to produce our formulas.

We now describe each formula and briefly indicate the proof technique. We only discuss the $K$-theoretic formulas here, since the associated multidegree versions follow from the standard relation between $K$-polynomials and multidegrees (see \S \ref{sect:Kpolys}).
%Each is modeled on the analogous formula for equioriented type $A$ quivers in \cite{KMS}; a detailed literature comparison is found in \S \ref{sect:existinglit}. \ryan{we repeat this at the bottom of the same page}
The \emph{ratio formula} (Theorem \ref{thm:ratioformula}) expresses each $K$-polynomial as a ratio of specialized double Grothendieck polynomials. It is a relatively straightforward consequence of the existence and properties of the bipartite Zelevinsky map.
The \emph{pipe formula} (Theorem \ref{thm:pipeformula}) expresses each $K$-polynomial  as a sum over pipe dreams that have a certain shape related to the Zelevinsky map.  Given the ratio formula, its proof also follows in a rather straightforward way from work of Woo and Yong on pipe formulas for Kazhdan-Lustzig varieties \cite{wooyong}.

The \emph{component formula} (Theorem \ref{thm:component}) is our main theorem. It expresses each $K$-polynomial as a sum of products of double Grothendieck polynomials, where the sum is taken over $K$-theoretic lacing diagrams for the corresponding orbit closure.  The proof of the component formula is more involved.  To prove the bipartite version, we  degenerate quiver loci to better understood varieties in a way that preserves $K$-polynomials.  More precisely, we produce a flat family of group schemes acting fiberwise on a flat family of varieties such that over a general fiber, the orbit closures are isomorphic to quiver loci, and over the special fiber, the orbit closures are unions of products of matrix Schubert varieties.  Proving the component formula then requires combinatorial methods to determine the M\"obius function of the poset of varieties obtained by taking all possible intersections of the irreducible components of the special fiber, decomposing each of those into irreducible components, and repeating this process. An overview of this proof can be found in the expository article \cite{KinserICRA} based on the present paper.

\begin{remark}
Though we use the algebraic language of multidegrees and $K$-polynomials in this paper, our formulas also hold in other settings as there are several other interpretations of multidegrees and $K$-polynomials. These interpretations are in the languages of equivariant cohomology and $K$-theory, \cite[\S1.4]{KMS}, and the virtual, rational representations of $\GL(\bd)$ \cite[\S3]{MR2492443}.  When cited literature is written from one of these perspectives, we will use the $K$-polynomial or multidegree version without explicit mention of the conversion.
\end{remark}

\subsection{Relation to existing literature}\label{sect:existinglit}
Most of the formulas for quiver loci in the literature are for the equioriented type $A$ case. Our formulas and proof techniques are modeled on \cite{KMS}. There one already finds a $K$-theoretic ratio formula for the equioriented case; we generalize this to arbitrary orientation.  Our pipe formulas and component formulas generalize theirs in two directions: by moving from a specific orientation to arbitrary orientation, and also by moving from multidegrees to $K$-polynomials.  
Some intermediate results that helped us find these generalizations are the $K$-theoretic pipe formula \cite[Thm.~3]{MR2137947} and component formula \cite[Thm.~6.3]{MR2114821} for the equioriented case, and the multidegree component formula for arbitrary orientation \cite[Thm.~1]{MR2306279}.  The $K$-theoretic component formula that we prove is \cite[Conj.~1]{MR2306279}.

Our proof of the component formulas uses the Gr\"obner degeneration ideas of \cite{KMS} (cf. the later paper \cite{yongComponent} for a purely combinatorial approach). However, our approach is more direct by taking advantage of \cite{wooyong} and geometric streamlining suggested by the referees of this article.
The proof in \cite{KMS} involves taking a certain limit as the dimension vector grows, obtained by adding copies of the projective-injective indecomposable representation of an equioriented type $A$ quiver (the ``longest lace'').  Other orientations never have a projective-injective representation, making it unclear to us what the analogous technique would be.
In particular, this is why we do not give an analogue of the tableau formula of \cite{KMS}, given in terms of Schur functions.  Its proof depends on the \emph{stable} component formula in terms of Stanley symmetric functions.  This formula is most closely related to the conjectured positive formula for quiver loci of all Dynkin types in \cite[Conj.~1.1]{MR2492443}.

Finally, we note that our work gives a geometric interpretation of the ``double'' quiver polynomials studied in \cite{MR1932326,KMS} and other works cited above: these come from the natural action of a larger torus on lifts of equioriented orbit closures to larger representation spaces of an associated bipartite quiver.

\subsection*{Acknowledgements}
We thank Alex Yong for directing us to his work \cite{wooyong} with Alex Woo, and Christian Stump for his pipe dream \LaTeX\ macros. We also thank Anders Buch and Alex Yong for their comments on the first draft of this paper. We are especially grateful to the anonymous referees for numerous helpful suggestions during the review process.

%%%%%%%%%%%%%%%%%%%%%%
%%%%%%%%%%%%%%%%%%%%%%
%%%%%%%%%%%%%%%%%%%%%%
%%%%%%%%%%%%%%%%%%%%%%

\section{Background and preliminary results}\label{sect:background}
We work over a fixed field $K$ throughout the paper, which is often omitted from our notation. To simplify the exposition below, we will assume that $K$ is algebraically closed. However, since all schemes appearing in this paper are defined over $\ZZ$, there is no difficulty generalizing results to arbitrary $K$.

\subsection{Type $A$ quiver loci} \label{sect:quiverloci}
Fix a quiver $Q$ with vertex set $Q_0$ and arrow set $Q_1$. 
Given a \emph{dimension vector} $\bd\colon Q_0 \to \mathbb{Z}_{\geq 0}$, we have the \emph{representation space}
\begin{equation}
\mathtt{rep}_Q(\bd) := \prod_{a \in Q_1} \Mat(\bd(ta), \bd(ha)),
\end{equation}
where $\Mat(m, n)$ denotes the algebraic variety of matrices with $m$ rows, $n$ columns, and entries in $K$, and $ta$ and $ha$ denote the \emph{tail} and \emph{head} of an arrow $ta \xrightarrow{a} ha$. Each $V=(V_a)_{a \in Q_1}$ in $\mathtt{rep}_Q(\bd)$ is a \emph{representation} of $Q$; each matrix $V_a$ maps row vectors in $K^{\bd(ta)}$ to row vectors in $K^{\bd(ha)}$ by right multiplication.  
We denote the total dimension of the dimension vector $\bd$ by $d = \sum_{z \in Q_0} \bd(z)$.  
%Observe that $\mathtt{rep}_Q(\bd)$ parametrizes all representations of $Q$ with vector space $K^{\bd(z)}$ at vertex $z\in Q_0$. 
There is a \emph{base change group} $\GL(\bd) := \prod_{z \in Q_0} \GL_{\bd(z)}$,
which acts on $\mathtt{rep}_Q(\bd)$. 
Here $GL_{\bd(z)}$ denotes the general linear group of invertible $\bd(z)\times \bd(z)$ matrices with entries in $K$.
Explicitly, if $g = (g_z)_{z\in Q_0}$ is an element of $\GL(\bd)$, and $V = (V_a)_{a\in Q_1}$ is an element of $\mathtt{rep}_Q(\bd)$, then the (right) action of $GL(\bd)$ on $\mathtt{rep}_Q(\bd)$ is given by
$V\cdot g = (g_{ta}^{-1}V_a g_{ha})_{a\in Q_1}.$ 
% Two points $V, W \in \mathtt{rep}_Q(\bd)$ lie in the same orbit if and only if $V$ and $W$ are isomorphic as representations of $Q$.  
  The closure of a $\GL(\bd)$-orbit in $\mathtt{rep}_Q(\bd)$ is called a \emph{quiver locus}.  
  An introduction to the theory of quiver representations can be found in the textbooks \cite{assemetal,Schiffler:2014aa}.  

In this paper, we only work with quivers of Dynkin type $A$.  We arbitrarily designate one endpoint ``left'' and the other ``right'' so that we can speak of arrows pointing left or right.   
The \emph{bipartite orientation}, where every vertex is either a sink or source, is the fundamental orientation in type $A$: understanding the geometry of quiver loci in all other orientations essentially reduces to this one (see \cite[\S5]{KR}).  We use the following running example throughout the paper.
\begin{equation}\label{eq:bipartiteEx}
Q=\quad
\vcenter{\hbox{\begin{tikzpicture}[point/.style={shape=circle,fill=black,scale=.5pt,outer sep=3pt},>=latex]
   \node[outer sep=-2pt] (y0) at (-1,1) {${y_3}$};
   \node[outer sep=-2pt] (x1) at (0,0) {${x_3}$};
  \node[outer sep=-2pt] (y1) at (1,1) {${y_2}$};
   \node[outer sep=-2pt] (x2) at (2,0) {${x_2}$};
  \node[outer sep=-2pt] (y2) at (3,1) {${y_1}$};
   \node[outer sep=-2pt] (x3) at (4,0) {${x_1}$};
  \node[outer sep=-2pt] (y3) at (5,1) {${y_0}$};
  \path[->]
  	(y0) edge node[left] {${\zb_3}$} (x1) 
	(y1) edge node[left] {${\za_3}$} (x1)
  	(y1) edge node[left] {${\zb_2}$} (x2) 
	(y2) edge node[left] {${\za_2}$} (x2)
  	(y2) edge node[left] {${\zb_1}$} (x3)
	(y3) edge node[left] {${\za_1}$} (x3);
   \end{tikzpicture}}}
\end{equation}
We assume that all of our bipartite type $A$ quivers start and end with a source vertex; other cases follow by setting $\bd(z) =0$ for choices of $z$ on either end.  As in \eqref{eq:bipartiteEx}, we label all sink vertices by $x_j$, $1\leq j\leq n$, all source vertices by $y_i$, $0\leq i\leq n$, left-pointing arrows by $\alpha_j$, $1\leq j\leq n$, and right-pointing arrows by $\beta_i$, $1\leq i\leq n$. Subscripts increase from right to left. Since the matrices in our quiver representations act on row vectors instead of column vectors in this paper, this notation slightly differs from \cite{KR}.

Until \S \ref{sect:arbitrary}, we work with a fixed bipartite type $A$ quiver $Q$ and dimension vector $\bd$ unless explicitly stated otherwise; hence, these will be omitted from the notation. In particular, we write $\rep$ instead of $\mathtt{rep}_Q(\bd)$, and $\GL$ instead of $\GL(\bd)$.
Throughout the paper, we use the notation $\zO$ to denote a $\GL$-orbit closure in $\rep$, and $\zO^\circ$ the dense orbit of $\zO$.  We also assume that $\dim \rep > 0$ to avoid some trivial, degenerate cases.

\subsection{Permutation and matrix conventions}\label{sect:matrices}
We now fix our conventions regarding matrices and permutations. A permutation $v$ in the symmetric group $S_m$ is a bijection $v\colon\{1,\dotsc,m\}\rightarrow \{1,\dotsc, m\}$. If $v,w\in S_m$, then $vw\in S_m$ will denote the composition of functions $i\mapsto v(w(i))$.  A \emph{simple transposition} $\tau_i$ is a permutation that switches $i$ and $i+1$, while fixing everything else, and the \emph{length} $\ell(v)$ of a permutation $v$ is the minimal number of factors needed to express $v$ as a product of simple transpositions.

To a permutation $v\in S_m$, one can associate an $m\times m$ permutation matrix, $v^T$, which has a $1$ in location $(i,j)$ if $v(i) = j$, and zeroes elsewhere. 
With this convention, we have that 
$\vec{e}_i v^T = \vec{e}_{v(i)}$, where $\vec{e}_i$ is the $1\times m$ row vector with a $1$ in position $i$ and zeroes elsewhere. 
From here on, we will use $v$ to denote both a permutation and its associated matrix. To avoid confusion over what is meant by $vw$, we declare that $vw$ always refers to the function $i\mapsto v(w(i))$ as above, and never multiplication of the corresponding permutation matrices; such matrix multiplication does not appear in the paper. Define the \emph{length of a permutation matrix} to be the length of its associated permutation.

Given a matrix $M$ with entries in a ring $R$, define $M_{p\times q}$ to be the matrix consisting of the intersection of the top $p$ rows and leftmost $q$ columns of $M$.  Similarly, let $M^{p\times q}$ be the matrix consisting of the intersection of the bottom $p$ rows and rightmost $q$ columns of $M$. Let $\minors(r,M)$ denote the ideal in $R$ generated by all $r\times r$ minors of $M$.

A $k\times l$ matrix $w$ is a \emph{partial permutation matrix} if its entries are all either 1 or 0 and there is at most one $1$ in each row and column.  Define the \emph{completion} $c(w)$ of $w$ as the permutation matrix which is of minimal possible dimensions such that $w$ lies in the northwest corner (i.e. $c(w)_{k\times l} = w$), and which has minimal length among such permutation matrices. If $c(w)$ is an $m\times m$ matrix, then we sometimes consider its associated permutation function as an element of a larger symmetric group  $S_{m'}$, $m'\geq m$, by setting $c(w)(i) := i$ for each $m+1\leq i\leq m'$.

Let $X$ be a space of matrices of fixed size with entries in $K$. The \emph{universal matrix over $X$} is the matrix of the same dimensions whose $(i,j)$ entry is the coordinate function on $X$ which returns the $(i,j)$ entry of $M$ when evaluated at $M \in X$ (which may be constant).

Let $\rot(M)$ denote the $180^\circ$ rotation of a matrix $M$, and let the \emph{antidiagonal} of an $m\times m$ matrix be the set of matrix entries in positions $(i, m+1-i)$ for $1 \leq i \leq m$.

\subsection{Lacing diagrams}\label{sect:laces}
Lacing diagrams were introduced in \cite{AdF} to visualize type $A$ quiver representations, and were interpreted as sequences of partial permutation matrices in \cite{KMS}.  In this section we recall the essentials of lacing diagrams in arbitrary orientation.  We use \cite{MR2306279} as our main reference, noting that our conventions differ from theirs in order to make the connection with pipe dreams in \S \ref{sect:pipeToLace} more natural.

A \emph{lacing diagram} of dimension vector $\bd$ for a type $A$ quiver $Q$ of arbitrary orientation is a sequence of columns of dots, together with arrows connecting dots in consecutive columns. The columns are indexed by $Q_0$, appearing in the same left to right order as in the quiver $Q$. There are $\bd(z)$ dots in the column associated to vertex $z\in Q_0$. Each dot may be connected to at most one dot in the column to the left of it, and to at most one dot in the column to the right of it. Arrows between dots point in the same direction as the corresponding arrow of $Q$. We vertically align dots as in \cite[Rem.~1]{MR2306279}: if two consecutive columns of dots are connected by a left-pointing arrow, then the two columns of dots are aligned at the bottom; if two columns of dots are connected by a right-pointing arrow, then the two columns of dots are aligned at the top. A lacing diagram continuing the running example is on the left side of Figure \ref{fig:laces}.

%A \emph{lacing diagram} for a type $A$ quiver $Q$ consists of columns of dots with arrows between certain pairs of dots. The columns are indexed by the vertices of $Q$, and the number of dots in a column is the value of the dimension vector $\bd$ at the corresponding vertex. Two columns are connected by left (respectively right) pointing arrows if the vertices in $Q$ corresponding to those columns are connected by a left (respectively right) pointing arrow.
%Our convention for vertically aligning the dots is as in \cite[Rem.~1]{MR2306279}: if the columns are connected by left pointing arrows, then the dots are aligned at the bottom; if the arrows connecting the columns point right, then the dots are aligned from the top.  Then arrows can be placed between dots in the direction of the corresponding arrow of $Q$, with the restriction that each dot can be connected to at most one dot in each of its adjacent columns.  A lacing diagram continuing the running example is on the left side of Figure \ref{fig:laces}.

A lacing diagram can be interpreted as a sequence of partial permutation matrices $\bw = (w_a)_{a \in Q_1}$, as also seen in Figure \ref{fig:laces}.  For each arrow $a$, the partial permutation matrix $w_a$ has a 1 in row $i$, column $j$ exactly when there is an arrow from the $i$th dot from the top of the source column to the $j$th dot from the top of the target column. 
%So, there is a $1$ in location $(i,j)$ in $w_a$ if and only if $\vec{e}_iw_a = \vec{e}_j$ where $\vec{e}_i$ denotes the $i^{\text{th}}$ standard basis (row) vector in the source vector space $V_{ta}$, and $\vec{e}_j$ denotes the $j^{\text{th}}$ standard basis (row) vector in the target vector space $V_{ha}$.
Since the sequence of partial permutation matrices $\bw$ encodes exactly the same data as a lacing diagram, we will simply refer to $\bw$ as a lacing diagram as well.

A lacing diagram can be completed to an \emph{extended lacing diagram} as follows: if $w$ is associated to a right-pointing arrow, replace $w$ with its completion $c(w)$, and if $w$ is associated to a left-pointing arrow, replace $w$ by $\rot(c(\rot(w)))$.  An extended lacing diagram can be visualized by adding \emph{virtual} red solid squares and dashed arrows to the original lacing diagram as in the right of Figure \ref{fig:laces}.
The length $|\bw|$ of a lacing diagram $\bw$ is defined as the sum of the lengths of the permutations in the extended lacing diagram, or equivalently, the total number of crossings of laces in the extended lacing diagram.

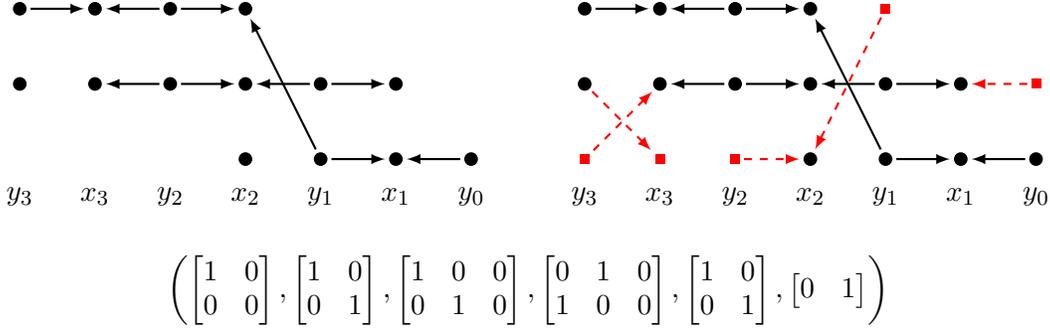
\begin{figure}
\begin{tikzpicture}[point/.style={shape=circle,fill=black,scale=.5pt,outer sep=3pt},>=latex] % lace1
\node[point] (1b) at (0,1) {};
\node[point] (1c) at (0,2) {};
\node[point] (2b) at (1,1) {};
\node[point] (2c) at (1,2) {};
\node[point] (3b) at (2,1) {};
\node[point] (3c) at (2,2) {};
\node[point] (4a) at (3,0) {};
\node[point] (4b) at (3,1) {};
\node[point] (4c) at (3,2) {};
\node[point] (5a) at (4,0) {};
\node[point] (5b) at (4,1) {};
\node[point] (6a) at (5,0) {};
\node[point] (6b) at (5,1) {};
\node[point] (7a) at (6,0) {};
  
\node at (0,-0.5) {$y_3$}; 
\node at (1,-0.5) {$x_3$}; 
\node at (2,-0.5) {$y_2$}; 
\node at (3,-0.5) {$x_2$}; 
\node at (4,-0.5) {$y_1$}; 
\node at (5,-0.5) {$x_1$}; 
\node at (6,-0.5) {$y_0$}; 
  
\path[->,thick]
   (1c) edge (2c)
   (3b) edge (2b)
   (3c) edge (2c)
   (3b) edge (4b)
   (3c) edge (4c)
   (5a) edge (4c)
   (5b) edge (4b)
   (5a) edge (6a)
   (5b) edge (6b)
   (7a) edge (6a);
  
  \end{tikzpicture}
\qquad
\begin{tikzpicture}[point/.style={shape=circle,fill=black,scale=.5pt,outer sep=3pt},epoint/.style={shape=rectangle,fill=red,scale=.5pt,outer sep=3pt},>=latex] %lace 1 completion
\node[epoint] (1a) at (0,0) {};
\node[point] (1b) at (0,1) {};
\node[point] (1c) at (0,2) {};
\node[epoint] (2a) at (1,0) {};
\node[point] (2b) at (1,1) {};
\node[point] (2c) at (1,2) {};
\node[epoint] (3a) at (2,0) {};
\node[point] (3b) at (2,1) {};
\node[point] (3c) at (2,2) {};
\node[point] (4a) at (3,0) {};
\node[point] (4b) at (3,1) {};
\node[point] (4c) at (3,2) {};
\node[point] (5a) at (4,0) {};
\node[point] (5b) at (4,1) {};
\node[epoint] (5c) at (4,2) {};
\node[point] (6a) at (5,0) {};
\node[point] (6b) at (5,1) {};
\node[point] (7a) at (6,0) {};
\node[epoint] (7b) at (6,1) {};
  
\node at (0,-0.5) {$y_3$}; 
\node at (1,-0.5) {$x_3$}; 
\node at (2,-0.5) {$y_2$}; 
\node at (3,-0.5) {$x_2$}; 
\node at (4,-0.5) {$y_1$}; 
\node at (5,-0.5) {$x_1$}; 
\node at (6,-0.5) {$y_0$}; 

\path[->,thick]
   (1b) edge[red,dashed] (2a)
   (1c) edge (2c)
   (1a) edge[red,dashed] (2b)
   (3b) edge (2b)
   (3c) edge (2c)
   (3a) edge[red,dashed] (4a)
   (3b) edge (4b)
   (3c) edge (4c)
   (5a) edge (4c)
   (5b) edge (4b)
   (5c) edge[red,dashed] (4a)
   (5a) edge (6a)
   (5b) edge (6b)
   (7b) edge[red,dashed] (6b)
   (7a) edge (6a);
  \end{tikzpicture}\\
  \vspace{.5cm}
$
\left(
\begin{bmatrix}
1 & 0\\
0 & 0
\end{bmatrix},
\begin{bmatrix}
1 & 0\\
0 & 1
\end{bmatrix},
\begin{bmatrix}
1 & 0 & 0\\
0 & 1 & 0\\
\end{bmatrix},
\begin{bmatrix}
0 & 1 & 0\\
1 & 0 & 0
\end{bmatrix},
\begin{bmatrix}
1 & 0\\
0 & 1
\end{bmatrix},
\begin{bmatrix}
0 & 1 
\end{bmatrix}
\right)$
    \caption{A minimal lacing diagram, its matrix form, and its extended diagram.}
\label{fig:laces}
\end{figure}

A lacing diagram $\bw$ is naturally an element of $\mathtt{rep}_Q(\bd)$ by assigning to each  $a \in Q_1$ the matrix $w_a$.
%The visualization as a lacing diagram allows one to easily see the indecomposable direct summands of this representation (the laces).  
Two lacing diagrams lie in the same $\GL(\bd)$-orbit if and only if they have the same number of laces between each pair of columns.  A \emph{minimal lacing diagram} is one whose length is minimal among those in its $\GL(\bd)$-orbit, and we denote by $W(\zO)$ the set of all minimal lacing diagrams in the orbit $\zO^\circ$.  For $\bw \in W(\zO)$ we have by \cite[Cor.~2]{MR2306279} that
\begin{equation}\label{eq:codimcrossings}
\codim \zO=|\bw|,
\end{equation}
where the codimension of $\zO$ is taken in $\rep_Q(\bd)$ throughout the paper.

By \cite[Prop.~1]{MR2306279}, two minimal lacing diagrams lie in the same orbit if and only if their extended diagrams are related by a series of transformations of the following form, 
\begin{equation}\label{eq:lacemove}
\vcenter{\hbox{  
\begin{tikzpicture}[point/.style={shape=circle,fill=black,scale=.5pt,outer sep=3pt},>=latex]
\node[point] (1a) at (0,0) {};
\node[point] (1b) at (0,1) {};
\node[point] (2a) at (1,0) {};
\node[point] (2b) at (1,1) {};
\node[point] (3a) at (2,0) {};
\node[point] (3b) at (2,1) {};

\path[-,thick]
   (1b) edge (2a)
   (1a) edge (2b)
   (2a) edge (3a)
   (2b) edge (3b);
\end{tikzpicture}}}
\longleftrightarrow
\vcenter{\hbox{  
\begin{tikzpicture}[point/.style={shape=circle,fill=black,scale=.5pt,outer sep=3pt},>=latex]
\node[point] (1a) at (0,0) {};
\node[point] (1b) at (0,1) {};
\node[point] (2a) at (1,0) {};
\node[point] (2b) at (1,1) {};
\node[point] (3a) at (2,0) {};
\node[point] (3b) at (2,1) {};

\path[-,thick]
   (1a) edge (2a)
   (1b) edge (2b)
   (2a) edge (3b)
   (2b) edge (3a);
\end{tikzpicture}}}
\end{equation}
where both middle dots and at least one dot in each of the outer columns is not virtual (otherwise one of the permutations would not be a minimal length extension of the original partial permutation).
There are also $K$-theoretic transformations of extended lacing diagrams
\begin{equation}
\vcenter{\hbox{  
\begin{tikzpicture}[point/.style={shape=circle,fill=black,scale=.5pt,outer sep=3pt},>=latex]
\node[point] (1a) at (0,0) {};
\node[point] (1b) at (0,1) {};
\node[point] (2a) at (1,0) {};
\node[point] (2b) at (1,1) {};
\node[point] (3a) at (2,0) {};
\node[point] (3b) at (2,1) {};

\path[-,thick]
   (1b) edge (2a)
   (1a) edge (2b)
   (2a) edge (3a)
   (2b) edge (3b);
\end{tikzpicture}}}
\longleftrightarrow
\vcenter{\hbox{  
\begin{tikzpicture}[point/.style={shape=circle,fill=black,scale=.5pt,outer sep=3pt},>=latex]
\node[point] (1a) at (0,0) {};
\node[point] (1b) at (0,1) {};
\node[point] (2a) at (1,0) {};
\node[point] (2b) at (1,1) {};
\node[point] (3a) at (2,0) {};
\node[point] (3b) at (2,1) {};

\path[-,thick]
   (1b) edge (2a)
   (1a) edge (2b)
   (2a) edge (3b)
   (2b) edge (3a);
\end{tikzpicture}}}
\longleftrightarrow
\vcenter{\hbox{  
\begin{tikzpicture}[point/.style={shape=circle,fill=black,scale=.5pt,outer sep=3pt},>=latex]
\node[point] (1a) at (0,0) {};
\node[point] (1b) at (0,1) {};
\node[point] (2a) at (1,0) {};
\node[point] (2b) at (1,1) {};
\node[point] (3a) at (2,0) {};
\node[point] (3b) at (2,1) {};

\path[-,thick]
   (1a) edge (2a)
   (1b) edge (2b)
   (2b) edge (3a)
   (2a) edge (3b);
\end{tikzpicture}}}
\end{equation}
with the same condition on the dots, and in addition the two middle dots should be consecutive in their column.  A lacing diagram for an orbit is called \emph{$K$-theoretic} if its extended lacing diagram can be obtained by a sequence of $K$-theoretic transformations from the extended diagram of a minimal lacing diagram for the orbit.  We let $KW(\zO)$ denote the set of $K$-theoretic lacing diagrams for the orbit $\zO^\circ.$

\subsection{Matrix Schubert varieties}\label{sect:matSchubert}
Matrix Schubert varieties \cite{Fulton} are subvarieties of $X=\text{Mat}(k,l)$ obtained by imposing rank conditions on certain submatrices. 
Let $w$ be a partial permutation matrix in $X$.
The \emph{northwest matrix Schubert variety} $X_w$ is the subvariety $X_w:= \{M\in X \mid \forall (p,q)\ \text{rank }M_{p\times q}\leq \text{rank }w_{p\times q}\}$, and the \emph{southeast matrix Schubert variety} $X^w$ is the subvariety $X^w:=\{M\in X \mid \forall (p,q)\ \text{rank }M^{p\times q}\leq \text{rank }w^{p\times q}\}$.
Let $Z = (z_{ij})$ be the universal matrix over $X$. The prime defining ideals of $X_w$ and $X^w$ are $J_w:=\sum_{(p,q)} \minors(1+\rank w_{p\times q}, Z_{p\times q})$ and $J^w:=\sum_{(p,q)} \minors (1+\rank w^{p\times q}, Z^{p\times q})$, respectively.

Each variety $X_w$ and $X^w$ is an orbit closure for the action of a product of Borel subgroups: let $B_+$ and $B_-$ denote Borel subgroups of invertible upper and lower triangular matrices respectively. Then we have that $X_w = \overline{B_- w B_+} \text{ and } X^w = \overline{B_+ w B_-}$, 
where the closures are taken inside of the space $X$, the Borel subgroups to the left of $w$ are inside of $\GL_k$, and the Borel subgroups to the right of $w$ are inside of $\GL_l$. 

\subsection{The bipartite Zelevinsky map}\label{sect:Zmap}
Here we recall the relationship established in \cite{KR} between quiver loci of bipartite type $A$ quivers and Schubert varieties. Retaining the notation from \S \ref{sect:quiverloci}, let $d_x = \sum_{j=1}^n \bd(x_j)$ and $d_y = \sum_{i=0}^n \bd(y_i)$, so that $d=d_x + d_y$. Writing $\bid_r$ for an $r\times r$ identity matrix, let $\mcell$ denote the space of $d\times d$ matrices of the form
\begin{equation}\label{eq:cell}
\begin{bmatrix}
* & \bid_{d_y}\\
\bid_{d_x} & 0 \\
\end{bmatrix}
\end{equation}
where entries in the block labeled by the asterisk are arbitrary elements of the base field $K$. 

Let $Z=(z_{ij})$ be the universal matrix over $\mcell$.
Given a $d\times d$ permutation matrix $v$, define the ideal $I_v:=\sum_{(p,q)}\minors(1+\rank v_{p\times q}, Z_{p\times q})$.
The affine scheme defined by the prime ideal $I_v$, which we denote by $Y_v\subseteq \mcell$, is called a \emph{Kazhdan-Lusztig variety}  and is isomorphic to the intersection of an opposite Schubert cell and a Schubert variety (see \cite{MR2422304} for details). 
To be precise, let $G:=\GL_d$ and let $P$ be the parabolic subgroup of block lower triangular matrices where the diagonals have block sizes 
\begin{equation}\label{eq:rowblocks}
\bd(y_0),\dotsc, \bd(y_n),\ \bd(x_n),\ \bd(x_{n-1}),\dotsc, \bd(x_1),
\end{equation}
listed from northwest to southeast. Then $\mcell$ is isomorphic to the opposite Schubert cell $P\backslash Pv_0B_{-}$, where
$v_0$ is the permutation matrix obtained from \eqref{eq:cell} by setting all entries in the block labeled by the asterisk to 0.
The affine scheme $Y_v$ is isomorphic to the intersection of the Schubert variety $\overline{P\backslash PvB_{+}}$ with the opposite cell $P\backslash Pv_0B_{-}$ in $P \backslash G$.

In \cite{KR} a closed immersion $\zeta\colon \mathtt{rep} \to \mcell$ was defined, called the \emph{bipartite Zelevinsky map}. The image of a representation $(V_a)_{a\in Q_1}$ under this map is shown in  Figure \ref{fig:bigzimage} in the case of our running example. The definition of the bipartite Zelevinsky map for an arbitrary bipartite type $A$ quiver is extrapolated in the obvious way.

\begin{theorem}\label{thm:Zmap}\cite[Thm~4.12]{KR}
The bipartite Zelevinsky map $\zeta$ restricts to an isomorphism from each quiver locus $\zO \subseteq \mathtt{rep}$ to a Kazhdan-Lusztig variety $Y_{v({\zO})} \subseteq \mcell$.
\end{theorem}

The permutation $v(\zO)$ in the statement of Theorem \ref{thm:Zmap} is known as the \emph{Zelevinsky permutation} of $\zO$, and will be discussed in the next subsection.  We define $v_*$ to be the Zelevinsky permutation of the entire space $\mathtt{rep}$, so that $Y_{v_*}$ is exactly the image of $\zeta$.

A dimension vector $\bd$ determines a subdivision of any $d\times d$ matrix into blocks as follows (see Figure \ref{fig:bigzimage} for an example).  Row blocks are indexed from top to bottom by $y_0,\dots,y_n, x_n, x_{n-1}, \dotsc, x_1$, with sizes determined by $\bd$ (i.e., as in \eqref{eq:rowblocks} above).  Column blocks are indexed from left to right by $x_n, x_{n-1}, \dotsc, x_1, y_0, y_1, \dotsc, y_n$, with sizes similarly determined by $\bd$: 
\begin{equation}\label{eq:colblocks}
\bd(x_n),\dotsc, \bd(x_1),\ \bd(y_0),\ \bd(y_1),\dotsc, \bd(y_n).
\end{equation}

\begin{definition}\label{def:snake}
%Let $M = (m_{ij})$ be a $d\times d$ matrix with block structure as above. 
For any $d\times d$ matrix with block structure as above, the \emph{block corresponding to $\beta_k$ (resp., $\alpha_k$)} is the intersection of the block row indexed by $y_k$ (resp. $y_{k-1}$) with the block column indexed by $x_k$ (resp., $x_k$). The \emph{snake region} of the matrix is the set of locations $(i,j)$ which are in a block corresponding to some $\alpha_k$ or $\beta_k$. (For example, in Figure \ref{fig:bigzimage}, the snake region is the set of locations occupied by $V_{\alpha_1}$ through $V_{\beta_3}$.)
\end{definition}

\begin{figure}
\[\zeta(V) =
\vcenter{\hbox{
\begin{tikzpicture}[every node/.style={minimum width=1em}]
\matrix (m0) [matrix of math nodes,left delimiter  = {[},%
             right delimiter = {]}, nodes in empty cells] at (0,0)
{
0& 0 & V_{\za_1}& \bid_{\bd(y_0)} &  & \\
0 &  V_{\za_2} & V_{\zb_1}& & \bid_{\bd(y_1)} & \\
V_{\za_3} & V_{\zb_2} & 0 & & & \bid_{\bd(y_2)}\\
V_{\zb_3} & 0 & 0 & & & & \bid_{\bd(y_3)} \\
\bid_{\bd(x_3)} &  & \\
& \bid_{\bd(x_{2})} &  & \\
& &\bid_{\bd(x_{1})} \\
};
%top labels
\node  at (0,2.5) {$\bd(y_0)$};
\node  at (1.2,2.5) {$\bd(y_1)$};
\node  at (2.4,2.5) {$\bd(y_2)$};
\node  at (3.6,2.5) {$\bd(y_3)$};
\node  at (-1.2,2.5) {$\bd(x_1)$};
\node  at (-2.4,2.5) {$\bd(x_2)$};
\node  at (-3.6,2.5) {$\bd(x_3)$};
%side labels
\node  at (-5.2,2) {$\bd(y_0)$};
\node  at (-5.2,1.4) {$\bd(y_1)$};
\node  at (-5.2,0.8) {$\bd(y_2)$};
\node  at (-5.2,0.2) {$\bd(y_3)$};
\node  at (-5.2,-0.6) {$\bd(x_3)$};
\node  at (-5.2,-1.3) {$\bd(x_2)$};
\node  at (-5.2,-2) {$\bd(x_1)$};
\node[scale=3] (zero) at (m0-6-2 -| m0-2-5.south east)  {$0$};
\draw[thick] (m0-5-1.north west) -- (m0-5-1.north west -| m0-4-7.south east);%across
\draw[thick] (m0-1-4.north west) -- (m0-1-4.north west |- m0-7-3.south east);%down
\end{tikzpicture}}}
\]
\caption{Image of the Zelevinsky map in $\mcell$, labeled by block sizes}\label{fig:bigzimage}
\end{figure}

%%%%%%%%%%%%%%%%%%%%%
%%%%%%%%%%%%%%%%%%%%%

\subsection{Zelevinsky permutations}\label{sect:Zperm}

In this section, we recall the \emph{bipartite Zelevinsky permutation} $v(\zO)$ from \cite[\S4]{KR}. The definition we give here is more direct than the one in \cite[\S4]{KR} to make this paper more self-contained. 
 We give a characterization of $v(\zO)$ in terms of lacing diagrams in Theorem \ref{thm:lacesZperm} below, which is an essential combinatorial ingredient in our proof of the component formula.  In our definition of Zelevinsky permutation below, we assume that matrices have the block structure described immediately preceding Definition \ref{def:snake}, and we let $M^{\text{block}}_{p\times q}$ denote the matrix which consists of the intersection of the top $p$ row blocks and leftmost $q$ column blocks of $M$.

Let $\zO \subseteq \rep(\bd)$ be a quiver locus.
The \emph{Zelevinsky permutation $v(\zO)$} is the unique permutation in $S_d$ which satisfies the following conditions, where $V$ is any element of $\zO^\circ$ \cite[Prop.~4.8]{KR}:
\begin{enumerate}
\item  the number of 1s in block $(p,q)$ of the $d\times d$ permutation matrix $v(\zO)$ is equal to 
\begin{equation}\label{eq:nwbr}
\rank \zeta(V)^{\text{block}}_{p \times q} + \rank \zeta(V)^{\text{block}}_{p-1 \times q-1} - \rank \zeta(V)^{\text{block}}_{p-1 \times q} - \rank \zeta(V)^{\text{block}}_{p \times q-1}
\end{equation}
where a summand is taken to be 0 if $p$ or $q$ is outside of the range $\{1, \dotsc, 2n+1\}$;
\item the 1s in $v(\zO)$ are arranged from northwest to southeast across each block row;
\item the 1s in $v(\zO)$ are arranged from northwest to southeast down each block column.
\end{enumerate}

An important property of the Zelevinsky permutation is that 
\begin{equation}\label{eq:Zpermlength}
\codim \zO= \ell(v(\Omega)) - \ell(v_*)
\end{equation}
%where the codimension of $\zO$ is taken in $\rep$ throughout the paper 
(see, for example, the proof of \cite[Cor.~4.13]{KR}).
The next theorem characterizes the Zelevinsky permutation of an orbit directly from any lacing diagram in the orbit, analogous to \cite[Prop.~1.6]{KMS}. See Figure \ref{fig:Zperm} for an example.

\begin{figure}
$  \begin{blockarray}{ccc|ccc|cc|c|cc|cc|cc}
            & x_3 && x_2 &&& x_1 && y_0 & y_1 && y_2 && y_3& \\
      \begin{block}{c[ccccccc|ccccccc]}
y_0 &  &  &  &  &  & 1 & 0 &  &\\
\cline{1-1}
y_1 &  &  & 1 & 0 & 0 &  &  &\\
       &  &  & 0 & 1 & 0 &  &  &\\
\cline{1-1}
y_2 & 1 & 0 &  &  &  &  &  & \\
       & 0 & 1 &  &  &  &  &  & \\
\cline{1-1}
y_3 &  &  &  &  &  &  &  & 1 &  &  &  &  & 0 & 0\\
       &  &  &  &  &  &  &  & 0 &  &  &  &  & 1 & 0\\
\cline{1-15}
x_3 &  &  &  &  &  & 0 & 1 &  &  &  &  &  & 0 & 0\\
       &  &  &  &  &  & 0 & 0 &  &  &  &  &  & 0 & 1\\
\cline{1-1}
x_2 &  &  & 0 & 0 & 1 &  &  &  &  &  & 0 & 0 &\\
       &  &  & 0 & 0 & 0 &  &  &  &  &  & 1 & 0 &\\
       &  &  & 0 & 0 & 0 &  &  &  &  &  & 0 & 1 &\\
\cline{1-1}
x_1 &  &  &  &  &  &  &  &  & 1 & 0 &\\
      &  &  &  &  &  &  &  &  & 0 & 1 &\\
      \end{block}
    \end{blockarray}$
    \caption{The Zelevinsky permutation matrix for the quiver locus associated to the lacing diagram in Figure \ref{fig:laces}. The empty blocks in the matrix contain only zeros.}
\label{fig:Zperm}
\end{figure}

\begin{theorem}\label{thm:lacesZperm}
Let $\bw \in \zO^\circ$ be a lacing diagram.  The Zelevinsky permutation $v(\zO)$ is the unique permutation matrix satisfying the following:
\begin{enumerate}[(Z1)]
\item for $z_i,z_j\in Q_0$ with $z_i$ to the left of $z_j$ in $Q$ (including the case $z_i=z_j$), 
 the number of 1s in block $(z_i, z_j)$ is the number of laces with left endpoint $z_i$ and right endpoint $z_j$; %whenever  in $Q$, including the case $z_i=z_j$ (visually, $z_i - \bullet - \cdots - \bullet - z_j$);
\item for $z_i,z_j\in Q_0$ with $z_i$ exactly one vertex to the right of $z_j$ in $Q$, the number of 1s in block $(z_i, z_j)$ is the number of arrows between the columns of $\bw$ indexed by $z_i$ and $z_j$;
% whenever $z_i$ is one vertex to the right of $z_j$ (visually, $z_j - z_i$);
\item the 1s  are arranged from northwest to southeast in each block row and column.
\end{enumerate}
\end{theorem}
\begin{proof}
%Let $M$ be the matrix described by (Z1)--(Z3).  We start by showing this is a permutation matrix. 
First consider the case where $\bw$ consists of a single lace, having left endpoint $z_i$ and right endpoint $z_j$, noting that in this situation all blocks of the matrix under consideration have size one or zero.   Then (Z1) puts a 1 only in row $z_i$, and every row except $z_i$ gets single 1 from (Z2) because the vertices which are not the left endpoint of $\bw$ are the right endpoint of some arrow of $\bw$.  
Therefore, a matrix satisfying (Z1) and (Z2) has exactly one entry equal to 1 in each row.
A similar argument shows there is exactly one entry equal to 1 in each column, and thus there exists a unique permutation matrix satisfying (Z1) and (Z2) when $\bw$ has a single lace.  
In the case of many laces, the number of 1s in each block of a matrix satisfying (Z1), (Z2) is clearly additive with respect to direct sum of representations (disjoint union of laces).  Given such a matrix, filling in the top block row with 1s and 0s subject to (Z3), then proceeding down block rows filling each west to east using (Z3), we see that there exists a unique permutation matrix $M$ associated to any lacing diagram by (Z1), (Z2), (Z3).  

To prove $M=v(\zO)$, it is enough to show that both have the same number of 1s in each block, since (Z3) then determines the arrangement of 1s as described above.
But the number of 1s in each block of $v(\zO)$ is also additive with respect to direct sum of representations: firstly, this quanity is a linear function of the values $\rank \zeta(\bw)^{\text{block}}_{p \times q}$ by \eqref{eq:nwbr}. 
Then these values are in turn additive with respect to direct sum of representations, because they are the dimensions of certain $\Hom$ spaces in the category of representations of $Q$ (up to constants determined by $\bd$; see \cite[Lem.~A.3]{KR} and the proof of \cite[Prop.~3.1]{KR}).  Therefore, we have reduced the theorem to showing that $M$ and $v(\zO)$ have the same number of 1s in each block in the case $\bw$ is a single lace.
Since both $M$ and $v(\zO)$ are permutation matrices, it is enough to show that whenever $M$ has a 1 in a certain block, $v(\zO)$ also does (then the rest of that row and column are zero in both $M$ and $v(\zO)$.

The number of 1s in a block of $v(\zO)$ is determined by \eqref{eq:nwbr}.  Since the blocks are of size one, \eqref{eq:nwbr} equals 1 if and only if
\[
\rank \zeta(\bw)_{p \times q} = \rank \zeta(\bw)_{p-1 \times q-1} +1 = \rank \zeta(\bw)_{p-1 \times q} +1 = \rank \zeta(\bw)_{p \times q-1}+ 1.
\]
We may simply verify that these equalities hold whenever $M$ has a 1 in position $(p,q)$. 
 Now this is easy to do by inspection, with three cases depending on whether $(p,q)$ corresponds to row and column labels $(y_i, x_{i+1})$ (the northwest antidiagonal of 1s contributed by (Z2)), $(x_i, y_i)$ (the southeast antidiagonal of 1s contributed by (Z2)), or the unique 1 contributed by (Z1) (whose row label is the left endpoint of $\bw$ and column label is the right endpoint of $\bw$).
An example of both $M$ and $\zeta(\bw)$ is shown below for $Q$ as in \eqref{eq:bipartiteEx}, with $\bw$ the lace from $y_3$ to $y_0$.
\begin{equation}\label{eq:pipelace}
M = 
  \begin{blockarray}{cccc|cccc}
            & x_3 & x_2 & x_1 & y_0 & y_1 & y_2 & y_3 \\
      \begin{block}{c[ccc|cccc]}
y_0 & 0 & 0 & 1 & 0 & 0 & 0 & 0\\
y_1 & 0 & 1 & 0 & 0 & 0 & 0 & 0\\
y_2 & 1 & 0 & 0 & 0 & 0 & 0 & 0\\
y_3 & 0 & 0 & 0 & 1 & 0 & 0 & 0\\
\cline{1-8}
x_3 & 0 & 0 & 0 & 0 & 0 & 0 & 1\\
x_2 & 0 & 0 & 0 & 0 & 0 & 1 & 0\\
x_1 & 0 & 0 & 0 & 0 & 1 & 0 & 0\\
      \end{block}
    \end{blockarray}
\qquad 
\zeta(\bw) = 
  \begin{blockarray}{cccc|cccc}
            & x_3 & x_2 & x_1 & y_0 & y_1 & y_2 & y_3 \\
      \begin{block}{c[ccc|cccc]}
y_0 & 0 & 0 & 1 & 1 & 0 & 0 & 0\\
y_1 & 0 & 1 & 1 & 0 & 1 & 0 & 0\\
y_2 & 1 & 1 & 0 & 0 & 0 & 1 & 0\\
y_3 & 1 & 0 & 0 & 0 & 0 & 0 & 1\\
\cline{1-8}
x_3 & 1 & 0 & 0 & 0 & 0 & 0 & 0\\
x_2 & 0 & 1 & 0 & 0 & 0 & 0 & 0\\
x_1 & 0 & 0 & 1 & 0 & 0 & 0 & 0\\
      \end{block}
    \end{blockarray}
\end{equation}
\end{proof}

%%%%%%%%%%%%%%%%%%%%%%%%%%%%%%%%%%%%%%%%%%%%%%%%

\subsection{Multigradings}\label{sect:multigrading}
We briefly recall how a torus action on a variety induces a multigrading of its coordinate ring.
Suppose we have an algebraic right action of an algebraic torus $T = (K^\times)^d$ on an affine $K$-variety $X=\Spec R$, with the action of $t \in T$ on $x \in X$ written as $x \cdot t$.  This induces a right action of $T$ on the $K$-algebra $R$, given by $(f \cdot t)(x) = f(x \cdot t^{-1})$ for $f \in R$. 
For each $\mathbf{e}= (e_1, \dotsc, e_d) \in \ZZ^d$, we have a \emph{weight space} $R_\mathbf{e} = \setst{f \in R}{f \cdot (t_1, \dotsc, t_d) =t_1^{e_1} \cdots t_d^{e_d} f}$, and a decomposition of $R$ as a $K$-vector space $R \simeq \bigoplus_{\mathbf{e} \in \ZZ^d} R_\mathbf{e}$, which is a multigrading of $R$ by $\ZZ^d$. We say that elements of $R_{\mathbf{e}}$ have $\mathbb{Z}^d$-degree $\mathbf{e}$, or simply degree $\mathbf{e}$ when there is no chance of confusion.
A closed subvariety $\Spec(R/I) \subseteq X$ is $T$-stable exactly when $I$ is homogeneous with respect to this multigrading.

Let $Q$ be any quiver and let $\bd$ be a dimension vector for $Q$.  Let $T$ be the maximal torus of $\GL(\bd)$ consisting of matrices which are diagonal in each factor.  The restriction of the $\GL(\bd)$-action on $\mathtt{rep}_Q(\bd)$ to $T$ induces a multigrading on $K[\mathtt{rep}_Q(\bd)]$ as in the previous paragraph, which makes the ideals of orbit closures homogeneous.  Now we introduce notation to explicitly describe this for bipartite type $A$ quivers, deferring the case of arbitrary orientation until \S \ref{sect:arbitrary}.

To explicitly describe the multigrading of $K[\mathtt{rep}]$ in the bipartite orientation, we associate an (ordered) alphabet $\bs^j=s^j_1, s^j_2, \dotsc, s^j_{\bd(x_j)}$ to the vertex $x_j$, and an alphabet $\bt^i=t^i_1, t^i_2, \dotsc, t^i_{\bd(y_i)}$ to the vertex $y_i$.
Let $\bt$ be the sequence $\bt^0, \bt^1, \dotsc, \bt^n$ and $\bs$ be the sequence $\bs^n, \dotsc, \bs^2, \bs^1$ (the order in which the alphabets are concatenated is indicated, with each individual alphabet still in its standard order).
Identify $\ZZ^{d_x}$ with the free abelian group on $\bs$ and $\ZZ^{d_y}$ with the free abelian group on $\bt$, and $\ZZ^d = \ZZ^{d_x} \oplus \ZZ^{d_y}$, where $d_x, d_y, d$ are as in \S \ref{sect:Zmap}.  
Then in the induced multigrading of the coordinate ring $K[\mathtt{rep}]$
by $\ZZ^d$, the coordinate function which picks out the $(i,j)$
entry of $V_{\za_k}$ has degree $t^{k-1}_i - s^{k}_{j}$, and the
coordinate function picking out the $(i,j)$ entry of $V_{\zb_k}$ has
degree $t^{k}_i - s^{k}_{j}$.

Our multigrading of $K[\mcell]$ by $\ZZ^d$ is induced by right multiplication of the torus of diagonal matrices $T \subseteq \GL_d$ on $P\backslash Pv_0B_{-}$.  To explicitly describe the multigrading, recall the block structure for matrices in $\mcell$ introduced in \S \ref{sect:Zmap}.
The coordinate function on $\mcell$ picking out the entry in the $(i,j)$-location of the block with block row index $y_k$ and block column index $x_l$ has degree $t^k_i - s^l_j$. 
Figure \ref{fig:multigrading} illustrates a labeling of row and column blocks of $\mcell$ by the alphabets $\bt, \bs$: that is, the degree of a coordinate function picking out a matrix entry is its row label minus its column label. 

\begin{figure}
\[
\vcenter{\hbox{
\begin{tikzpicture}[every node/.style={minimum width=1em}]
\matrix (m0) [matrix of math nodes,left delimiter  = {[},%
             right delimiter = {]}, nodes in empty cells] at (0,0)
{
* & * & * & \bid_{\bd(y_0)} &  & \\
* &  * & *& & \bid_{\bd(y_1)} & \\
* & * & * & & & \bid_{\bd(y_2)}\\
* & * & * & & & & \bid_{\bd(y_3)} \\
\bid_{\bd(x_3)} &  & \\
& \bid_{\bd(x_{2})} &  & \\
& &\bid_{\bd(x_{1})} \\
};
%top labels
\node  at (0,2.5) {$\bt^0$};
\node  at (1.2,2.5) {$\bt^1$};
\node  at (2.4,2.5) {$\bt^2$};
\node  at (3.6,2.5) {$\bt^3$};
\node  at (-1.2,2.5) {$\bs^1$};
\node  at (-2.4,2.5) {$\bs^2$};
\node  at (-3.6,2.5) {$\bs^3$};
%side labels
\node  at (-5,2) {$\bt^0$};
\node  at (-5,1.4) {$\bt^1$};
\node  at (-5,0.8) {$\bt^2$};
\node  at (-5,0.2) {$\bt^3$};
\node  at (-5,-0.6) {$\bs^3$};
\node  at (-5,-1.3) {$\bs^2$};
\node  at (-5,-2) {$\bs^1$};
\node[scale=3] (zero) at (m0-6-2 -| m0-2-5.south east)  {$0$};
\draw[thick] (m0-5-1.north west) -- (m0-5-1.north west -| m0-4-7.south east);%across
\draw[thick] (m0-1-4.north west) -- (m0-1-4.north west |- m0-7-3.south east);%down
\end{tikzpicture}}}
\]
\caption{Alphabets assigned to each block in the multigrading of $K[\mcell]$}\label{fig:multigrading}
\end{figure}

We use the following easy-to-check result implicitly throughout the paper. It lets us use known formulas for $K$-polynomials and multidegrees of Kazhdan-Lusztig varieties to deduce results on $K$-polynomials and multidegrees of bipartite type $A$ quiver loci.
\begin{lemma}\label{lem:GLeqvt}
Let $\zeta^\sharp: K[\mcell]\rightarrow K[\rep]$ be the map on coordinate rings induced by the bipartite Zelevinsky map. If $z_{ij}\in K[\mcell]$ is the coordinate function on $\mcell$ which picks out an entry in the snake region of Definition \ref{def:snake}, then $z_{ij}$ and $\zeta^{\sharp}(z_{ij})$ have the same degree, with respect to the multigradings of $K[\rep]$ and $K[\mcell]$ defined above.
\end{lemma}
Let $X = \text{Mat}(k,l)$ with its natural right action of $\GL_k \times \GL_l$. A northwest matrix Schubert variety $X_w \subseteq X$ has an induced action of the diagonal torus $T=(K^\times)^k \times (K^\times)^l \subset \GL_k \times \GL_l$.  For an explicit description of the multigrading of $K[X_w]$, assign letters $\ba = a_1, \dotsc, a_k$ to the rows, from top to bottom, and letters $\bb=b_1, \dotsc, b_l$ to the columns, from left to right.  The coordinate function picking out the $(i,j)$ entry of a matrix in $X_w$ has degree $a_i - b_j$.  For southeast matrix Schubert varieties, the same torus acts, and thus the coordinate function picking out the $(i,j)$ entry of a matrix in $X^w$ also has degree $a_i - b_j$. 
We will use the following relationship between northwest and southeast matrix Schubert varieties. 
Here, and throughout the rest of the paper, the symbol $\tilde{\phantom{\bx}}$ over a finite alphabet denotes that the order of the alphabet is reversed, so that for $\bc = c_1,\dots, c_n$, we have $\tilde{\bc}= \tilde{c}_1, \dotsc, \tilde{c}_n$ with $\tilde{c}_i = c_{n-i+1}$.

\begin{lemma}\label{lem:SEmatrixschubert}
The operation $\rot$ of \S \ref{sect:matrices} induces an isomorphism of varieties $X_{\rot(w)}\to X^{w}$ and thus an isomorphism of rings $K[X^w] \to K[X_{\rot(w)}]$.
The isomorphism of rings is degree preserving after re-multigrading $K[X_{\rot(w)}]$ by replacing $\ba$ with $\tilde{\ba}$ and $\bb$ with $\tilde{\bb}$ so that the coordinate function picking out the $(i,j)$-entry of $M\in X_{\rot(w)}$ has degree $\tilde{a}_i-\tilde{b}_j$.
\end{lemma}

\comment{
\begin{proof}
The first part is clear. The second statement follows by observing that the map $\rot$ sending $X^{w}$ to $X_{\rot(w)}$ is 
$T = (K^\times)^k \times (K^\times)^l$ equivariant with respect to the morphism $\eta: T\rightarrow T$ given by %$\text{Rev}: (K^\times)^k \times (K^\times)^l\rightarrow (K^\times)^k \times (K^\times)^l$
\begin{equation}
((p_1,p_2\dots,p_k),(q_1,q_2\dots,q_l))\mapsto ((p_k,\dots, p_2,p_1),(q_l,\dots,q_2,q_1)).
\end{equation}
\end{proof}
}

%%%%%%%%%%%%%%%%%%%%%
%%%%%%%%%%%%%%%%%%%%%

\subsection{$K$-polynomials and Grothendieck polynomials}\label{sect:Kpolys}
We briefly introduce $K$-polynomials here, following \cite[Ch.~8]{MS}.  Geometrically, these represent classes in equivariant $K$-theory.
Let $S$ be a polynomial ring over $K$, multigraded by $\ZZ^d$.  Assume the multigrading is \emph{positive} (which is the case for all multigradings which appear before \S5 of this paper), meaning that the only elements of degree $\mathbf{0}$ are the constants in the base field $K$.  Let $\ba=a_1, \dotsc, a_d$ be an alphabet with $d$ letters, and identify $\ZZ^d$ with the free abelian group on $\ba$.   For each $\be = \sum e_i a_i \in \ZZ^d$, where $e_i \in \ZZ$, we define the monomial $\ba^\be = a_1^{e_1} \cdots a_d^{e_d}$.  Note that $\be$ can be recovered from $1-\ba^\be$ by substituting $1-a_i$ for each $a_i$, substituting the power series expansion $(1-a_i)^{-1} = \sum_{j\geq 0} a_i^j$ as necessary to get all positive exponents, and then $e_i$ is the coefficient of $a_i$.  
The expressions $\be$ and $\ba^\be$ are sometimes referred to as ``additive'' versus ``multiplicative'' notation in the literature.

The \emph{Hilbert series} of a finitely generated, multigraded $S$-module $M = \bigoplus_{\be\in \ZZ^d}M_{\be}$ is the expression $H(M; \ba) = \sum_{\be \in \ZZ^d} (\dim_K M_\be) \ba^\be$
in the additive group $\prod_{\be \in \ZZ^d} \ZZ \ba^\be$. %Such series are often called \emph{Laurent series} in $\ba$.
The \emph{$K$-polynomial} of $M$ can be defined as
\begin{equation}\label{eq:KHilbert}
\cK_S(M; \ba) = \frac{H(M; \ba)}{H(S; \ba)}
\end{equation}
where $S$ is considered as a module over itself on the right hand side.  This ratio is actually a Laurent polynomial, meaning an element of $\bigoplus_{\be \in \ZZ^d} \ZZ \ba^\be$.  The $K$-polynomial can also be computed from a multigraded free resolution of $M$.  In this paper we are most interested in modules $M = S/I$ where $I$ is a homogeneous ideal in $S$; in this case, we sometimes write $\cK_X(Y; \ba)$ for $\cK_S (S/I; \ba)$, where $Y=\Spec(S/I) \subseteq X=\Spec(S)$.  %Also note that if we extend $I$ to a larger polynomial ring $S[x]$, then the $K$-polynomial of $S/I$ does not change.

The \emph{multidegree} of $M$ is obtained from $\cK_S(M; \ba)$ by substituting $1-a_i$ for each variable $a_i$, substituting $(1-a_i)^{-1} = \sum_{j\geq 0} a_i^j$ as necessary to get positive exponents, then taking the terms of lowest total degree.  It generalizes the classical notion of degree to the multigraded setting, and when $M = S/I$, it can be geometrically interpreted as the class of $\Spec(S/I)$ in the equivariant Chow ring $A^*_T(\Spec(S))$.

Our formulas for $K$-polynomials of quiver loci will be in terms of the double Grothendieck polynomials of Lascoux and Sch\"utzenberger \cite{LS}, which we briefly review now, following the definitions in \cite{KM05}.
Let $\ba = a_1, a_2, \dotsc, a_m$ and $\bb =b_1, b_2, \dotsc, b_m$ be alphabets and $w_0$ the longest element of the symmetric group $S_m$. First we define
\[
\fG_{w_0}(\ba; \bb) = \prod_{i+j \leq m} \left(1 - \frac{a_i}{b_j}\right).
\]
The \emph{Demazure operator} $\overline{\partial_i}$ acts on polynomials in $\ba$ by
\[
\overline{\partial_i} f (a_i, a_{i+1}) = \frac{a_{i+1}f(a_i, a_{i+1}) - a_i f(a_{i+1}, a_i)}{a_{i+1} - a_i}
\]
where $f$ is written as a polynomial in only $a_i, a_{i+1}$ with the other variables considered as coefficients.  With this, we can inductively define the \emph{double Grothendieck polynomial} of a permutation $v \tau_i\in S_m$ as $\fG_{v \tau_i} (\ba; \bb) = \overline{\partial_i} \fG_{ v} (\ba; \bb)$ whenever $\ell(v \tau_i) < \ell(v)$.  This differs slightly from the definition in \cite{fultonlascoux}; their Grothendieck polynomials $G_v(\ba; \bb)$ are related to ours by the substitution $G_v(\ba; \bb) = \fG_v(\ba^{-1}; \bb^{-1})$, where $\ba^{-1}=a_1^{-1}, a_2^{-1}, \dotsc,a_m^{-1}$ and similarly for $\bb^{-1}$.  The \emph{double Schubert polynomial} $\fS_v (\ba; \bb)$ of a permutation $v$ is obtained from $\fG_{v} (\ba; \bb)$ by the same process that a multidegree is obtained from a $K$-polynomial.  

Double Grothendieck polynomials give $K$-polynomials of matrix Schubert varieties through the following theorem.  Note that this theorem was proven in the language of degeneracy loci in \cite[Thm.~2.1]{MR1932326}; the formulation below is \cite[Thm.~A]{KM05}.  The multidegree version appeared in \cite{Fulton}.

\begin{theorem}\label{thm:Kpolymatrixschubert}
Let $X=\Mat(m,m)$ and $v\in S_m$. 
The $K$-polynomial of the northwest matrix Schubert variety $X_v\subseteq X$ is the double Grothendieck polynomial
$\cK_{X} (X_v; \ba, \bb) = \fG_{v} (\ba; \bb)$, and its multidegree is $\fS_v(\ba; \bb)$.
\end{theorem}

If $w$ is a $k \times l$ partial permutation matrix, then only the variables $a_1, \dotsc, a_k$ and $b_1, \dotsc, b_l$ actually appear in $\fG_{c(w)}(\ba; \bb)$, so we may define $\fG_{w} (a_1, \dotsc, a_k; b_1, \dotsc, b_l) := \fG_{c(w)} (\ba; \bb)$.
Since $K$-polynomials are invariant under extension of an ideal in a polynomial ring to a larger polynomial ring, the relation between defining ideals $J_{c(w)} = J_w K[\Mat(m,m)]$ implies 
\begin{equation}\label{eq:Kpartialperm}
\cK_{\Mat(k,l)} (\Mat(k,l)_w; a_1, \dotsc, a_k, b_1, \dotsc, b_l)=  \fG_{w} (a_1, \dotsc, a_k; b_1, \dotsc, b_l).
\end{equation}

\comment{
Suppose $w$ is a $k \times l$ partial permutation matrix and $c(w)$ is $m\times m$. Since $K$-polynomials are invariant under extension of an ideal in a polynomial ring to a larger polynomial ring, the relation between defining ideals $J_{c(w)} = J_w K[\Mat(m,m)]$ implies 
\begin{equation}\label{eq:Kpartialperm}
\cK_{\Mat(k,l)} (\Mat(k,l)_w; a_1, \dotsc, a_k; b_1, \dotsc, b_l)
\end{equation}
}
%%%%%%%%%%%%%%%%%%%%%%%%%%%%%%%%%%%%%%%%%%%%%%%%%%%%%%%%%%%%%%%%%%

\subsection{Pipe dreams}\label{sect:pipeDreams}
Our formulas use the language of pipe dreams, an introduction to which can be found in \cite[\S16.1]{MS}. These are the same as the RC-graphs originally introduced by Bergeron and Billey \cite{BB}, which are based on the pseudo-line arrangements introduced by Fomin and Kirillov to the study of Grothendieck and Schubert polynomials \cite{MR2307216,FK}.  The papers \cite{BRspec, wooyong} also give helpful treatments of the topic.

Consider a $k \times l$ grid of squares.  We use standard matrix terminology to refer to positions on the grid: horizontal strips of the squares are rows, vertical strips of squares are columns, and the square in row $i$ from the top and column $j$ from the left is labeled $(i,j)$.  A \emph{pipe dream} on this grid is a subset $P \subseteq \{1, \dotsc, k\} \times \{1, \dotsc, l \}$, visualized by tiling the grid using two kinds of tiles. For each $(i,j)$ in the grid, place the tile according to the rule below:
\[
\textcross\ \text{(cross)} \ \text{if }(i,j) \in P, \qquad \textturn\ \text{(elbow)}\ \text{if }(i,j) \not \in P.
\] 
We write $|P|$ for the cardinality of the subset $P$ (the number of cross tiles), and $\rot(P)$ for the pipe dream obtained by $180^\circ$ rotation of $P$.
Two examples of pipe dreams on a $7 \times 7$ grid are in Figure \ref{fig:pipes}; we have $|P_1| = 27$ and $|P_2| = 28$.  The row and column labels by $\bt, \bs$ come from identification of this grid with the northwest quadrant of $\mcell$ as in Figure \ref{fig:multigrading}. 
The outline of the snake region and colors of the pipes are a visual aide for the combinatorics relating lacing diagrams and pipe dreams developed in \S \ref{sect:pipeToLace}.
%\jenna{We haven't defined multigradings yet, so maybe these labels should be removed. Also, we should indicate here that the colors will only be meaningful later once we discuss pipes to laces.}
\begin{figure}
$P_1 = \vcenter{\hbox{\begin{tikzpicture}[scale=1,>=latex]
    \pipedream{0.5}{(0,0)}{$s^3_1$,$s^3_2$,$s^2_1$,$s^2_2$,$s^2_3$,$s^1_1$,$s^1_2$}
    {$t^0_1$,$t^1_1$,$t^1_2$,$t^2_1$,$t^2_2$,$t^3_1$,$t^3_2$}
    {%Always + part
      1/1/1/pink/pink,1/2/1/pink/pink,1/3/1/pink/pink,1/4/1/pink/pink,1/5/1/pink/pink,
      2/1/1/pink/pink,2/2/1/pink/pink,
      3/1/1/pink/pink,3/2/1/pink/pink,
      4/6/1/pink/pink,4/7/1/pink/pink,
      5/6/1/pink/pink,5/7/1/pink/pink,
      6/3/1/pink/pink,6/4/1/pink/pink,6/5/1/pink/pink,6/6/1/pink/pink,6/7/1/pink/pink,
      7/3/1/pink/pink,7/4/1/pink/pink,7/5/1/pink/pink,7/6/1/pink/pink,7/7/1/pink/pink,
      %variable pipes
      1/6/0/pink/red,1/7/0/red/black,
      2/3/0/pink/pink,2/4/0/pink/black,2/5/1/red/black,2/6/0/black/black,2/7/0/black/pink,
      3/3/0/pink/black,3/4/1/black/black,3/5/1/red/black,3/6/0/black/pink,3/7/0/pink/pink,
      4/1/0/pink/pink,4/2/0/pink/black,4/3/0/black/black,4/4/0/black/red,4/5/0/red/pink,
      5/1/0/pink/black,5/2/0/black/black,5/3/0/black/red,5/4/0/red/pink,5/5/0/pink/pink,
      6/1/0/black/red,6/2/0/red/red,
      7/1/1/red/red,7/2/0/red/pink}

    %outline of snake
    \latticepath{0.5}{(3.5,-0.5)}{very thick}{
     0/-2/black,-2/0/black,0/-2/black,-3/0/black,0/-2/black,-2/0/black,0/2/white,0/2/black,2/0/black,0/2/black,3/0/black,0/1/black,2/0/black}
% internal green wall
    \latticepath{0.5}{(3.5,0)}{dashed}{
    0/-1/black,-2/0/black,0/-2/black,-3/0/black,0/-2/black,-2/0/black,0/-2/black}
  \end{tikzpicture}}}$
  \qquad \qquad
$P_2 = \vcenter{\hbox{\begin{tikzpicture}[scale=1,>=latex]
    \pipedream{0.5}{(0,0)}{$s^3_1$,$s^3_2$,$s^2_1$,$s^2_2$,$s^2_3$,$s^1_1$,$s^1_2$}
    {$t^0_1$,$t^1_1$,$t^1_2$,$t^2_1$,$t^2_2$,$t^3_1$,$t^3_2$}
    {%Always + part
      1/1/1/pink/pink,1/2/1/pink/pink,1/3/1/pink/pink,1/4/1/pink/pink,1/5/1/pink/pink,
      2/1/1/pink/pink,2/2/1/pink/pink,
      3/1/1/pink/pink,3/2/1/pink/pink,
      4/6/1/pink/pink,4/7/1/pink/pink,
      5/6/1/pink/pink,5/7/1/pink/pink,
      6/3/1/pink/pink,6/4/1/pink/pink,6/5/1/pink/pink,6/6/1/pink/pink,6/7/1/pink/pink,
      7/3/1/pink/pink,7/4/1/pink/pink,7/5/1/pink/pink,7/6/1/pink/pink,7/7/1/pink/pink,
      %variable pipes
      1/6/0/pink/red,1/7/0/red/black,
      2/3/0/pink/pink,2/4/0/pink/black,2/5/1/red/black,2/6/0/black/black,2/7/0/black/pink,
      3/3/0/pink/black,3/4/1/black/black,3/5/1/red/black,3/6/0/black/pink,3/7/0/pink/pink,
      4/1/0/pink/pink,4/2/0/pink/black,4/3/0/black/black,4/4/0/black/red,4/5/0/red/pink,
      5/1/0/pink/black,5/2/0/black/black,5/3/0/black/red,5/4/0/red/pink,5/5/0/pink/pink,
      6/1/0/black/red,6/2/1/red/red,
      7/1/1/red/red,7/2/0/red/pink}

    %outline of snake
    \latticepath{0.5}{(3.5,-0.5)}{very thick}{
     0/-2/black,-2/0/black,0/-2/black,-3/0/black,0/-2/black,-2/0/black,0/2/white,0/2/black,2/0/black,0/2/black,3/0/black,0/1/black,2/0/black}
% internal green wall
    \latticepath{0.5}{(3.5,0)}{dashed}{
    0/-1/black,-2/0/black,0/-2/black,-3/0/black,0/-2/black,-2/0/black,0/-2/black}
  \end{tikzpicture}}}$
    \caption{Pipe dreams related to the running example}
\label{fig:pipes}
\end{figure}
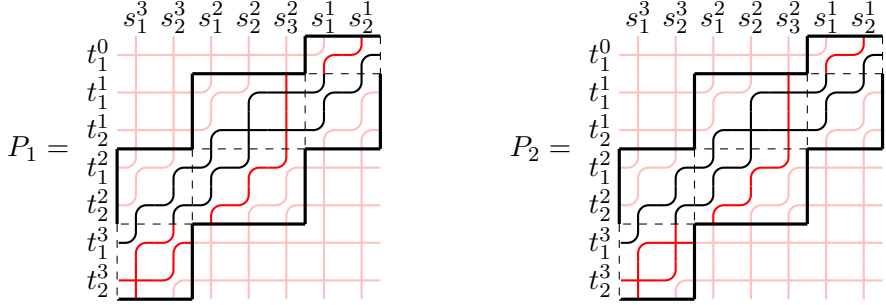

A pipe dream $P$ determines a word in the Coxeter generators $\tau_1, \tau_2, \dotsc$ of the symmetric group on $\ZZ_{> 0}$, as follows.  Starting at the top right of the grid, reading first along a row from right to left, then proceeding down to the next row, one writes the letter $\tau_{i+j-1}$ whenever $(i,j) \in P$.  The word is written left to right.  Evaluating this word using the relations
\begin{equation}
\tau_i^2 = \tau_i, \qquad \tau_i \tau_{i+1} \tau_i = \tau_{i+1} \tau_i \tau_{i+1}, \qquad \tau_i \tau_j = \tau_j \tau_i, \text{ for } |i-j| \geq 2
\end{equation}
yields a permutation $\delta(P)$ known as the \emph{Demazure product} of $P$.
Given a permutation $v$, we say $P$ \emph{is a pipe dream for} $v$ if $\delta(P)=v$.
If $|P| = \ell(\delta(P))$, then $P$ is \emph{reduced}.  
Note that $\delta$ is order-preserving in the sense that $P' \subseteq P$ implies that $\delta(P') \leq \delta(P)$ in the Bruhat order, and that every $P$ contains a reduced $P' \subseteq P$ such that $\delta(P')=\delta(P)$.
Both pipe dreams in Figure \ref{fig:pipes} have Demazure product equal to the Zelevinsky permutation in Figure \ref{fig:Zperm}.  
The pipe dream on the left is reduced; the one on the right not, 
because $|P_2|=28$
while $\ell(\delta(P)) = 27$. 
We also see two pipes crossing twice in the bottom square of $P_2$, which does not happen in reduced pipe dreams.

We define a $k \times l$ partial permutation matrix $w(P)$ associated to a reduced pipe dream $P$ on a $k \times l$ grid by \emph{following the pipes} from the left boundary to the top boundary. That is, $w(P)$ has a 1 in position $(i,j)$ precisely when $P$ has a pipe connecting row $i$ of its left boundary to column $j$ of its upper boundary.
The following lemma is straightforward and its proof is omitted.

\begin{lemma}\label{lem:pipefacts}
Let $P$ be a reduced pipe dream on a $k \times l$ grid and let $w=w(P)$.  Then $\delta(P)=c(w)$ as permutation functions in $S_{k+l}$ %\footnote{According to the definition of $c(w)$ in  \S\ref{sect:matrices}, $c(w)$ may be an element of $S_m$ with $m<k+l$. 
%In this case, we naturally consider $c(w)$ as an element of $S_{k+l}$ by declaring $c(w)(i) := i$ for $m<i\leq k+l$.} %Here, both $\delta(P)$ and $c(w)$ are interpreted as elements of $S_{k+l}$.}
if and only if no pair of pipes passing through the bottom of the grid cross, nor do any pair of pipes passing through the right of the grid.
%\noindent (b) If $v$ is a $k \times l$ partial permutation function and $P$ a reduced pipe dream on any size grid such that $\delta(P)=v$, then all cross tiles of $P$ lie in its northwest $k \times l$ grid.
\end{lemma}
%\ryan{(a) is used to see that $\delta(P)$ produces actual lacing diagrams for pipes to laces.  (b) is used in Lemma \ref{lem:inPrec}.}
%\ryan{reason: if $P$ had a $+$ in column $j>l$, then so would $P_{top}$.  The reduced word for $\delta(P)$ read off from $P_{top}$ would then end in $s_j$ (at right), so $\delta(P)$ would have a descent at $j$, a contradiction.  Use a similar argument to see $P$ has no $+$ in row $i > k$, considering $P_{bottom}$.}

In this paper we are primarily concerned with pipe dreams on a $d\times d$ grid for which all cross tiles lie in the northwest $d_y \times d_x$ rectangle (e.g., Figure \ref{fig:pipes}).  We identify this $d_y \times d_x$ rectangle with the northwest quadrant in Figure \ref{fig:multigrading} and only draw this rectangle.
 For $v \in S_d$, denote by $\Pipes(v_0, v)$ the set of pipe dreams $P$ for $v$ such that $P$ is a subset of this northwest quadrant, and define $\RedPipes(v_0, v)$ to consist of the reduced pipe dreams in $\Pipes(v_0,v)$.

Fix a $k \times l$ grid which will be tiled to produce a pipe dream. Let $\ba$ and $\bb$ be alphabets indexing the rows and columns of the grid, respectively.  For a pipe dream $P$, define
\[
(\one - \ba / \bb)^P = \prod_{(i,j) \in P}(1-a_i / b_j) \quad \text{and} \quad (\ba - \bb)^P = \prod_{(i,j) \in P}(a_i - b_j).
\]
The following formulas for $K$-polynomials of Kazhdan-Lusztig varieties are special cases of those presented in \cite[Thm.~4.5]{wooyong}.  These formulas previously appeared in other forms in \cite{AJS, billey, graham, willems}.

\begin{theorem}\label{thm:schubertpipe}
The $K$-polynomial of the Kazhdan-Lusztig variety $Y_v$ inside $\mcell$ is given by each of the following two formulas:
\begin{equation}\label{eq:wooyong1}
\mathcal{K}_{\mcell}(Y_v; \bs, \bt) = \sum_{P \in \Pipes(v_0,v)} (-1)^{|P|-\ell(v)}(\one - \bt / \bs)^{P}
\end{equation}
\begin{equation}\label{eq:wooyong2}
\mathcal{K}_{\mcell}(Y_v; \bs,\bt) = \fG_{ v}(\bt,\bs; \bs,\bt).
\end{equation}
\end{theorem}

%%%%%%%%%%%%%%%%%%%%%%
%%%%%%%%%%%%%%%%%%%%%%
%%%%%%%%%%%%%%%%%%%%%%
%%%%%%%%%%%%%%%%%%%%%%

\section{Bipartite ratio and pipe formulas}
We introduce a shorthand for the $K$-polynomials studied in this section and the next, where we continue to omit the fixed bipartite quiver $Q$ of type $A$ and dimension vector $\bd$ from the notation.

\begin{definition}\label{def:quiverPoly}
Let $\zO$ be a quiver locus for a bipartite type $A$ quiver $Q$. The \emph{$K$-theoretic quiver polynomial} $K\cQ_{\zO}(\bt/\bs)$ (resp., \emph{quiver polynomial} $\cQ_\zO (\bt - \bs)$) is the $K$-polynomial (resp., multidegree) of $\zO$ with respect to its inclusion in $\rep$ and multigrading of \S \ref{sect:multigrading}. 
\end{definition}

\subsection{Bipartite ratio formula}
We include a proof of the ratio formula for completeness, though it follows easily from our construction of the bipartite Zelevinsky map using the same argument as the equioriented case \cite[Thm.~2.7]{KMS}.

\begin{theorem}[Bipartite ratio formula]\label{thm:biratio}
For any bipartite type $A$ quiver locus $\zO$, we have
\begin{equation}\label{eq:bipartiteratio}
K\cQ_\zO(\bt / \bs) = \frac{\fG_{v(\zO)}(\bt,\bs; \bs,\bt) }{\fG_{v_*}(\bt,\bs; \bs,\bt) }.
\end{equation}
\end{theorem}

\begin{proof} 
Omitting the variables $\bs, \bt$ throughout the proof, we have the following equations relating Hilbert series to $K$-polynomials from \eqref{eq:KHilbert}:
\[
H(\zO) = \cK_{\mathtt{rep}} (\zO) H(\mathtt{rep}), \quad H(\zO) = \cK_{\mcell} (\zeta(\zO)) H( \mcell), \quad H(\mathtt{rep}) = \cK_{\mcell} (\zeta(\mathtt{rep})) H(\mcell).
\]
Substituting the second two equations above into the first and rearranging yields
\begin{equation} \label{eq:ratioproof}
K\cQ_{\zO}(\bt/\bs) = \cK_{\mathtt{rep}} (\zO) = \frac{\cK_{\mcell}(\zeta(\zO))}{\cK_{\mcell}(\zeta(\mathtt{rep}))}.
\end{equation}
Now $\zeta(\zO) =  Y_{v(\zO)}$ and $\zeta(\mathtt{rep}) = Y_{v_*}$ by Theorem \ref{thm:Zmap}, so  
\eqref{eq:wooyong2} implies the result.% (since $K\cQ_{\zO}(\bt/\bs)$ is the $K$-polynomial $\cK_{\mathtt{rep}} (\zO)$).
\end{proof}

%%%%%%%%%%%%%%%%%%%%%
%%%%%%%%%%%%%%%%%%%%%

\subsection{Bipartite pipe formula}\label{sect:bipipedream}

Next, we provide a formula for $K\cQ_\zO (\bt / \bs)$ in terms of pipe dreams. To begin, we show that all pipe dreams in $\Pipes(v_0,v(\zO))$ contain a particular pipe dream as a subset: define $P_*$ to be the pipe dream on a $d \times d$ grid which only has cross tiles in the northwest $d_y \times d_x$ rectangle, and within this rectangle there is a cross tile in location $(i,j)$ if and only if $(i,j)$ lies outside the snake region (Definition \ref{def:snake}).

\begin{lemma}\label{lem:homsubdiag}
  Each $P\in \Pipes(v_0,v(\zO))$ contains $P_*$ as a subset. In particular, $\Pipes(v_0, v_*) = \{P_*\}$ for the Zelevinsky permutation 
  $v_*$ associated to the entire representation space $\rep$.
\end{lemma}
\begin{proof}
The chain of closed subvarieties $\zeta(\zO) \subseteq \zeta(\rep) \subseteq \mcell$ induces a chain of ideals $I_{v_*} \subseteq I_{v(\zO)} \subseteq K[\mcell]\simeq K[z_{ij}]$, where $\{z_{ij}\}$ are the coordinate functions picking out matrix entries which are non-constant on $\mcell$.  Consider the monomial order $\prec$ on $K[\mcell]$ defined by $z_{ij} \prec z_{kl}$ when either $j < l$, or $j=l$ and $k < i$.
Then we have
\begin{equation}\label{eq:pstarequation}
I_{v_*}= \langle z_{ij}\mid (i,j)\in P_*\rangle = \init_\prec I_{v_*} \subseteq \textrm{in}_\prec I_{v(\zO)} = \bigcap_{P\in \RedPipes(v_0,v(\zO))}\langle z_{ij}\mid (i,j)\in P\rangle
\end{equation} 
where the second equality is because $I_{v_*}$ is generated by linear monomials (thus equal to its initial ideal with respect to any monomial order), and the last equality is \cite[Thm~3.2]{wooyong}.  This shows that every $P\in \RedPipes(v_0,v(\zO))$ contains $P_*$.
Since every (possibly nonreduced) pipe dream contains a reduced pipe dream for the same permutation as a subset, we obtain the first statement of the lemma.
For the second statement, take $\zO = \rep$ on the right hand side of \eqref{eq:pstarequation}.  Since this is a minimal prime decomposition of $I_{v_*}$, which is already prime, we find that $P_*$ is the unique element of $\RedPipes(v_0,v_*)$.  But every element of $\Pipes(v_0,v(\zO))$ is a union of elements of $\RedPipes(v_0,v(\zO))$ by \cite[Lem.~2]{MR2137947}, so we see that $\Pipes(v_0,v_*) = \{P_*\}$.
\end{proof}

\begin{theorem}[Bipartite pipe formula]\label{thm:bipipe}
For any bipartite type $A$ quiver locus $\zO$, we have
\begin{equation}\label{eq:Kbipipe}
K\cQ_\zO (\bt / \bs)= \sum_{P \in \Pipes(v_0,v(\zO))} (-1)^{|P\setminus P_*|-\codim{\zO}}(\one - \bt / \bs)^{P \setminus P_*}.
\end{equation}
\end{theorem}

\begin{proof}
Express the numerator and denominator of \eqref{eq:ratioproof} using \eqref{eq:wooyong1}, noting that the denominator has only one term by Lemma \ref{lem:homsubdiag}:
\begin{equation}
K\cQ_\zO (\bt / \bs) = \frac{\sum_{P \in \Pipes(v_0,v(\zO))} (-1)^{|P|-\ell(v(\zO))}(\one - \bt / \bs)^P}{(\one - \bt / \bs)^{P_*}}.
\end{equation}
By the same lemma, $(\one - \bt / \bs)^{P_*}$ divides each summand in the numerator. To obtain \eqref{eq:Kbipipe} it remains to check that $|P|-\ell(v(\zO)) = |P\setminus P_*|-\codim{\zO}$. This follows from \eqref{eq:Zpermlength} together with the equality $\ell(v_*) = |P_*|$.%yields \eqref{eq:Kbipipe}.  
\end{proof}

\section{Degeneration and bipartite component formula}\label{sect:degenAndComponent}
In this section we prove the bipartite $K$-theoretic component formula (Theorem \ref{thm:bicomponent}).  
We Gr\"obner degenerate a bipartite type $A$ quiver locus  to a reduced union of products of matrix Schubert varieties (Theorem \ref{thm:lacedegen}); this leaves the $K$-polynomial invariant, 
then we show that the bipartite $K$-theoretic component formula computes the $K$-polynomial of this degeneration.
%We use the fact that $K$-polynomials are invariant under Gr\"obner degeneration (see \S \ref{sect:degensetup}).
The degeneration is discussed in \S\S \ref{sect:degensetup} through \ref{sect:initReduced} and \S \ref{sect:componentproof}, while the $K$-polynomial computation appears in \S \ref{sect:Kcomponent}. The combinatorial results of \S \ref{sect:pipeToLace} are key ingredients in the proofs of Theorems \ref{thm:lacedegen} and \ref{thm:bicomponent}.

\subsection{General setup for degeneration}\label{sect:degensetup}

Given a variety $X$ with a right action of $K^\times$, any closed subvariety $Y \subseteq X$ determines a family $\tilde{Y} \subseteq X \times \mathbb{A}^1$ as follows. We take 
\begin{equation}\label{eq:Yfamily}
\tilde{Y}^\circ = \setst{(y \cdot t, t)}{y \in Y,\ t\in K^\times}.
\end{equation}
The family $\tilde{Y}$ is defined to be the closure of $\tilde{Y}^\circ$ in $X \times \mathbb{A}^1$  (where we identify $K^\times = \Spec K[t, t^{-1}] \subset \AA^1$).  
By construction the projection $\tilde{Y} \to \AA^1$ is $K^\times$-equivariant
with respect to the scaling action on $\AA^1$.
We denote the fiber over a point $t \in \mathbb{A}^1$ by $\tilde{Y}(t)$, 
which by the $K^\times$-equivariance is equal to $Y \cdot t \simeq Y$ for $t \neq 0$.

\subsubsection*{Gr\"obner degenerations} 
Let $S := K[x_1,\dots, x_m]$. Consider the case $X = \Spec S$ with a right $K^\times$-action, and $Y = \Spec S/I$ where $I\subseteq S$ is a homogeneous ideal. Using the induced right action of $K^\times$ on $S$, the family $\tilde{Y}$ is seen to be $\Spec S[t]/\tilde{I}$, where $\tilde{I}$ is the contraction of the ideal $\langle f\cdot t \mid f \in I \rangle \subseteq S[t, t^{-1}]$ to $S[t]$. In this subsection we show that the family $\tilde{Y}$ is flat over $\mathbb{A}^1 = \Spec K[t]$ and that the \emph{special fiber} $\tilde{Y}(0)$ can be computed by taking an \emph{initial ideal} of $I$. In this context, we say that $Y$ \emph{Gr\"obner degenerates} to $\tilde{Y}(0)$.

Let $\zl_1, \dotsc, \zl_m \in \ZZ$ with $K^\times$ acting on points of $X$ by $(p_1,\dots, p_m) \cdot t= (t^{\lambda_1}p_1,\dots, t^{\lambda_m}p_m)$.  The action determines an \emph{integral weight function} $\zl \colon \ZZ^m \to \ZZ$ defined by $\zl(\ba) := \sum_i \zl_i a_i$ for $\ba = (a_1, \dotsc, a_m) \in \ZZ^m$.
The induced right action of $K^\times$ on the coordinate ring $S$ acts on a monomial $\mathbf{x}^\mathbf{a} := x_1^{a_1}x_2^{a_2}\cdots x_m^{a_m}$  by $\bx^{\ba}\cdot t = t^{-\zl(\ba)} \bx^\ba$.  We define\footnote{This follows the sign convention of \cite[15.8]{eisenbud}. Since this torus action is used only for degeneration, and not a multigrading, there is no potential conflict with the convention of \S\ref{sect:multigrading}.} $\lambda(\bx^\ba) := \lambda(\ba)$, and say that $\zl$ is \emph{an integral weight function for $S$} and that the \emph{weight} of $x^\mathbf{a}$ is $\lambda(\mathbf{a})$. 
The original $K^\times$-action is easily recovered from $\lambda$, so the ideal $\tilde{I}\subseteq S[t]$, which defines the family $\tilde{Y}$, can be computed from $\lambda$. We refer to $\tilde{I}$ as the $\lambda$-\emph{homogenization of} $I$.

The \emph{initial form} $\init_{\lambda} f$ of a polynomial $f\in S$ is the sum of the terms of $f$ of highest $\lambda$-weight. The \emph{initial ideal} $\init_\lambda I$ of an ideal $I\subseteq S$ is $\init_\lambda I :=\langle \init_\lambda f\mid f\in I\rangle$.
We also use standard facts about \emph{monomial orders} $<$ on $S$ and their \emph{initial ideals} $\init_<I := \langle \text{in}_< f\mid f\in I\rangle$ (see \cite[Ch.~15]{eisenbud}).  For $Y = \Spec S/I$ and a monomial order $<$ on $S$, we write $\init_< Y :=\Spec(S/\init_< I)$.

\begin{theorem}\label{thm:grobnerDegeneration}
Let $I\subseteq S := K[x_1,\dots, x_m]$ be a homogenous ideal with respect to the standard grading. Let $\lambda\colon\mathbb{Z}^m\rightarrow \mathbb{Z}$ be an integral weight function. The following hold:
\begin{enumerate}[(a)]
\item $S[t]/\tilde{I}$ is a free (and thus flat) $K[t]$-module;
\item the fiber over $t\neq 0$ is isomorphic to $S/I$;
\item the fiber over $t=0$ is $S/\init_\lambda I$.
\end{enumerate}
\end{theorem}

\begin{proof}
This theorem is a variation of \cite[Thm.~15.17]{eisenbud}, which does not assume that $I$ is homogeneous, but whose proof assumes that the weight order induced by $\lambda$ can be refined to a monomial order.  Let $N$ be any integer so that $\lambda(x_i)+N>0$ for all $1\leq i\leq m$, and define a second weight function $\mu:\mathbb{Z}^m\rightarrow \mathbb{Z}$ with $\mu(x_i) = \lambda(x_i)+N$. 
Now let $f \in S$ be homogeneous with respect to the standard grading by degree.  Then $\init_{\lambda} f = \init_{\mu} f$, so $\init_\lambda I = \init_{\mu} I$ for any ideal $I$ which is homogeneous with respect to the standard grading. Furthermore, the $\lambda$-homogenization and $\mu$-homogenization of such an ideal are identical.  Now since $\mu(x_i) > 0$ for all $i$, the weight order induced by $\mu$ can be refined to a monomial order by breaking ties $\mu(\mathbf{x}^\mathbf{a}) = \mu(\mathbf{x}^\mathbf{b})$ with a lexicographical order on the variables.  Applying \cite[Thm.~15.17]{eisenbud} to $\mu$ gives the statements of this theorem for $\lambda$.
\end{proof}

We record two easy but useful facts for reference later.

\begin{lemma}\label{lem:degenfacts}
Let $S, I, \zl$ be as in Theorem \ref{thm:grobnerDegeneration}. 
\begin{enumerate}[(a)]
\item If $S$ is positively multigraded and $I\subseteq S$ is homogeneous with respect to this multigrading, then the $K$-polynomials of $S/I$ and $S/\init_\lambda I$ are equal.
%$S/I$ and $S/\init_\lambda I$ have $K$-polynomials for this multigrading, and these two $K$-polynomials are equal.
\item If $\Spec(S/I)$ is equidimensional, then the dimension of any irreducible component of the induced reduced subscheme of $\Spec(S/\init_\lambda I)$ is equal to the dimension of $\Spec(S/I)$.
\end{enumerate}
\end{lemma}
\begin{proof}
Part (a) is proven exactly as in \cite[Thm.~8.36]{MS}, except that in the sections preceding the proof, we replace their use of \cite[Prop.~8.26]{MS} (which is equivalent to \cite[Thm.~15.17]{eisenbud}) with Theorem \ref{thm:grobnerDegeneration} above.

In (b), each such irreducible component corresponds to a minimal prime of $S/\init_\lambda I=S[t]/(\tilde{I} + \langle t\rangle)$.  Applying Krull's principal ideal theorem to $\langle t \rangle \subseteq S[t]/\tilde{I}$, noting that $t$ is not a zero divisor on $S[t]/\tilde{I}$ by Theorem \ref{thm:grobnerDegeneration}(a), we find that the dimension of every irreducible component of $\Spec(S/\init_\lambda I)$ is equal to $\dim S[t]/\tilde{I} -1 = \dim S/I$.
\end{proof}

\subsubsection*{A family of groups acting on a family of schemes}
Our degeneration technique is an application of the following general setup.\footnote{We thank the referees for this observation, which greatly improves the conceptual clarity of our work.}
Let $G$ be an algebraic group, let $H \subset G$ be a closed subgroup, let $X$ be a $G$-variety, and $Y \subseteq X$ an $H$-stable closed subvariety.   Let $\mu \colon K^\times \to G$ be a group homomorphism, and consider the conjugation action of $K^\times$ on $G$ by $g \cdot t = \mu(t^{-1})\, g\, \mu(t)$, and induced action of $K^\times$ on $X$ by $x \cdot t = x \cdot \mu(t)$.  Then for $t \neq 0$, the fiber $\tilde{H}(t)$ is a subgroup acting on the fiber $\tilde{Y}(t)$ since the equation $(x\cdot t) \cdot (g\cdot t) = (x \cdot g) \cdot t$ is readily verified.  In our application, we will see that this action can be extended to the special fiber $t=0$.

\subsection{The flat families $\tilde{\GL_{\text{diag}}}$ and $\tilde{\rep}$}\label{sect:degenQuiverLoci}
We apply the general setup of the previous subsection to our situation. In that notation, $X=\rep$, while $Y=\Omega$, and $G=\GL^2:=\prod_{z\in Q_0}(\GL_{\bd(z)}\times \GL_{\bd(z)})$.  
We write a typical element of $\GL^2$ as $(\,(g^L_{z},g^R_{z})\,)_{z \in Q_0}$, where $g^L_z, g^R_z\in \GL_{\bd(z)}$. There is a right action of $\GL^2$ on $\rep$ given by
\begin{equation}\label{eq:GL2action}
(\dots,V_{\beta_i},V_{\alpha_i},\dots )\cdot (\,(g^L_{z},g^R_{z})\,)_{z \in Q_0} = (\dots,\, (g^R_{y_i})^{-1}V_{\beta_i}g^L_{x_i},\ (g^L_{y_{i-1}})^{-1}V_{\alpha_i}g^R_{x_i},\,\dots).
\end{equation}
The superscripts $L$ and $R$ are used to emphasize that $g^L_z$ (resp., $g^R_z$) acts on the matrix over the arrow to the left (resp., right) of vertex $z$.
Our $H$ is the diagonal subgroup $\GL_{\text{diag}} := \{(\,(g^L_{z},g^R_{z})\,)_{z \in Q_0}\mid g_z^L = g_z^R\}$ inside $\GL^2$.
Note that the action of $\GL_{\text{diag}}\simeq \GL$ on $\rep$ is the standard one of \S\ref{sect:quiverloci}, so each $\zO$ is stable under this action.

We next describe the group homomorphism $\mu\colon K^\times \to \GL^2$. For $t \in K^\times$, let $\rho_z(t)$ denote the diagonal matrix of size $\bd(z)\times \bd(z)$ which has the sequence of entries $t, t^2, t^3, \dotsc, t^{\bd(z)}$ down the diagonal, starting in the upper left. The map $\mu$ is then defined to be
\begin{equation}\label{eq:mu}
\mu\colon K^\times \to \GL^2,\qquad t\mapsto (\,(\rho_z(t^{-1}),\rho_z(t))\,)_{z \in Q_0}.
\end{equation}
Let $\widetilde{\GL_{\text{diag}}}$ and $\widetilde{\zO}$ denote the families obtained from $\mu$ as in the previous subsection.

\begin{proposition}\label{prop:familyacts}
The families $\widetilde{\GL_{\textup{diag}}}$ and $\tilde{\Omega}$ are flat, and $\widetilde{\GL_{\textup{diag}}}$ acts fiberwise on $\tilde{\Omega}$.
\end{proposition}
\begin{proof}
%For $\GL_{\text{diag}}$, we first work in 
Let $\Mat^2:=\prod_{z\in Q_0}(\Mat(\bd(z),\bd(z))\times \Mat(\bd(z),\bd(z)))$.  
There is a diagonal subvariety $\Mat_{\text{diag}} \subset \Mat^2$ defined analogously to $\GL_{\text{diag}}\subset \GL^2$.  Since $\GL^2$ acts on $\Mat^2$ by conjugation, the $K^\times$-action that $\mu$ induces on $\GL^2$ extends to a $K^\times$-action on $\Mat^2$.

%To prove that the families $\widetilde{\Mat_{\text{diag}}}$ and $\widetilde{\zO}$ are flat, we apply Theorem \ref{thm:grobnerDegeneration}(a) to homogeneous ideals
The coordinate rings $K[\Mat^2]$ and $K[\rep]$ are polynomial rings where the variables are coordinate functions picking out matrix entries. % of $\Mat^2$ and $\rep$ respectively. 
Observe that the defining ideal of a closed subvariety of either $\Mat^2$ or of $\rep$ is homogeneous with respect to the standard grading if and only if the subvariety is invariant under simultaneous scaling of all matrix entries.
It is clear that this holds for $\Mat_{\text{diag}}\subset \Mat^2$. It also holds for $\zO\subseteq \rep$ since $\Omega$ is invariant under $\GL$ and so in particular under simultaneous scaling of the matrix entries. %since this simultaneous scaling commutes with the action of $\GL$. 
Thus, the families $\widetilde{\Mat_{\text{diag}}}$ and $\widetilde{\zO}$ are flat by Theorem \ref{thm:grobnerDegeneration}(a). 
Restricting to the open subscheme $\widetilde{\GL_{\text{diag}}}\subset \widetilde{\Mat_{\text{diag}}}$ preserves flatness and puts us in the general setup of ``a family of groups acting on a family of schemes'' from \S \ref{sect:degensetup}, so there is a fiberwise action of $\widetilde{\GL_{\textup{diag}}}$ on $\tilde{\Omega}$ for $t \neq 0$. For $t=0$, the fiberwise action follows from \cite[Prop.~4.1]{KMS}.
\end{proof}

We now examine the special fiber of the family $\widetilde{GL_{\text{diag}}}$.  Define the subgroup $\cB_+ \times \cB_- \subseteq \GL^2$ as the set of elements $(\,(g^L_{z},g^R_{z})\,)_{z \in Q_0}$ where each $g^L_z$ is upper-triangular and each $g^R_z$ is lower-triangular. Then we denote by $\cB_+ \times_T \cB_-$ the subgroup of $\cB_+ \times \cB_-$ consisting of elements such that $g^L_{z}$ and $g^R_{z}$ have the same diagonal for each $z$.
The proof of the next lemma is as in \cite[Lem.~4.2]{KMS} with minor changes in notation, and thus omitted.

\begin{lemma}\label{lem:specialGL}
The special fiber $\widetilde{GL_{\textup{diag}}}(0)$ is equal to $\cB_+ \times _T \cB_-$.
\end{lemma}

\subsection{The special fiber $\tO(0)$ is reduced}\label{sect:initReduced}
We will prove that $\tO(0)$ is reduced using the algebraic language of integral weight functions and initial ideals.  Let $\xi$ be the integral weight function on $K[\rep]$ induced by the $K^\times$-action on $\rep$ obtained via the homomorphism $\mu: K^\times \rightarrow \GL^2$ from \eqref{eq:mu}. Thus,
\begin{enumerate}
\item[(W1)] the coordinate function picking out the $(k,l)$ entry of a matrix over an $\za$-type arrow has weight $k+l$;
\item[(W2)] the coordinate function picking out the $(k,l)$ entry of a matrix over a $\zb$-type arrow has weight $-k-l$.
\end{enumerate}

Throughout, we let $Z$ denote the universal matrix over $\mcell$, so $Z_{d_y\times d_x}=(z_{ij})$ is the region of variable entries in the northwest part of $Z$, and $K[\mcell] \simeq K[z_{ij}]$.
$$Z = \begin{bmatrix}
Z_{d_y\times d_x} &\bid_{d_y}\\
\bid_{d_x}&0\\
\end{bmatrix}$$ 
Recall that $P_*$ denotes the set of locations in $Z_{d_y\times d_x}$ which are \emph{outside} of the snake region. Identifying $K[\rep]$ with the subring $K[z_{ij}\mid (i,j)\notin P_*]$ of $K[z_{ij}]$, we extend $\xi$ to an integral weight function $\lambda$ on $K[z_{ij}]$ defined by $\lambda(z_{ij}) = \xi(z_{ij})$ if $(i,j)\notin P_*$, and $\lambda(z_{ij}) = 0$ if $z_{ij}\in P_*$ (it will be clear from the proofs below that the particular extension does not matter).  Equivalently, we have extended the $K^\times$-action on $\rep$ above to a $K^\times$-action on $\mcell$ by acting trivially on the matrix coordinates outside of the snake region.

We need two technical lemmas.
Fix sequences of integers $1 \leq i_1 < i_2 < \cdots < i_r \leq d_y$ and $1 \leq j_1 < j_2 < \cdots < j_r \leq d_x$ corresponding to $r$ distinct rows and $r$ distinct  columns of $Z_{d_y\times d_x}$, respectively, and let $Z'$ be the $r \times r$ matrix obtained from $Z_{d_y\times d_x}$ as the intersection of these rows and columns.  We refer to such $Z'$ as a \emph{submatrix} of $Z_{d_y\times d_x}$.  Consider the monomial $\bz = \prod_{k=1}^r z_{i_k, j_k}$. For each permutation $\sigma \in S_r$ set $\bz_\sigma = \prod_{k=1}^r z_{i_k, j_{\sigma(k)}}$, so that $\det(Z')= \sum_{\sigma \in S_r} (-1)^{\ell(\sigma)} \bz_\sigma$.
We are particularly interested in $\bz_{w_0}$, where $w_0\in S_r$ is the longest element, as usual.

\begin{lemma}\label{lem:simplifyminors}
Any nonzero minor of $Z$ equals $\pm \det(Z')$ for some submatrix $Z'$ of $Z_{d_y\times d_x}$.
\end{lemma}
\begin{proof}
This is a straightforward induction on the size of the minor.
Each minor of $Z$ is equal to $\text{det}(W)$ for a submatrix $W$ of $Z$. If $W$ has no constant entries, then $W$ is necessarily a submatrix of $Z_{d_y\times d_x}$.  So suppose that $W$ has some constant entry (i.e. a $1$ or $0$). Then either the entire row of $W$ or the entire column of $W$ containing that constant entry consists of constant entries, since that is true for $Z$. Furthermore that row or column of constants has at most one 1 and the rest are 0, for the same reason.  So either the minor is $0$, or it is equal to a minor of a strictly smaller submatrix of $W$ by cofactor expansion, in which case the result follows by induction.
\end{proof}

\begin{lemma}\label{lem:initlambda}
For any submatrix $Z'$ of $Z_{d_y\times d_x}$, if $\bz_{w_0} \not\in I_{v_*}$ then $\pm \bz_{w_0}$ is a term of $\init_\zl(\det Z')$.
\end{lemma}
\begin{proof}
Recall that $\tau_e \in S_r$ denotes the simple transposition switching $e$ and $e+1$.  We will prove below that for any $\sigma \in S_r$ and $\tau_e$ such that $\ell(\sigma \tau_e) < \ell(\sigma)$, both of the following hold:
\begin{enumerate}[(i)]
\item $\bz_\sigma \in I_{v_*} \Rightarrow \bz_{\sigma \tau_e} \in I_{v_*}$
\item $\bz_{\sigma \tau_e} \not\in I_{v_*}  \Rightarrow \zl(\bz_{\sigma\tau_e}) \leq \zl(\bz_\sigma)$.
\end{enumerate}
Assuming these, applying the contrapositive of (i) with downward induction on $\ell(\sigma)$ %shows that $\bz_\sigma \not\in I_{v_*}$ for all $\sigma$ whenever $\bz_{w_0} \not\in I_{v_*}$.
shows that if any $\bz_\sigma \notin I_{v_*}$ then also $\bz_{w_0}\notin I_{v_*}$. 
Furthermore, applying downward induction on $\ell(\sigma)$ along with (ii) shows that whenever $\bz_{w_0} \not \in I_{v_*}$,  the weight $\zl(\bz_{w_0})$ is maximal among weights of monomials in $\det(Z')$. Consequently, if $\det(Z')\notin I_{v_*}$ then $\pm \bz_{w_0}$ is a term of $\init_\zl(\det Z')$.

To prove (i) and (ii), note that the assumption $\ell(\sigma \tau_e) < \ell(\sigma)$ is equivalent to $\sigma(e+1) < \sigma(e)$.  Thus, the relative positions in $Z'$ of the only variables differing between $\bz_{\sigma \tau_e}$ and $\bz_\sigma$ can be visualized as below (where there may be additional columns of $Z'$ between the two columns seen here):
\begin{equation}\label{eq:zijinversion}
\begin{bmatrix}
z_{i_e, j_{\sigma(e+1)}} & z_{i_e, j_{\sigma(e)}} \\
z_{i_{e+1}, j_{\sigma(e+1)}} & z_{i_{e+1}, j_{\sigma(e)}}
\end{bmatrix}
=: 
\begin{bmatrix}
a & b \\
c & d
\end{bmatrix}.
\end{equation}
In this shorthand, $bc$ appears as a factor in $\bz_\sigma$, and $\bz_{\sigma \tau_e}$ is obtained from $\bz_\sigma$ by replacing the factor $bc$ with $ad$.

To prove (i), $\bz_\sigma \in I_{v_*}$ if and only if some factor $\bz_{i_k, j_{\sigma(k)}} \in I_{v_*}$, which is if and only if some position $(i_k, j_{\sigma(k)})$ is outside the snake region.  If $k \neq e, e+1$, then $\bz_{i_k, j_{\sigma(k)}}$ is a factor of $\bz_{\sigma \tau_e}$ as well, so that $\bz_{\sigma \tau_e} \in I_{v_*}$.
If $k=e$ and $(i_e, j_{\sigma(e)})$ is above the snake region, then so is $(i_e, j_{\sigma(e+1)})$ since $\sigma(e+1) < \sigma(e)$.  So the factor $a$ of $\bz_{\sigma \tau_e}$ is in $I_{v_*}$.
If $k=e$ and $(i_e, j_{\sigma(e)})$ is below the snake region, then so is $(i_{e+1}, j_{\sigma(e)})$ so the factor $d$ of $\bz_{\sigma \tau_e}$ is in $I_{v_*}$.
The two cases where $k=e+1$ are seen similarly, proving (i).

To prove (ii), the assumption $\bz_{\sigma \tau_e} \not\in I_{v_*}$ and (i) imply that all four matrix positions illustrated in  \eqref{eq:zijinversion} are inside the snake region.  
By the sentence below \eqref{eq:zijinversion}, we have $\zl(\bz_{\sigma\tau})= \zl(\bz_\sigma) +\varepsilon$ where we define
%\begin{equation}\label{eq:epsilondef}
$\varepsilon :=-\zl(b) - \zl(c) + \zl(a) + \zl(d)$.
%\end{equation}
There are three cases to consider.  The first is when $a,b,c,d$ are in the same block of the snake region, in which case $\varepsilon = 0$ by (W1) or (W2).  In the other two cases, the variables in \eqref{eq:zijinversion} are split between $\za$ and $\zb$ type blocks and (W1), (W2) give inequalities on the variables within each block:
\begin{equation}\label{eq:3cases}
  \begin{blockarray}{ccc}
      \begin{block}{c[cc]}
\za & a & b \\
\cline{1-3}
\zb & c & d \\
      \end{block}
    \end{blockarray}
\Rightarrow \zl(a) \leq \zl(b),\ \zl(c) \geq \zl(d);
\qquad
\begin{blockarray}{c|c}
 \za & \zb \\
      \begin{block}{[c|c]}
 a & b \\
 c & d \\
      \end{block}
    \end{blockarray}
 \Rightarrow \zl(a) \leq \zl(c),\ \zl(b) \geq \zl(d).
\end{equation}
In each case, it is immediate to verify that $\varepsilon \leq 0$, so (ii) holds.
\end{proof}

\begin{proposition}\label{prop:Oreduced}
The initial scheme $\tilde{\zO}(0)$ is reduced.
\end{proposition}
\begin{proof}
Since the integral weight function $\zl$ extends $\xi$, the Zelevinsky map $\zeta$ of Theorem \ref{thm:Zmap} is $K^\times$-equivariant with respect to the associated $K^\times$-actions.  Therefore, $\zeta$ induces an isomorphism $\widetilde{\zO} \simeq \widetilde{\zeta(\Omega)}$ which restricts to an isomorphism of each fiber $\widetilde{\zO}(t) \simeq \widetilde{\zeta(\Omega)}(t)$.  In particular, $\widetilde{\zO}(0)$ is reduced if and only if $\widetilde{\zeta(\Omega)}(0)$ is reduced, 
and the latter is equivalent to showing that $\init_{\lambda}I_{v(\zO)}$ is a radical ideal by Theorem \ref{thm:grobnerDegeneration}(c).
This is done by further degenerating via the monomial order $\prec$ defined by $z_{ij} \prec z_{kl}$ when either $j < l$, or $j=l$ and $k < i$.
Specifically, we will prove the key equality
\begin{equation}\label{eq:keyequality}
\init_{\prec}\init_{\lambda}I_{v(\zO)} = \init_{\prec}I_{v(\zO)}
\end{equation}
using the Gr\"obner basis
\begin{equation}\label{eq:GBKL}
\cG=\bigcup_{1 \leq p,q \leq d} \{\text{minors of }Z_{p \times q} \text{ of size }1+\rank v(\Omega)_{p\times q}\} 
\end{equation}
of $I_{v(\Omega)}$ from \cite[Thm.~2.1]{wooyong}.
Once we have this, it will follow that $\init_{\lambda} I_{v(\Omega)}$ is radical. Indeed, the definition of a Gr\"obner basis says that the initial terms of the elements of $\cG$ generate the right hand side of \eqref{eq:keyequality}. Since the initial term of any minor is squarefree, $\init_{\prec}I_{v(\zO)}$ is a squarefree monomial ideal and thus radical. Then from the key equality we have $\init_{\prec} \init_{\lambda} I_{v(\zO)}$ is radical, which implies the ideal $\init_{\lambda} I_{v(\Omega)}$ is also radical (for example, by \cite[Prop.~3.3.7]{HH}).

To prove the key equality, we first show that the right hand side is contained in the left.
Each nonzero element of $\cG$ is equal to $\pm \det(Z')$ for some submatrix $Z'$ of $Z_{d_y\times d_x}$ by Lemma \ref{lem:simplifyminors}.  So using that $\cG$ is a Gr\"obner basis of the right hand side, it is enough to show that the monomial $\init_\prec (\det Z')$ is in the left hand side of \eqref{eq:keyequality}. 
By the definition of $\prec$, it is easy to see that $\init_\prec \text{det}(Z')=\bz_{w_0}$.  Consider two cases of whether $\bz_{w_0}$ is in $I_{v_*}$ or not.  If so, then  $\bz_{w_0} \in I_{v_*}= \init_\prec \init_\zl I_{v_*}  \subseteq \textrm{in}_{\prec}\textrm{in}_{\lambda} I_{v(\Omega)}$, where the first equality holds since $I_{v_*}$ is generated by a subset of the variables, and the containment holds since $I_{v_*}\subseteq I_{v(\zO)}$. So in this case $\bz_{w_0}$ is in the left hand side of \eqref{eq:keyequality}.
On the other hand, if $\bz_{w_0} \not\in I_{v_*}$,  Lemma \ref{lem:initlambda} implies that $\pm \bz_{w_0}$ is a term of $\init_\zl(\det Z')$, so that 
$\pm \bz_{w_0}=\init_\prec \init_\zl (\det Z') \in \textrm{in}_{\prec}\textrm{in}_{\lambda} I_{v(\Omega)}$ shows that $\bz_{w_0}$ is in the left hand side of \eqref{eq:keyequality} in this case as well.
Therefore, the right hand side of \eqref{eq:keyequality} is contained in the left.

Now both ideals in \eqref{eq:keyequality} are initial ideals of $I_{v(\Omega)}\subseteq K[z_{ij}]$, so their associated quotient rings both have the same Hilbert series as $K[z_{ij}]/I_{v(\Omega)}$ (see \cite[Prop.~8.28]{MS}). % and the translation between $K$-polynomials and Hilbert series. 
%Thus, the Hilbert series of these two rings are equal, so the 
The equality of Hilbert series together with the fact that one side of \eqref{eq:keyequality} is contained in the other forces \eqref{eq:keyequality} to be an equality.
\end{proof}

%Finally, we collect some facts about the special fibers $\tO(0)$ and $\tGL(0)$.

\subsection{Relating pipe dreams and lacing diagrams}\label{sect:pipeToLace}
In this section we establish some combinatorial results which are necessary for our proof of the bipartite component formula. 
We define symmetric groups associated to the arrows of $Q$ by $S_{\za_k} = S_{\bd(y_{k-1})+\bd(x_k)}$ and $S_{\zb_k} = S_{\bd(y_k)+\bd(x_k)}$.
Consider the set $S_\bd:= \prod_{k=1}^n S_{\zb_k} \times S_{\za_k} $, noting that there is one factor for each arrow of $Q$. 
A typical element is denoted by $\bv = (v_n, v^n, \dotsc, v_1, v^1)$ with $v^k \in S_{\za_k}$ and $v_k \in S_{\zb_k}$,
and we let $|\bv| = \sum_k \ell(v^k) +\ell(v_k)$.

Let $P$ be a pipe dream on a $d_y\times d_x$ grid, such that $P_*\subseteq P$.  Let $P^k$ be the \emph{mini pipe dream} on a $\bd(y_{k-1})\times \bd(x_k)$ grid extracted from $P$ by restricting to the block of the snake region indexed by $\za_k$, and $P_k$ be the mini pipe dream on a $\bd(y_k)\times \bd(x_k)$ grid extracted from $P$ by restricting to the block of the snake region indexed by $\zb_k$.
Define an operation $\pi(P)=\bv \in S_\bd$ by $v^k =\delta(\rot(P^k))$ and $v_k =\delta(P_k)$, where $\delta(-)$ is Demazure product (see \S\ref{sect:pipeDreams}).

Let $\cL_\bd$ be the set of all lacing diagrams of dimension vector $\bd$.  We consider $\cL_\bd$ as a subset of $S_\bd$ by identifying a lacing diagram $\bw=(w_n, w^n, \dotsc, w_1, w^1)$, where $w_k$ is associated to $\zb_k$ and $w^k$ associated to $\za_k$, with $(c(w_n),\, c(\rot(w^n)), \dotsc, c(w_1),\, c(\rot(w^1))) \in S_\bd$.  Here, we consider each $c(w_k) \in S_{\zb_k}$ and $c(\rot(w^k)) \in S_{\za_k}$. 
Let $P$ be as above and furthermore assume that all $P_k,\, P^k$ are reduced pipe dreams. We define $\bw(P) = (w_n, w^n, \dotsc, w_1, w^1)\in \cL_\bd$ by $w_k = w(P_k)$ and $w^k=\rot(w(\rot(P^k)))$, where $w(-)$ is the ``follow the pipes'' operation defined in \S\ref{sect:pipeDreams}.

For example, taking the pipe dreams of Figure \ref{fig:pipes}, we have $\pi(P_2) = \pi(P_1)$.
Following the pipes in each block of $P_1$, which is reduced, yields that $\bw(P_1)$ is the lacing diagram of Figure \ref{fig:laces}.  Notice that some pipe crossings are only visible in the extended diagram of $\bw(P_1)$.

\begin{lemma}\label{lem:followPipes}
If $P\in \RedPipes(v_0, v(\zO))$, then $\pi(P)=\bw(P)$, under the identification of lacing diagrams with elements of $S_\bd$ described above.
\end{lemma}
\begin{proof}
Recall that the $1$s in the permutation matrix $v(\zO)$ appear northwest to southeast along block rows and columns by \eqref{eq:nwbr}.
Therefore, any reduced pipe dream for $v(\zO)$, when extended by elbow tiles to the full $d\times d$ grid, has the property that two pipes do not cross if those pipes enter the $d\times d$ grid on the left in the same block row or exit the $d\times d$ grid on the top in the same block column.  
By considering the configuration of cross and elbow tiles outside of the snake region, we see that two pipes that enter block $\beta_k$ on the bottom or exit on the right do not cross. Similarly, two pipes that enter block $\alpha_k$ from the left or exit on the top do not cross.
Thus Lemma \ref{lem:pipefacts} applied to each $P_k$ and $\rot(P^k)$ yields the result by direct substitution into the definitions of $\bw(P)$ and $\pi(P)$.  
\end{proof}

The following is a bipartite analog of \cite[Thm.~5.10]{KMS}. 

\begin{proposition}\label{prop:pipetolacea}
If $P \in \RedPipes(v_0, v(\zO))$, then $\bw(P) \in W(\zO)$.
\end{proposition}
\begin{proof}
To show that $\bw(P) \in W(\zO)$, we need to show first that $\bw \in \zO^\circ$, and then that $|\bw|$ is minimal among lacing diagrams in $\zO^\circ$.  For the first, we need to see that, for each pair of vertices $(z_i, z_j)$, the number of laces with left endpoint $z_i$ and right endpoint $z_j$ is the same in $\bw(P)$ as in a lacing diagram in $\zO^\circ$.

For a lacing diagram in $\zO^\circ$, the number of laces with left endpoint $z_i$ and right endpoint $z_j$ (where $z_i$ is weakly left of $z_j$) is the number of 1s in block $(z_i,z_j)$ of $v(\zO)$ by Theorem \ref{thm:lacesZperm}.  
Since $P$ is reduced and $\delta(P) = v(\zO)$, the number of ones in this block is the number of pipes in $P$ (extended to the full $d\times d$ grid as usual) from block row $z_i$ to block column $z_j$.  
We claim this is the number of laces in $\bw$ with left endpoint $z_i$ and right endpoint $z_j$.  
Using the correspondence between blocks of the snake region and arrows of $Q$, we get an identification of the vertices of $Q$ with certain walls of blocks of the snake region (see examples in Figure \ref{fig:walls}): namely, the wall associated to a vertex is the wall separating the blocks of the two arrows incident to that vertex (extending this pattern in the obvious way for vertices $y_0, y_n$).  
\begin{figure}
$\vcenter{\hbox{\begin{tikzpicture}[scale=1,>=latex]
    \pipedreamalt{0.5}{(0,0)}{}{}
    {%Always + part
      1/1/1/pink/pink,1/2/1/pink/pink,1/3/1/pink/pink,1/4/1/pink/pink,1/5/1/pink/pink,
      2/1/1/pink/pink,2/2/1/pink/pink,
      3/1/1/pink/pink,3/2/1/pink/pink,
      4/6/1/pink/pink,4/7/1/pink/pink,
      5/6/1/pink/pink,5/7/1/pink/pink,
      6/3/1/pink/pink,6/4/1/pink/pink,6/5/1/pink/pink,6/6/1/pink/pink,6/7/1/pink/pink,
      7/3/1/pink/pink,7/4/1/pink/pink,7/5/1/pink/pink,7/6/1/pink/pink,7/7/1/pink/pink,
      %variable pipes
      5/1/0/red/white,4/1/0/white/red,4/2/1/white/red,4/3/0/black/white,3/3/0/white/black,
      3/4/1/white/black,3/5/0/black/white,2/5/0/white/black,2/6/0/black/white,1/6/0/white/black,1/7/1/white/black}
      {\node at (3.5,-0.25) {$y_0$};\node at (3,-0.5) {$x_1$};\node at (2.5,-1) {$y_1$};\node at (1.75,-1.5) {$x_2$};
      \node at (1,-2) {$y_2$};\node at (0.5,-2.5) {$x_3$};\node at (0,-3) {$y_3$};}

    %outline of snake
    \latticepath{0.5}{(3.5,-0.5)}{very thick}{
     0/-2/black,-2/0/black,0/-2/black,-3/0/black,0/-2/black,-2/0/black,0/2/white,0/2/black,2/0/black,0/2/black,3/0/black,0/1/black,2/0/black}
% internal green wall
    \latticepath{0.5}{(3.5,0)}{dashed}{
    0/-1/black,-2/0/black,0/-2/black,-3/0/black,0/-2/black,-2/0/black,0/-2/black}
  \end{tikzpicture}}}$
  \qquad
  $\vcenter{\hbox{\begin{tikzpicture}[scale=1,>=latex]
    \pipedreamalt{0.5}{(0,0)}{}{}
    {%Always + part
      1/1/1/pink/pink,1/2/1/pink/pink,1/3/1/pink/pink,1/4/1/pink/pink,1/5/1/pink/pink,
      2/1/1/pink/pink,2/2/1/pink/pink,
      3/1/1/pink/pink,3/2/1/pink/pink,
      4/6/1/pink/pink,4/7/1/pink/pink,
      5/6/1/pink/pink,5/7/1/pink/pink,
      6/3/1/pink/pink,6/4/1/pink/pink,6/5/1/pink/pink,6/6/1/pink/pink,6/7/1/pink/pink,
      7/3/1/pink/pink,7/4/1/pink/pink,7/5/1/pink/pink,7/6/1/pink/pink,7/7/1/pink/pink,
      %variable pipes
      5/3/0/white/red,5/4/0/red/white,4/4/2/red/white,3/4/2/red/white,2/4/2/red/white}
      {\node at (3.5,-0.25) {$y_0$};\node at (3,-0.5) {$x_1$};\node at (2.5,-1) {$y_1$};\node at (1.75,-1.5) {$x_2$};
      \node at (1,-2) {$y_2$};\node at (0.5,-2.5) {$x_3$};\node at (0,-3) {$y_3$};}

    %outline of snake
    \latticepath{0.5}{(3.5,-0.5)}{very thick}{
     0/-2/black,-2/0/black,0/-2/black,-3/0/black,0/-2/black,-2/0/black,0/2/white,0/2/black,2/0/black,0/2/black,3/0/black,0/1/black,2/0/black}
% internal green wall
    \latticepath{0.5}{(3.5,0)}{dashed}{
    0/-1/black,-2/0/black,0/-2/black,-3/0/black,0/-2/black,-2/0/black,0/-2/black}
  \end{tikzpicture}}}$
    \caption{Visual aid for proof of Prop.~\ref{prop:pipetolacea}}
\label{fig:walls}
\end{figure}
Now by the definition of $\bw(P)$, each of its laces corresponds to a specific pipe in $P$, and that lace passes through a given vertex if and only if the corresponding pipe passes through the associated wall.  The claim can then be verified by inspection. 
 For example, a lace has left \emph{endpoint} $y_k$ exactly when the corresponding pipe passes through the wall labeled $y_k$ but not the wall labeled $x_{k+1}$, and this happens exactly when that pipe enters the snake region at left in block row $x_{k+1}$, and thus enters the overall grid $d \times d$ grid in block row $x_{k+1}$ due to the shape of $P_*$.  
 Similarly, the left \emph{endpoint} is $x_k$ exactly when the corresponding pipe travels through the wall labeled $x_k$ but not the wall labeled $y_k$, which happens exactly when that pipe enters the snake region at bottom in block column $y_k$, and thus enters the overall $d \times d$ grid in block row $y_k$ due to the shape of $P_*$ and extension by elbows outside the original $d_y \times d_x$ grid of $P$. 
 A visual aid is found in Figure \ref{fig:walls}: on the left grid we see a pipe contributing a lace with left endpoint $y_2$ and right endpoint $y_0$, and on the right grid we see a pipe contributing a lace whose left and right endpoints are both $x_2$.

Now the fact that $\bw$ is minimal follows from \eqref{eq:codimcrossings} and the computation:
\begin{equation}
|\bw| = \sum\nolimits_k \ell(w_k) +\ell(w^k) = |P \setminus P_*| = |P| - |P_*| = \ell(\delta(P)) - \ell(\delta(P_*)) = \ell(v(\zO)) - \ell(v_*) 
\end{equation} 
where the second equality holds because $P$ is reduced (so each extracted mini pipe dream is reduced), the third by Lemma \ref{lem:homsubdiag}, the fourth because $P, P_*$ are reduced, and the last by assumption.  Now the right hand side equals $\codim \zO$ by \eqref{eq:Zpermlength}.
\end{proof}

\subsection{Irreducible components of $\tilde{\zO}(0)$}\label{sect:componentproof}
In this section, we provide a complete description of $\tilde{\zO}(0)$ in terms of northwest and southeast matrix Schubert varieties (Theorem \ref{thm:lacedegen}). We begin with some technical lemmas. The proof of the first is as in \cite[Lem.~4.10]{KMS}, with minor changes in notation.

\begin{lemma}\label{lem:onefiberall}
If a lacing diagram $\bw$ is in the fiber $\tO(1)$,
then $\bw \in \tO(t)$ for all $t \in \AA^1$.
\end{lemma}

For a lacing diagram $\bw=(w_n, w^n, \dotsc, w_1, w^1)$, we define a product of matrix Schubert varieties (see \S\ref{sect:matSchubert}) inside $\rep$ by
\begin{equation}\label{eq:Owdef}
\cO_\bw := \prod_{k=1}^n \Mat(\bd(y_k), \bd(x_k))_{w_k} \times \prod_{k=1}^n \Mat(\bd(y_{k-1}), \bd(x_{k}))^{w^k}.
\end{equation}
Let $\mathcal{B}_+\times \mathcal{B}_-\subseteq \GL^2$ be as in \S\ref{sect:degenQuiverLoci}, with its right action on $\rep$ from \eqref{eq:GL2action}.

\begin{lemma}\label{lem:O(w)Description}
Let $V\in \rep$. The closure, inside of $\rep$, of the $\mathcal{B}_+\times \mathcal{B}_-$ orbit of $V$ is equal to $\cO_\bw$ for some lacing diagram $\bw$.
\end{lemma}
\begin{proof}
The $\mathcal{B}_+\times \mathcal{B}_-$-orbit of $V = (V_{\beta_n}, V_{\alpha_n},\dots, V_{\beta_1}, V_{\alpha_1})\in \rep$
can be expressed as
\begin{equation}\label{eq:B+B-}
V\cdot (\mathcal{B}_+\times \mathcal{B}_-) = \prod_{i=1}^n B_-V_{\beta_i}B_+ \times \prod_{i=1}^n B_+V_{\alpha_i}B_-.
\end{equation}
where each $B_+$ (resp., $B_-$) on the right hand side denotes a Borel subgroup of upper (resp., lower) triangular matrices of the appropriate size. Let $w_i$ (resp., $w^i$) denote the unique partial permutation matrix contained in $B_-V_{\beta_i}B_+$ (resp., $B_+V_{\alpha_i}B_-$). Then,
\begin{equation}
V\cdot (\mathcal{B}_+\times \mathcal{B}_-) = \prod_{i=1}^n B_-w_iB_+ \times \prod_{i=1}^n B_+w^iB_-.
\end{equation}
Taking closures on both sides shows that the $\mathcal{B}_+\times \mathcal{B}_-$ orbit closure of $V$, inside of $\rep$, is equal to $\cO_\bw$ for the lacing diagram $\bw$ with matrix form $(w_n,w^n,\dots, w_1,w^1)$.
\end{proof}

As in \S \ref{sect:initReduced}, let $Z$ denote the universal matrix over $\mcell$ and write $K[\mcell] = K[z_{ij}]$. For $P$ a pipe dream on a $d_y\times d_x$ grid, define $I_P:=\langle z_{ij} \mid (i,j) \in P\rangle \subseteq K[z_{ij}]$.

\begin{lemma}\label{lem:inPrec}
The scheme  $\init_\prec \zeta(\cO_\bw)$ is reduced, and each of its irreducible components is of the form $\Spec(K[z_{ij}]/I_P)$ where $P$ is a pipe dream on a $d_y\times d_x$ grid which contains $P_*$ and satisfies $\bw(P) = \bw$.
\end{lemma}

\begin{proof}
Let $I(\zeta(\mathcal{O}_\bw))$ denote the defining ideal of (the reduced scheme) $\zeta(\mathcal{O}_{\bw})$.  
We prove the equivalent algebraic statement: $\init_\prec I(\zeta(\mathcal{O}_\bw))$ is a radical ideal, and each prime of $K[z_{ij}]$ minimal over it is of the form $I_P$ with $P$ as in the lemma statement. 

For each $a \in Q_1$, let $Z_{a}$ denote the block of $Z$ corresponding to $a$, and let $K[Z_a]$ denote the polynomial subring of $K[z_{ij}]$ generated by the matrix entries of $Z_{a}$.  Write $z^a_{ij} \in K[Z_a]$ for the $(i,j)$ entry of $Z_a$.
Taking $\bw = (w_n, w^n, \dotsc, w_1, w^1)$ as usual, 
as defined in \S\ref{sect:matSchubert} let $J_k :=J_{w_k} \subseteq K[Z_{\beta_k}]$ (resp., $J^k:=J^{w^k}\subseteq K[Z_{\alpha_k}]$) be the ideal of the northwest (resp., southeast) matrix Schubert variety defined from $w_k$ (resp., $w^k$). 

Let $\prec$ be the monomial order on $K[\mcell]$ defined in the proof of Proposition \ref{prop:Oreduced}, and denote by the same symbol the restriction of this order to each $K[Z_{a}]$.  We need to first describe the minimal primes of the ideals $\init_\prec J_k$ and $\init_\prec J^k$.
For a pipe dream $P$ on a $\bd(ta) \times \bd(ha)$ grid, we define the corresponding ideal of $K[Z_a]$ by $I_P=\langle z^a_{ij} \mid (i,j) \in P\rangle$.

Since the initial term of any minor of $Z$ is its antidiagonal term, \cite[Thm.~B]{KM05} tells us that every prime of $K[Z_{\beta_k}]$ minimal over $\init_\prec J_k$ is of the form $I_P$ for some reduced pipe dream $P$ on a $\bd(y_{k}) \times \bd(x_k)$ grid such that $\delta(P)=c(w_k)$ (see Theorems 16.18 and 16.28 of \cite{MS} for the partial permutation version).
To study $\init_\prec J^k$, consider the isomorphism of coordinate rings $\rot\colon K[Z_{\za_k}] \xrightarrow{\sim} K[Z_{\za_k}]$ induced by $180^\circ$ rotation of matrices, noting that $J^k = \rot(J_{\rot(w^k)})$ by Lemma \ref{lem:SEmatrixschubert}.  
Although $\rot(\init_\prec g) \neq \init_\prec \rot(g)$ in general, equality does hold when $g$ is a minor. Indeed, in this setting, $\init_\prec g$ is the antidiagonal term of $g$, and $\rot$ takes the antidiagonal term of any (square) submatrix to the antidiagonal term of another submatrix. 
Since minors form a Gr\"obner basis here \cite[Thm.~B]{KM05}, we get $\init_\prec J^k = \init_\prec \rot(J_{\rot(w^k)})= \rot (\init_\prec J_{\rot(w^k)})$.

Since $\rot$ is an isomorphism, each prime of $K[Z_{\za_k}]$ minimal over $\rot (\init_\prec J_{\rot(w^k)})$ is of the form $\rot(\fp)$ where $\fp$ is a prime of $K[Z_{\za_k}]$ minimal over $\init_\prec J_{\rot(w^k)}$.  Again using \cite[Thm.~B]{KM05}, such $\fp$ is equal to $I_{P'}$ for some reduced pipe dream $P'$ on a $\bd(y_{k-1})\times \bd(x_k)$ grid such that $\delta(P') = c(\rot(w^k))$.  Thus we have shown that every prime of $K[Z_{\za_k}]$ minimal over $\init_\prec J^k$ is of the form $\rot(I_{P'}) = I_{\rot(P')}$ for some $P'$ as above.

Extending all ideals below to $K[z_{ij}]$, the defining ideal of $\zeta(\mathcal{O}_\bw)\subseteq \mcell$ is
\begin{equation}\label{eq:idealOw}
I(\zeta(\mathcal{O}_\bw)) = \sum_{k=1}^n J_k + \sum_{k=1}^n J^k+I_{P_*}.
\end{equation}
Distinct summands on the right side of \eqref{eq:idealOw} have reduced Gr\"obner bases in disjoint sets of variables, so the initial ideal of the sum on the right side of \eqref{eq:idealOw} is the sum of the initial ideals:
\begin{equation}\label{eq:initIdealOw}
\text{in}_\prec I(\zeta(\mathcal{O}_\bw)) = \sum_{k=1}^n \init_\prec J_k + \sum_{k=1}^n \init_\prec J^k+I_{P_*}.
\end{equation}
Since every summand on the right hand side of \eqref{eq:initIdealOw} is squarefree monomial ideal, $\text{in}_\prec I(\zeta(\mathcal{O}_\bw))$ is radical. 
Each of the irreducible components of $\text{in}_\prec \zeta(\mathcal{O}_\bw)$ comes from a prime of $K[z_{ij}]$ minimal over $\init_\prec I(\zeta(\mathcal{O}_\bw))$; fix such a prime.
% Since \ryan{explain why it is a sum of extensions of primes from above}, 
From above, it is of the form 
%\red{Now, from above, we know the primes in $K[z_{ij}]$ minimal over each summand of \eqref{eq:initIdealOw}. Consequently, each prime in $K[z_{ij}]$ minimal over $\init_\prec I(\zeta(\mathcal{O}_\bw))$ has the form:}
 \begin{equation}\label{eq:IPs}
 \sum_{k=1}^n I_{P_k} + \sum_{k=1}^n I_{\rot(P^k)} + I_{P_*}
 \end{equation}
 for some collection of mini pipe dreams $P_k,\, P^k$ as above.  
Identifying the grid of each $P_k$ and $\rot(P^k)$ with its corresponding block in the snake, we can consider each $P_k$ and $\rot(P^k)$ as a pipe dream on a $d_y \times d_x$ grid which has cross tiles only in that block.
We see that the prime in \eqref{eq:IPs} is equal to $I_P$ where $P = P_*\cup (\bigcup_{k=1}^n P_k)\cup(\bigcup_{k=1}^n \text{rot} (P^k))$.
By construction, the mini pipe dream of $P$ extracted from block $\zb_k$ is $P_k$, and $\delta(P_k) = c(w_k)$.  Similarly, the  mini pipe dream of $P$ extracted from block $\za_k$ is $\rot(P^k)$, and $\delta(\rot(\rot(P^k)))=\delta(P^k) = c(\rot(w^k))$, so we have by Lemma \ref{lem:pipefacts} that $\pi(P)=\bw(P) = \bw$.
\end{proof}

\begin{theorem}\label{thm:lacedegen}
A bipartite type $A$ quiver locus $\Omega \subseteq \rep$ degenerates to a reduced union of products of northwest and southeast of matrix Schubert varieties. In particular:
\begin{equation}\label{eq:quiverlocusdegen}
\tilde{\zO}(0) = \bigcup_{\bw \in W(\zO)} \cO_\bw.
\end{equation}
\end{theorem}

\begin{proof}
We first show that the scheme $\tilde{\Omega}(0)$, which is reduced by Proposition \ref{prop:Oreduced}, is a union of some collection of $\cO_\bw$. By Proposition \ref{prop:familyacts} and Lemma \ref{lem:specialGL}, $\tilde{\Omega}(0)$ is a union of $\cB_+ \times_T \cB_-$ orbits.
As the groups $\cB_+ \times \cB_-$ and $\cB_+ \times_T \cB_-$ have the same orbits in $\mathtt{rep}$ (which is proven just as in the equioriented case \cite[Prop.~3.4]{KMS}), we see that $\tilde{\Omega}(0)$ is a union of $\cB_+ \times \cB_-$ orbits. Thus, each irreducible component is a $\cB_+\times \cB_-$ orbit closure, which by Lemma \ref{lem:O(w)Description} is equal to some $\cO_\bw$. 
Let $J$ denote the set of lacing diagrams which index the irreducible components of $\tilde{\Omega}(0)$. We will show that $J = W(\zO)$.

Let $\bw \in W(\zO)$.  Since $\bw \in \zO = \tilde{\zO}(1)$, Lemma \ref{lem:onefiberall} implies that $\bw \in \tilde{\zO}(0)$, so its $\tGL(0)$-orbit closure $\cO_\bw$ is contained in $\tO(0)$.   We know that $\codim_\rep \cO_\bw =|\bw|$ for any $\bw$ by \cite[Prop.~15.30]{MS}.  
On the other hand, \eqref{eq:codimcrossings} and Lemma \ref{lem:degenfacts}(b) imply that this is equal to the codimension in $\rep$ of any irreducible component of $\tilde{\Omega}(0)$.
Therefore $\cO_\bw$ is an irreducible component of $\tilde{\Omega}(0)$, showing that $W(\zO)\subseteq J$.

Next suppose that $\bw\in J$.  Since $\zeta(\cO_\bw) \subseteq \zeta(\tilde{\zO}(0))$, we have that 
\begin{equation}\label{eq:initchain}
\init_\prec \zeta(\cO_\bw) \subseteq \init_\prec \zeta(\tilde{\zO}(0)) = \init_\prec Y_{v(\zO)},
\end{equation}
where the equality is the geometric version of \eqref{eq:keyequality}.  Observe that the dimensions of both sides of \eqref{eq:initchain} are equal, since $\cO_\bw$ is an irreducible component of $\tilde{\zO}(0)$ and $\zeta$ is a closed embedding, and families determined by monomial orders are flat \cite[Thm.~15.17]{eisenbud}.  Now take an irreducible component of $\init_\prec \zeta(\cO_\bw)$ of maximal dimension, which by Lemma \ref{lem:inPrec} is of the form $\Spec(K[z_{ij}]/I_{P_0})$ with $P_0$ a pipe dream satisfying $\bw(P_0) = \bw$.  By the dimension observation, this is also an irreducible component of $\init_\prec Y_{v(\zO)}$.  On the other hand, the geometric translation of \cite[Thm.~3.2]{wooyong} tells us that all irreducible components of $\init_\prec Y_{v(\zO)}$ are of the form $\Spec(K[z_{ij}]/I_P)$ for some $P \in \RedPipes(v_0, v(\zO))$.  Since $I_{P_0} = I_P$ if and only if $P_0 = P$, we find that $P_0 \in \RedPipes(v_0, v(\zO))$.  Proposition \ref{prop:pipetolacea} then implies that $\bw = \bw(P_0) \in W(\zO)$, completing the proof.
\end{proof}

We end by extracting a corollary of the above proof, which we need in the next subsection. 
It is a bipartite analogue of \cite[Cor.~6.18]{KMS}.

\begin{corollary}\label{prop:pipetolaceb}
If $\bw \in W(\zO)$, then there exists $P \in \RedPipes(v_0, v(\zO))$ such that $\bw(P) = \bw$. 
\end{corollary}

\begin{proof}
Let $\bw\in W(\zO)$ be given. Then $\cO_{\bw}$ is an irreducible component of $\tilde{\zO}(0)$ by Theorem \ref{thm:lacedegen}. The pipe dream $P_0$ from the last paragraph of the proof of Theorem \ref{thm:lacedegen} therefore satisfies $P_0\in \RedPipes(v_0, v(\zO))$ and $\bw(P_0) = \bw$ as desired.
%By Lemma \ref{lem:inPrec}, we can find some pipe dream $P$ on a $d_y\times d_x$ grid with contains $P_*$, such that $\bw(P) = \bw$ and such that $I_P$ is a minimal prime of $\init_{\prec} I(\zeta(\cO_{\bw}))$. But then $P\in \RedPipes(v_0, v(\zO))$ by the argument given in the last paragraph of the above proof.
\end{proof}

\subsection{Bipartite component formula}\label{sect:Kcomponent}
To each lacing diagram $\bw=(w_n, w^n, \dotsc, w_1, w^1)$ for a bipartite type $A$ quiver, we assign the following product of double Grothendieck polynomials as a shorthand:
\begin{equation}\label{eq:gwdef}
\fG_\bw(\bt; \bs) =\left(\prod_{k=1}^n \fG_{w_k}(\bt^{k}; \bs^{k}) \right) \cdot \left(\prod_{k=1}^n  \fG_{\rot(w^k)}(\widetilde{\bt^{k-1}}; \widetilde{\bs^{k}}) \right).
\end{equation}
By Lemma \ref{lem:SEmatrixschubert}, and \eqref{eq:Kpartialperm} this is equal to $\cK_{\mathtt{rep}}(\cO_\bw; \bs, \bt)$.

Recall that the \emph{M\"obius function} of a finite poset $T$ is the (unique) function $\mu \colon T \to \mathbb{Z}$ such that $\sum_{x \geq y} \mu_T(x) = 1$ for all $y \in T$.  It is typically used in the following way:  for $x \in T$, we define ``open'' and ``closed'' characteristic functions
\begin{equation}
\chi^\circ_x (y) =
\begin{cases}
1 & y = x\\
0 & y \neq x,\\
\end{cases}
\qquad 
\chi_x(y) =
\begin{cases}
1 & y \leq x\\
0 & y \not \leq x.\\
\end{cases}
\end{equation}
By inclusion-exclusion, we can write the constant function taking value 1 on $T$ as
\begin{equation}\label{eq:mobiusdef}
1_T = \sum_x \chi^\circ_x = \sum_x \mu_T(x) \chi_x .
\end{equation}

\begin{theorem}[Bipartite component formula]\label{thm:bicomponent}
For any bipartite type $A$ quiver locus $\Omega$, we have
\begin{equation}\label{eq:bicomponentformula}
K\cQ_\Omega (\bt / \bs)= \sum_{\bw \in KW(\Omega)} (-1)^{|\bw| - \codim(\Omega)} \fG_\bw(\bt; \bs) .
\end{equation}
\end{theorem}
\begin{proof}
By Lemma \ref{lem:degenfacts}(a), we have $K\cQ_\Omega (\bt / \bs)= \cK_{\mathtt{rep}}(\tilde{\zO}(0); \bs, \bt)$.
Define the subset of lacing diagrams $\cM := \setst{\bw}{ \cO_\bw \subseteq \tilde{\zO}(0)}$,
and partially order it by $\bw \leq \bw'$ if and only if $\cO_\bw \subseteq \cO_{\bw'}$.
So the maximal elements of $\cM$ correspond to the irreducible components of $\tilde{\zO}(0)$.
The collection of closed subvarieties $\{\cO_\bw \subseteq \tilde{\zO}(0)\}$
has the intersect-decompose property of \cite{Knutson:2009aa} because the intersection of any collection of matrix Schubert varieties is a union of matrix Schubert varieties.
The hypotheses of \cite[Thm.~1]{Knutson:2009aa} are satisfied by this collection because it can be simultaneously compatibly Frobenius split \cite[\S7.2]{Knutson:2009bb}, so we get
\begin{equation}\label{eq:kclassOr}
\cK_{\mathtt{rep}}(\tilde{\zO}(0); \bs, \bt) = \sum_{\bw \in \cM} \mu_{\cM}(\bw) \cK_{\mathtt{rep}}(\cO_\bw; \bs, \bt).
\end{equation}

To compute this M\"obius function, we pass to the pipe complex of \cite[\S3]{wooyong}.  Let $\zD:= \zD_{v_0, v(\Omega)}$ be the set of pipe dreams
on a $d_y \times d_x$ rectangle which contain a reduced pipe dream for $v(\Omega)$.  The partial order on $\zD$ is determined by reverse containment of pipe dreams.  Considering $\zD$ as a simplicial complex, its maximal faces are the elements of $\RedPipes(v_0, v(\Omega))$.
Since every element of $\zD$ contains a pipe dream for $v(\Omega)$, every element contains $P_*$ by Lemma \ref{lem:homsubdiag}, and the operation $\pi$ of \S\ref{sect:pipeToLace} gives a map $\pi \colon \zD \to S_\bd$, which is order preserving when we take reverse (strong) Bruhat order on each factor of the target.  
Recall from \S\ref{sect:pipeToLace} that we consider the set of lacing diagrams of dimension vector $\bd$ as a subset of $S_\bd$. This identification respects the partial orders, so $\cM$ is naturally a subposet of $S_\bd$.
We now describe the image $\pi(\zD)=:S_\zO \subseteq S_\bd$.
The maximal elements of $\zD$ are $P \in \RedPipes(v_0, v(\zO))$, and $\pi(P)=\bw(P)$ is a minimal lacing diagram by Lemma \ref{lem:followPipes} and Proposition \ref{prop:pipetolacea}.  
On the other hand, every minimal lacing diagram for $v(\Omega)$ is in the image of $\pi$ by Corollary \ref{prop:pipetolaceb}, so by Theorem \ref{thm:lacedegen} the maximal elements of $S_\zO$ are exactly the maximal elements of $\cM$.
Clearly $\zD$ has a unique minimal element $P_0$ where the entire $d_y \times d_x$ grid is filled with cross tiles, and $\pi(P_0)$ is the lacing diagram consisting of partial permutation matrices filled with 0s, the unique minimal element of $\cM$.  This shows that $S_{\zO}$ is the intersection of the down-ideal of $S_\bd$ generated by all minimal lacing diagrams with the up-ideal generated by $\pi(P_0)$, and all of $\cM$ is contained in $S_\zO$.

It is straightforward to check that the containment of posets $\cM \subseteq S_\zO$ satisfies the hypotheses of \cite[Lem.~2]{Knutson:2009aa} (keeping in mind that we use \emph{reverse} Bruhat order), so this implies that  $\mu_{S_\zO}(\bv)= \mu_{\cM}(\bv)$ for $\bv \in \cM$ and $\mu_{S_\zO}(\bv)=0$ for $\bv \notin \cM$. 
Taken together, the reductions above yield the following equation, showing that it is enough to compute $\mu_{S_\zO}$.
\begin{equation}\label{eq:simplifiedcomponent}
K\cQ_\Omega (\bt / \bs)= \sum_{\bw \in \cM} \mu_{S_\zO}(\bw) \fG_\bw(\bt; \bs)
\end{equation}

The map $\pi$ induces a map on rings of $\ZZ$-valued functions 
$\pi^* \colon \ZZ[S_\zO] \to \ZZ[\zD].$ 
By \cite[Prop.~3.3]{wooyong} via \cite[Thm.~3.7]{MR2047852}, we see $\zD$ is homeomorphic to a ball (since there is a unique pipe dream for $v_0$), and its M\"obius function is
\begin{equation}\label{eq:mobiusdelta}
\mu_\zD (P) = 
\begin{cases}
(-1)^{\codim_\zD P} & \text{if}\ \delta(P)=v(\Omega)\\
0 & \text{otherwise}.\\
\end{cases}
\end{equation}
Applying \eqref{eq:mobiusdef} in $\ZZ[\zD]$ then gives
\begin{equation}\label{eq:iddelta1}
1_\zD = \sum_{P \in \zD} \chi^\circ_P = \sum_{P \in \Pipes(v_0, v(\Omega))} (-1)^{\codim_\zD P} \chi_P.
\end{equation}
Similarly, in $\ZZ[S_\zO]$ we have the expression
\begin{equation}
1_{S_\zO} = \sum_{\bv \in S_\zO} \chi^\circ_\bv = \sum_{\bv \in S_\zO} \mu_{S_\zO}(\bv) \chi_\bv
\end{equation}
which gives another way of writing
\begin{equation}\label{eq:iddelta2}
1_\zD = \pi^*(1_{S_\zO}) = \sum_{\bv \in S_\zO} \mu_{S_\zO}(\bv) \pi^*(\chi_\bv) .
\end{equation}
To compare this with \eqref{eq:iddelta1}, we need to express $\pi^*(\chi_\bv)$ in the basis $\{\chi_P \}$.  Setting  $\zD(\bv) := \setst{P \in \zD}{\pi(P) \leq \bv}$, by the definitions then M\"obius inversion we have that 
\begin{equation}\label{eq:pichiw}
\pi^*(\chi_\bv) =  1_{\zD(\bv)} = \sum_{\substack{P \in \zD \\ \pi(P) \leq \bv}} \chi^\circ_P =  \sum_{\substack{P \in \zD \\ \pi(P) \leq \bv}} \mu_{\zD(\bv)} (P) \chi_P .
\end{equation}
To compute $\mu_{\zD(\bv)} (P)$, write $\bv = (v_n, v^n, \dotsc, v_1, v^1)$ as in \S\ref{sect:pipeToLace} and notice  $\zD(\bv)$ is a product of pipe complexes
\begin{equation}
\zD(\bv) = \prod_{i=1}^n \zD_{\za_i}(v^i) \times \zD_{\zb_i}(v_i)
\end{equation}
where $\zD_{\za_i}(v^i)$ is the complex of pipe dreams on a $\bd(y_{i-1}) \times \bd(x_i)$ grid which contain a pipe dream for $v^i$, and similarly for $\zD_{\zb_i}(v_i)$.  Thus, $\zD(\bv)$ is also homeomorphic to a ball and its M\"obius function can be computed in each factor as in \eqref{eq:mobiusdelta}. We get
\begin{equation}\label{eq:muDeltaw}
\mu_{\zD(\bv)} (P) = 
\begin{cases}
(-1)^{\codim_{\zD(\bv)} P} & \pi(P) = \bv\\
0 & \pi(P) \neq \bv.\\
\end{cases}
\end{equation}

Substituting \eqref{eq:pichiw} and \eqref{eq:muDeltaw} into \eqref{eq:iddelta2} we get
\begin{equation}\label{eq:iddelta3}
1_{\zD} = \sum_{\bv \in S_\zO} \mu_{S_\zO}(\bv) \left(\sum_{\substack{P \in \zD \\ \pi(P) = \bv}} (-1)^{\codim_{\zD(\bv)} P} \chi_P \right) = \sum_{P \in \zD} \mu_{S_\zO}(\pi(P)) (-1)^{\codim_{\zD(\pi(P))} P} \chi_P .
\end{equation}
Comparing the coefficient of $\chi_P$ in \eqref{eq:iddelta1} and \eqref{eq:iddelta3} we find 
\begin{equation}\label{eq:solvemu}
\mu_{S_\zO}(\pi(P)) = 
\begin{cases}
(-1)^{\codim_\zD P - \codim_{\zD(\pi(P))} P} & \text{if}\ P \in \Pipes(v_0, v(\Omega))\\
0 & \text{otherwise.}
\end{cases}
\end{equation}
Since $\pi(\zD)=S_\zO$, this computes $\mu_{S_\zO}(\bv)$ for all $\bv \in S_\zO$.
In particular, the sum \eqref{eq:simplifiedcomponent} simplifies to a sum over $\bw$ such that there exist $P \in \Pipes(v_0, v(\Omega))$ with $\pi(P) = \bw$ under the identification of $\cM$ with a subposet of $S_\zO$.  By results of Buch, Feh\'er, and Rim\'anyi (\cite[Lem.~6.2]{MR2114821} and \cite[Thm.~3]{BFR}), this is equivalent to summing over all $K$-theoretic lacing diagrams for $\Omega^{\circ}$.\footnote{Their work is in the equioriented case but readily generalizes to arbitrary orientation.}

It remains to simplify the exponent of $-1$ in \eqref{eq:solvemu}.  We have $\codim_\zD P = |P| - \ell(v(\Omega))$ by definition, and applying the analogue of this to each factor of $\zD(\bv)$ gives $\codim_{\zD(\bv)} P = |P\setminus P_*| - |\bv|$ for $\bv=\pi(P)$.
Therefore, we can simplify $\codim_\zD P - \codim_{\zD(\pi(P))} P$ to
\begin{equation}\label{eq:codimcompute}
(|P| - \ell(v(\zO))) - (|P \setminus P_*| - |\bv|) = |\bv| +|P_*|  - \ell(v(\zO)) = |\bv| - \codim \Omega,
\end{equation}
where the last equality uses $|P_*| = \ell(v_*)$ and \eqref{eq:Zpermlength}.  This completes the proof.
\end{proof}

\section{Generalizing the formulas to arbitrary orientation}\label{sect:arbitrary}
\subsection{Reduction to the bipartite setting}
In this section, we recall the connection between type $A$ quiver loci in the arbitrarily oriented and bipartite settings. We use this to prove the main formulas of the paper via straightforward substitutions into our bipartite formulas.  Since we work with more than one quiver and dimension vector, these are reinstated to the notation.

Our notation for a type $A$ quiver of arbitrary orientation $Q$ will be $Q_0=\{1, \dotsc, m\}$ and $Q_1 = \{\zg_1, \dotsc, \zg_{m-1}\}$, with the arrow $\zg_i$ connecting vertices $i$ and $i+1$.  A vertex $i$ is to the \emph{left} of a vertex $j$ if $i > j$, and similarly with arrows.
Given such a quiver, we obtain an associated bipartite quiver $\tilde{Q}$ by inserting a ``backwards'' arrow  in the middle of each path of length 2.
Rigorously, for each length two path in $Q$ we double the vertex in the middle and create a new arrow by defining sets of \emph{added vertices and arrows}
\[
Q'_0 = \setst{i'}{h\zg_i = t\zg_{i-1} \text{ or } t\zg_i = h\zg_{i-1}} \quad Q'_1 = \setst{\zg'_i}{h\zg_i = t\zg_{i-1} \text{ or } t\zg_i = h\zg_{i-1}}.
\]
We set $\tq_0 = Q_0 \coprod Q'_0$ and $\tq_1 = Q_1\coprod Q'_1$. %Elements of $Q_0, Q_1$ will be called \emph{original vertices and arrows} of $\tq$, respectively.  
Added arrows will sometimes be called \emph{inverted arrows}, and anything indexed by them (e.g., blocks of the snake region) can be referred to as \emph{inverted}.
The tail and head functions for $\tq$ are illustrated by the diagrams:
\[
i+1 \xrightarrow{\zg_i} i \xrightarrow{\zg_{i-1}}i-1 \text{ in $Q$ yields }\qquad
\vcenter{\hbox{\begin{tikzpicture}[xscale=1.3,point/.style={shape=circle,fill=black,scale=.5pt,outer sep=3pt},>=latex]
   \node[outer sep=-2pt] (1) at (-1,0.5) {$i+1$};
   \node[outer sep=-2pt] (2) at (0,0) {$\color{red}{i'}$};
   \node[outer sep=-2pt] (3) at (1,0.5) {$i$};
  \node[outer sep=-2pt] (4) at (2,0) {$i-1$};
  \path[->]
  	(1) edge node[above,pos=0.6] {$\zg_i$} (2) 
  	(3) edge[thick,dashed,red] node[above,pos=0.5] {$\zg'_i$} (2) 
	(3) edge node[above,pos=0.6] {$\zg_{i-1}$} (4);
 \end{tikzpicture}}}  \text{ in }\tq
\]
\[
i+1 \xleftarrow{\zg_i} i \xleftarrow{\zg_{i-1}}i-1 \text{ in $Q$ yields }\qquad
\vcenter{\hbox{\begin{tikzpicture}[xscale=1.3,point/.style={shape=circle,fill=black,scale=.5pt,outer sep=3pt},>=latex]
   \node[outer sep=-2pt] (1) at (-1,0) {$i+1$};
   \node[outer sep=-2pt] (2) at (0,0.5) {$\color{red}{i'}$};
   \node[outer sep=-2pt] (3) at (1,0) {$i$};
  \node[outer sep=-2pt] (4) at (2,0.5) {$i-1$};
  \path[->]
  	(2) edge node[above,pos=0.5] {$\zg_i$} (1) 
  	(2) edge[thick,dashed,red] node[above,pos=0.5] {$\zg'_i$} (3) 
	(4) edge node[above,pos=0.5] {$\zg_{i-1}$} (3);
 \end{tikzpicture}}}  \text{ in }\tq.
\]
We define $\nu\colon \tq_0 \to Q_0$ by $\nu(i') = \nu(i) = i$ for $i \in Q_0$ and $i' \in Q'_0$.

For example, a type $A$ quiver $Q$ and its associated bipartite quiver $\tilde{Q}$ are below.
\begin{equation}\label{eq:arbitraryq}
Q=\qquad\vcenter{\hbox{\begin{tikzpicture}[yscale=0.6,point/.style={shape=circle,fill=black,scale=.5pt,outer sep=3pt},>=latex]
   \node[outer sep=-2pt] (z5) at (-1,1) {$6$};
   \node[outer sep=-2pt] (z0) at (0,2) {$5$};
   \node[outer sep=-2pt] (z1) at (1,1) {$4$};
  \node[outer sep=-2pt] (z2) at (2,0) {$3$};
   \node[outer sep=-2pt] (z3) at (3,1) {$2$};
  \node[outer sep=-2pt] (z4) at (4,2) {$1$};
  \path[->]
  	(z1) edge node[below,pos=0.3] {$\zg_3$} (z2) 
	(z3) edge node[below,pos=0.3] {$\zg_2$} (z2)
  	(z0) edge node[below,pos=0.3] {$\zg_4$} (z1) 
	(z4) edge node[below,pos=0.3] {$\zg_1$} (z3)
	(z0) edge node[below,pos=0.3] {$\zg_5$} (z5);
   \end{tikzpicture}}}
\end{equation}

\begin{equation}\label{eq:qcover}
\tq=\vcenter{\hbox{\begin{tikzpicture}[xscale=1.2,point/.style={shape=circle,fill=black,scale=.5pt,outer sep=3pt},>=latex]
   \node[outer sep=-2pt] (z5) at (-2,0) {$6$};
   \node[outer sep=-2pt] (z0) at (-1,1) {$5$};
   \node[outer sep=-2pt] (w1) at (0,0) {$\color{red}{4'}$};
   \node[outer sep=-2pt] (z1) at (1,1) {$4$};
  \node[outer sep=-2pt] (z2) at (2,0) {$3$};
   \node[outer sep=-2pt] (w3) at (3,1) {$\color{red}{2'}$};
  \node[outer sep=-2pt] (z3) at (4,0) {$2$};
  \node[outer sep=-2pt] (z4) at (5,1) {$1$};
  \path[->]
	(z0) edge node[left,pos=0.2] {$\zg_5$} (z5)
  	(z0) edge node[left,pos=0.7] {$\zg_4$} (w1) 
  	(z1) edge[thick,dashed,red] node[above,pos=0.6] {$\zg'_4$} (w1) 
	(z1) edge node[left,pos=0.7] {$\zg_3$} (z2)
  	(w3) edge node[left,pos=0.2] {$\zg_2$} (z2) 
	(w3) edge[thick, dashed,red] node[left,pos=0.7] {$\zg'_2$} (z3)
	(z4) edge node[left,pos=0.2] {$\zg_1$} (z3);
   \end{tikzpicture}}}\qedhere
\end{equation}

Given a dimension vector $\bd$ for $Q$, define a dimension vector $\td$ for $\tilde{Q}$ as the natural lifting $\td(z) := \bd(\nu(z))$ for $z \in \tq_0$, noting that $\td(ta) = \td(ha)$ when $a\in Q'_1$ is an added arrow. 
Consider the closed embedding $\sigma:\rep_Q(\bd)\rightarrow \rep_{\tilde{Q}}(\tilde{\bd})$ defined by
\begin{equation}
\sigma\left((V_{\gamma_k})_{1\leq k\leq m-1}\right) = (W_a)_{a \in \tq_1}, \quad \text{where} \quad
W_a = 
\begin{cases}
V_{\zg_k} & a=\gamma_k \in Q_1\\
\mathbf{1}_{\td(ta)} & a \in Q'_1.
\end{cases}
\end{equation}
As usual, $\mathbf{1}_k$ denotes a $k\times k$ identity matrix. 
Note that $\sigma$ is equivariant with respect to the natural inclusion $\GL(\bd) \into \GL(\tilde{\bd})$ defined by
\begin{equation}\label{eq:equivariance}
(g_i)_{i \in Q_0}\mapsto (h_z)_{z \in \tq_0}, \qquad
h_z = g_{\nu(z)}.
\end{equation} 
If $\Omega\subseteq \rep_Q(\bd)$ is an orbit closure, let $\tilde{\Omega}\subseteq \rep_{\tilde{Q}}(\tilde{\bd})$ be the orbit closure $\overline{\GL(\tilde{\bd})\cdot \sigma(\Omega)}$\footnote{Throughout this section, $\tilde{\Omega}$ denotes a \emph{single, lifted} orbit in the \emph{single, lifted} space $\widetilde{\rep}$, not a family of orbits in a family of spaces as in \S \ref{sect:degenAndComponent}.}.  
By \cite[Thm.~5.3]{KR}, the ring map $\sigma^{\sharp}: K[\rep_{\tilde{Q}}(\tilde{\bd})]\rightarrow K[\rep_Q(\bd)]$ induces an isomorphism of the coordinate rings $K[\zO] \simeq K[\tO]/\text{ker}(\sigma^{\sharp})$, and $\codim \zO = \codim \tilde{\zO}$. Note that the codimension on the left side (resp. right side) of the equality is in $\rep_Q(\bd)$ (resp. $\rep_{\tilde{Q}}(\tilde{\bd})$).

We now establish notation for the multigrading of $K[\mathtt{rep}_Q(\bd)]$ defined in \S\ref{sect:multigrading}.  
For $i \in Q_0$, define an alphabet $\bu^i = u^i_1, \dotsc, u^i_{\bd(i)}$, and denote the concatenation of these alphabets by $\bu = \bu^1, \dotsc, \bu^m$.  For left pointing $k+1 \xleftarrow{\zg_k} k$, the coordinate function picking out the $(i,j)$-entry of the matrix $V_{\zg_k}$ is given degree $u_i^k - u_j^{k+1}$, and for right pointing $k+1 \xrightarrow{\zg_k} k$, the coordinate function picking out the $(i,j)$-entry of the matrix $V_{\zg_k}$ is given degree $u_i^{k+1} - u_j^k$.

We fix an identification of $\tq$ with a bipartite quiver whose vertices and arrows are labeled as in the previous sections of the paper.
Let $\tilde{d}$ be the total dimension of $\tilde{\bd}$ and $d$ the total dimension of $\bd$.  As before, we can identify the group which grades $K[\mathtt{rep}_{\tilde{Q}}(\tilde{\bd})]$ with the free abelian group on the concatenated alphabet $\bs, \bt$, and this group is isomorphic to $\ZZ^{\tilde{d}}$.  Similarly, the group which grades $K[\mathtt{rep}_Q(\bd)]$ is identified with the free abelian group on $\bu$, and is isomorphic to $\ZZ^d$.
Our identification of $\tq$ with a bipartite quiver having vertices labeled as in previous sections allows us to use the notation $\nu(x_i)$ and $\nu(y_i)$ to specify vertices of $Q$, and define a quotient morphism of abelian groups $sub: \mathbb{Z}^{\tilde{d}}\rightarrow \mathbb{Z}^{d}$ by 
\begin{equation}\label{eq:subdef}
s^i_j \mapsto u^{\nu(x_i)}_j, \qquad t^i_j \mapsto u^{\nu(y_i)}_j.
\end{equation}
The map $sub$ defines a $\mathbb{Z}^d$-grading of $K[\mathtt{rep}_{\tilde{Q}}(\tilde{\bd})]$, which is a coarsening of our usual $\mathbb{Z}^{\tilde{d}}$-grading.
In Figure \ref{fig:zsimage} we have labeled the block rows and columns of the image of $\zeta\circ \sigma$ to show how $sub$ interacts with the Zelevinsky map.
\begin{figure}
\[\zeta(\sigma(V)) =
\vcenter{\hbox{
\begin{tikzpicture}[every node/.style={minimum width=1em}]
\matrix (m0) [matrix of math nodes,left delimiter  = {[},%
             right delimiter = {]}, nodes in empty cells] at (0,0)
{
0 &0& 0 & V_{\zg_1}& \bid_{\bd(1)} &  & \\
0 & 0 &  V_{\zg_2} & \bid_{\bd(2)} & & \bid_{\bd(2)} & \\
%\phantom{X} & \\
0 & \bid_{\bd(4)} & V_{\zg_3} & 0 & & & \bid_{\bd(4)}\\
V_{\zg_5} & V_{\zg_4} & 0 & 0 & & & & \bid_{\bd(5)} \\
\bid_{\bd(6)} &  & \\
& \bid_{\bd(4)} &  & \\
&& \bid_{\bd(3)} &  & \\
&& &\bid_{\bd(2)} \\
};
%top labels
\node  at (0.5,3) {$\bu^1$};
\node  at (1.7,3) {$\bu^2$};
\node  at (2.9,3) {$\bu^4$};
\node  at (4.1,3) {$\bu^5$};
\node  at (-0.6,3) {$\bu^2$};
\node  at (-1.6,3) {$\bu^3$};
\node  at (-2.6,3) {$\bu^4$};
\node  at (-3.6,3) {$\bu^6$};
%side labels
\node  at (-5,2.4) {$\bu^1$};
\node  at (-5,1.8) {$\bu^2$};
\node  at (-5,1.1) {$\bu^4$};
\node  at (-5,0.4) {$\bu^5$};
\node  at (-5,-0.3) {$\bu^6$};
\node  at (-5,-1) {$\bu^4$};
\node  at (-5,-1.6) {$\bu^3$};
\node  at (-5,-2.3) {$\bu^2$};
\node[scale=3] (zero) at (m0-6-3 -| m0-2-6.south east)  {$0$};
\draw[thick] (m0-5-1.north west) -- (m0-5-1.north west -| m0-4-8.south east);%across
\draw[thick] (m0-1-5.north west) -- (m0-1-5.north west |- m0-8-3.south east);%down
\end{tikzpicture}}}
\]
\caption{The image of $\zeta \circ \sigma$ for the example \eqref{eq:arbitraryq}.}\label{fig:zsimage}
\end{figure}

The above definitions are made so that the %restriction 
map on coordinate rings
\begin{equation}\label{eq:sigmasharp}
\sigma^\# \colon K[\mathtt{rep}_{\tilde{Q}}(\tilde{\bd})] \to K[\mathtt{rep}_Q(\bd)]
\end{equation}
is compatible with $sub$, that is, the degree of $\sigma^\#(f)$ is $sub$ of the degree of $f$ for any polynomial $f$ which is homogeneous with respect to our $\mathbb{Z}^{\tilde{d}}$-grading of $K[\mathtt{rep}_{\tilde{Q}}(\tilde{\bd})]$. This is a consequence of the $\GL(\bd)$-equivariance of the map $\sigma$ (see \eqref{eq:equivariance}).

The map $sub$ of abelian groups induces a map from the ring of Laurent polynomials in the concatenated alphabet $\bt, \bs$ to the ring of Laurent polynomials in the alphabet $\bu$. %We also denote this map by $sub$:
%\begin{equation}\label{eq:subdef}
%sub: \mathbb{Z}[\bt^{\pm 1}, \bs^{\pm 1}]\rightarrow \mathbb{Z}[\bu^{\pm 1}],\qquad s^i_j \mapsto u^{\nu(x_i)}_j, \qquad t^i_j \mapsto u^{\nu(y_i)}_j.
%\end{equation}

Finally, we write $K\cQ_\zO  (\bu)$ (resp., $\cQ_\zO  (\bu)$) for the $K$-polynomial (resp., multidegree) of $\zO$ with respect to its inclusion in $\rep$ and the multigrading described above.

\begin{proposition}\label{prop:alphabetsub}
For any quiver locus $\Omega\subseteq \rep_Q(\bd)$, we have $K\cQ_\zO  (\bu)= sub (K\cQ_{\tilde{\zO}}(\bt / \bs) )$.
\end{proposition}

\begin{proof}
Taking a $\mathbb{Z}^{\tilde{d}}$-graded free resolution of $K[\tilde{\Omega}]$ over $K[\mathtt{rep}_{\tilde{Q}}(\tilde{\bd})]$ and reducing modulo $\ker (\sigma^{\#})$ yields a $\mathbb{Z}^d$-graded free resolution of $K[\Omega]$ over $K[\mathtt{rep}_Q(\bd)]$.
The proof of this fact is the same as the equioriented case \cite[Prop.~17.31]{MS}, essentially relying on the fact that type $A$ orbit closures are Cohen-Macaulay \cite{MR1967381,KR}.   Since the $K$-polynomial of $\zO$ can be computed as an alternating sum of the $K$-polynomials of the terms in a finite multigraded free resolution \cite[Def.~8.32]{MS}, this reduces the proposition to the case of free modules, which is exactly the statement that \eqref{eq:sigmasharp} is compatible with $sub$.
\end{proof}

\comment{
Let $\mathcal{F}_\bullet $ be a $\mathbb{Z}^{\tilde{d}}$-graded free resolution of $K[\tilde{\Omega}]$ over $K[\mathtt{rep}_{\tilde{Q}}(\tilde{\bd})]$. Then $\mathcal{F}_\bullet/\ker (\sigma^{\sharp})\mathcal{F}_{\bullet}$ is a complex of $\mathbb{Z}^d$-graded free modules over $K[\mathtt{rep}_{\tilde{Q}}(\tilde{\bd})]/\ker(\sigma^{\sharp})$. 
Indeed, the $\mathbb{Z}^d$-grading of $K[\mathtt{rep}_{\tilde{Q}}(\tilde{\bd})]$, defined via the map $sub$, is a coarsening of the $\mathbb{Z}^{\tilde{d}}$-grading, and this coarsening makes the ideal $\ker(\sigma^\sharp)$ homogeneous since the variables set equal to $1$ have degree $0$.
Furthermore, $\mathcal{F}_\bullet/\ker (\sigma^{\sharp})\mathcal{F}_{\bullet}$ is actually a $\mathbb{Z}^d$-graded free resolution of $K[\tilde{\Omega}]/\ker(\sigma^\sharp)$. The proof of this is the same as in the equioriented case \cite[Prop.~17.31]{MS}, essentially relying on the fact that type $A$ quiver loci are Cohen-Macaulay \cite{MR1967381} (or \cite{KR}). 
By identifying $K[\rep_Q(\bd)]$ with $K[\mathtt{rep}_{\tilde{Q}}(\tilde{\bd})]/\ker(\sigma^{\sharp})$ via $\sigma^\sharp$, we see that $\mathcal{F}_\bullet/\ker (\sigma^{\sharp})\mathcal{F}_{\bullet}$ is a $\mathbb{Z}^d$-graded free resolution of $K[\Omega]$ over $K[\rep_Q(\bd)]$.
}
%Next, recall that a $K$-polynomial of a multigraded module can be computed as an alternating sum of the $K$-polynomials of the terms in a (finite) multigraded free resolution \cite[Def.~8.32]{MS}. This together with the above paragraph completes the proof.

\begin{remark}
The reader familiar with equivariant $K$-theory may observe that Proposition \ref{prop:alphabetsub} is an algebraic way of stating that the equivariant $K$-class of $\Omega$ is obtained by pulling back the equivariant $K$-class of $\tilde{\Omega}$ along the $\GL(\bd)$-equivariant map $\sigma$.
\end{remark}

\subsection{Formulas in arbitrary orientation}
We now generalize our formulas to arbitrary orientation.  
We only prove the $K$-polynomial versions, the multidegree formulas following from this as usual.

\begin{theorem}[Ratio formula]\label{thm:ratioformula}
Let $Q$ be a type $A$ quiver of arbitrary orientation, $\Omega \subseteq \mathtt{rep}_Q(\bd)$ a quiver locus, and $\tilde{\Omega}\subseteq \rep_{\tilde{Q}}(\tilde{\bd})$ the associated bipartite quiver locus. Then:
\begin{equation}\label{eq:ratioformula}
K\cQ_\zO(\bu) =  \frac{sub(\fG_{v(\tilde{\Omega})}(\bt,\bs; \bs,\bt)) }{sub(\fG_{v_*}(\bt,\bs; \bs,\bt)) } \quad \text{and} \quad
\cQ_\zO(\bu) = \frac{sub(\fS_{v(\tilde{\Omega})}(\bt,\bs; \bs,\bt)) }{sub(\fS_{v_*}(\bt,\bs; \bs,\bt)) } .
\end{equation}
%and the same formula holds for multidegree upon replacing Grothendieck polynomials with Schubert polynomials.
\end{theorem}
\begin{proof}
Applying Proposition \ref{prop:alphabetsub} to %Theorem \ref{thm:biratio}
 \eqref{eq:bipartiteratio} for $\tilde{\zO}$, the only thing we need to check is that the denominator is nonzero.  But $sub(\fG_{v_*}(\bt,\bs; \bs,\bt))=sub\left((1-\bt/\bs)^{P_*}\right)$ as in the proof of Theorem \ref{thm:bipipe}, which could only be zero if $P_*$ has a cross tile in a position with row label $t^i_j$ and column label $s^k_l$ and $sub(t^i_j) = sub(s^k_l)$.  By \eqref{eq:subdef}, the latter could occur only when $\nu(x_i)=\nu(y_i)$,  which would require $P_*$ to have a cross tile in an inverted block, so in particular a block in the snake region.  But $P_*$ has no cross tiles in the snake region.
\end{proof}

Let $\tilde{d}_x =\sum_i \tilde{\bd}(x_i)$ and $\tilde{d}_y =  \sum_i \tilde{\bd}(y_i)$.  Identifying a $\tilde{d}_y \times \tilde{d}_x$ grid with the northwest quadrant of the target of $\zeta\circ \sigma$, we get a labeling of rows and columns of the grid by elements of $\bu$ as in Figure \ref{fig:zsimage}.
To be precise, let $\bu^{\row} = (u^{\row}_i)_{1\leq i\leq d_{y}}$ (resp. $\bu^{\text{col}} = (u^{\text{col}}_i)_{1\leq i\leq d_x}$) denote the list obtained by applying $sub$ to the entries of the alphabet $\bt$ (resp. $\bs$). 
Label the rows of the $\tilde{d}_y \times \tilde{d}_x$ grid by $\bu^{\text{row}}$ and columns by $\bu^{\text{col}}$, so that each pipe dream $P$ on this grid determines a term $(\one - \bu^{\text{row}} / \bu^{\text{col}})^P$ as in \S\ref{sect:pipeDreams}, which is clearly equal to $sub\left( (\one - \bt/\bs)^P \right)$.
%Then $z_{ij}$, the $(i,j)$-entry of the universal matrix over $\mcell$, has $\mathbb{Z}^d$-degree $u^{\text{row}}_i-u^{\text{col}}_j$.
Define a partition $\Pipes(v_0,v(\tilde{\zO})) = G \cup B$ into ``good'' and ``bad'' pipe dreams where $G$ consists of pipe dreams $P$ such that all $(i,j) \in P \setminus P_*$ are in blocks of the snake region corresponding to arrows of $Q$, and $B$ consists of those with at least one cross tile in an inverted block.  

\begin{theorem}[Pipe formula]\label{thm:pipeformula}
For a type $A$ quiver $Q$ of arbitrary orientation, the $K$-polynomial of a quiver locus $\zO \subseteq \mathtt{rep}_Q(\bd)$ is given by
\begin{equation}\label{eq:pipeformula}
K\cQ_\zO(\bu) = \sum_{P \in G} (-1)^{|P\backslash P_*|-\codim \zO}\ (\one - \bu^{\row} / \bu^{\col})^{P \setminus P_*}.
\end{equation}
The formula for $\cQ_\zO(\bu)$ is obtained by replacing $(\one - \bu^{\row} / \bu^{\col})$ with $(\bu^{\row} - \bu^{\col})$, with the sum taken over reduced pipe dreams in $G$.
\end{theorem}
\begin{proof}
%We only prove the $K$-polynomial statement, the multidegree statement following from this as usual.  
Applying Proposition \ref{prop:alphabetsub} to \eqref{eq:Kbipipe} for $\tilde{\zO}$, we need to both simplify the exponent of $-1$ in the formula, and also show that the sum of terms indexed by $B$ is zero after applying $sub$.
The simplification of the exponent follows from the equality $\codim \widetilde{\zO}=\codim \zO$. 
%Since $\ell(v(\widetilde{\zO})) = \codim \widetilde{\zO} + \ell(v_*)$ from \eqref{eq:Zpermlength} and $\codim \widetilde{\zO}=\codim \zO$, we get $|P|-\ell(v(\widetilde{\zO})) = |P\backslash P_*| - \codim \zO$.

Fix consideration of a single inverted block, say corresponding to the arrow $\beta_r$ of $\tq$ (the case of $\za_r$ is similar).  Define an equivalence relation $\sim$ on $B$ by declaring $P_1 \sim P_2$ exactly when $P_1$ and $P_2$ have the same elements $(i,j)$ outside of the $\beta_r$ block.  This gives a partition $B =\coprod_k B_k$, where each $B_k$ is an equivalence class.  Further partition each $B_k$ into subsets $B_k(w)$ indexed by permutations $w\neq id$, such that $P \in B_k(w)$ if and only if the Demazure product of the restriction of $P$ to the designated block is $w$.  Then by \cite[Thm.~2.1]{BRspec} the sum of terms in \eqref{eq:Kbipipe} over $P\in B_k(w)$ has the factor $\fG_w(\bt^r; \bs^r)$, where $\bt^r, \bs^r$ are the subsets of $\bt, \bs$ as in \S \ref{sect:multigrading}.  Applying $sub$ to this and noting that $\nu(y_r)=\nu(x_r)$ since $\zb_r$ is inverted yields
\begin{equation}\label{eq:BRcor}
\fG_w(\bu^{\nu(y_r)}; \bu^{\nu(x_r)}) = 
\begin{cases}
1 & \text{if}\ w=id,\\
0 & \text{if}\ w \neq id.
\end{cases}
\end{equation}
by \cite[Cor.~2.4]{BRspec}. Repeating over all inverted blocks completes the proof.
\end{proof}

\begin{remark}
The last statement of the theorem gives a simplification of existing formulas in the literature for the equioriented case.  Namely, the sums of \cite[``Pipe formula'' p. 236]{KMS} for multidegrees and \cite[Thm.~3]{MR2137947} for $K$-polynomials can be taken over fewer pipe dreams (those with no cross tile on the block antidiagonal, in the setup of those papers).
\end{remark}

We finally arrive at the component formula.  To each lacing diagram $\bw=(w_i)_i$ for $Q$, where $w_i$ is associated to $\zg_i \in Q_1$, we define %the product of double Grothendieck polynomials
\begin{equation}
\fG_\bw(\bu)=\left(\prod_{\xrightarrow{\zg_i}}\fG_{w_i}(\bu^{i+1};\bu^{i})\right)\left(\prod_{\xleftarrow{\zg_i}}\fG_{\rot(w_i)}(\widetilde{\bu^{i}};\widetilde{\bu^{i+1}})\right)
\end{equation}
where the first product is taken over right pointing arrows and the second over left pointing arrows.

%  Note that minimal lacing diagrams for $Q$ can be naturally identified with the subset of minimal lacing diagrams for $\tilde{Q}$ which have no crossings at inverted arrows (equivalently, the identity permutation assigned to inverted arrows), and the same is true for $K$-theoretic lacing diagrams.

\begin{theorem}[Component formula]\label{thm:component}
For a type $A$ quiver $Q$ of arbitrary orientation, the $K$-polynomial of a quiver locus $\zO \subseteq \mathtt{rep}_Q(\bd)$ is given by
\begin{equation}\label{eq:componentformula}
K\cQ_\zO (\bu) = \sum_{\bw \in KW(\zO)} (-1)^{|\bw| - \codim\zO} \fG_\bw(\bu).
% \quad \text{and} \quad \cQ_\zO (\bt - \bs)   =\sum_{\bw \in W(\zO)} \fS_\bw(\bt; \bs).
\end{equation}
The formula for $\cQ_\zO (\bu)$ is obtained by replacing each $\fG$ with $\fS$ and $KW(\zO)$ with $W(\zO)$.
\end{theorem}

\begin{proof}
Applying Proposition \ref{prop:alphabetsub} to \eqref{eq:bicomponentformula} for $\tilde{\zO}$, we \textit{a priori} have a sum over $KW(\tilde{\zO})$.
Let $\beta_r$ be an inverted arrow of $\tq$ (the case of $\za_r$ is similar) and let $\tilde{\bw} \in KW(\tilde{\zO})$.  Then the factor of $sub(\fG_{\tilde{\bw}}(\bt; \bs))$ corresponding to $\beta_r$ is of the form %$\fG_w(\bs^r; \bs^r)$ 
$\fG_w(\bu^{\nu(y_r)}; \bu^{\nu(x_r)})$ for some permutation $w$. Since $\beta_r$ is an inverted arrow, $\nu(y_r) = \nu(x_r)$, and $\fG_w(\bu^{\nu(y_r)}; \bu^{\nu(x_r)})$ is 0 unless $w=1$ as in \eqref{eq:BRcor}. 
Applying this reasoning to all inverted arrows, we see that $sub(\fG_{\tilde{\bw}}(\bt; \bs))=0$ unless $\tilde{\bw}$ has no crossings over inverted arrows.   The set of such $\tilde{\bw}$ is clearly in bijection with $KW(\zO)$ by contracting the lace edges over inverted arrows and identifying the vertices at each end.  Given such $\tilde{\bw}$, denote by $\bw$ the contraction: then \eqref{eq:BRcor} along with the definition \eqref{eq:gwdef} shows that $sub(\fG_{\tilde{\bw}}(\bt; \bs)) = \fG_\bw (\bu)$.
Furthermore, for such $\tilde{\bw}$ we have that $\tilde{\bw}$ is minimal if and only if $\bw$ is, so $|\tilde{\bw}|=|\bw|$ implies that $\codim\tilde{\zO} = \codim \zO$, completing the proof.
\end{proof}

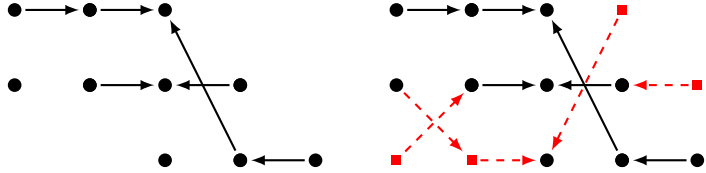
\begin{figure}[h]
\begin{tikzpicture}[point/.style={shape=circle,fill=black,scale=.5pt,outer sep=3pt},>=latex] % lace1
\node[point] (1b) at (0,1) {};
\node[point] (1c) at (0,2) {};
\node[point] (2b) at (1,1) {};
\node[point] (2c) at (1,2) {};
\node[point] (3b) at (1,1) {};
\node[point] (3c) at (1,2) {};
\node[point] (4a) at (2,0) {};
\node[point] (4b) at (2,1) {};
\node[point] (4c) at (2,2) {};
\node[point] (5a) at (3,0) {};
\node[point] (5b) at (3,1) {};
\node[point] (6a) at (3,0) {};
\node[point] (6b) at (3,1) {};
\node[point] (7a) at (4,0) {};
  
\path[->,thick]
   (1c) edge (2c)
   (3b) edge (4b)
   (3c) edge (4c)
   (5a) edge (4c)
   (5b) edge (4b)
   (7a) edge (6a);
  
  \end{tikzpicture}
\qquad
\begin{tikzpicture}[point/.style={shape=circle,fill=black,scale=.5pt,outer sep=3pt},epoint/.style={shape=rectangle,fill=red,scale=.5pt,outer sep=3pt},>=latex] %lace 1 completion
\node[epoint] (1a) at (0,0) {};
\node[point] (1b) at (0,1) {};
\node[point] (1c) at (0,2) {};
\node[epoint] (2a) at (1,0) {};
\node[point] (2b) at (1,1) {};
\node[point] (2c) at (1,2) {};
\node[epoint] (3a) at (1,0) {};
\node[point] (3b) at (1,1) {};
\node[point] (3c) at (1,2) {};
\node[point] (4a) at (2,0) {};
\node[point] (4b) at (2,1) {};
\node[point] (4c) at (2,2) {};
\node[point] (5a) at (3,0) {};
\node[point] (5b) at (3,1) {};
\node[epoint] (5c) at (3,2) {};
\node[point] (6a) at (3,0) {};
\node[point] (6b) at (3,1) {};
\node[point] (7a) at (4,0) {};
\node[epoint] (7b) at (4,1) {};
  
\path[->,thick]
   (1b) edge[red,dashed] (2a)
   (1c) edge (2c)
   (1a) edge[red,dashed] (2b)
   (3a) edge[red,dashed] (4a)
   (3b) edge (4b)
   (3c) edge (4c)
   (5a) edge (4c)
   (5b) edge (4b)
   (5c) edge[red,dashed] (4a)
   (7b) edge[red,dashed] (6b)
   (7a) edge (6a);
  \end{tikzpicture}
    \caption{Minimal lacing diagram for $Q$ from contracting inverted arrows}
\label{fig:laces2}
\end{figure}

Continuing our running example, we take dimension vector $\bd = (\bd(6),\dotsc,\bd(1))= (0,2,2,3,2,1)$ for the quiver $Q$ in \eqref{eq:arbitraryq}.  Then, ignoring the vertex where $\bd$ is 0, the associated bipartite quiver $\tq$ is the same as in the running example from previous sections \eqref{eq:bipartiteEx}. The bipartite $\tilde{Q}$-lacing diagram of Figure \ref{fig:laces} has no crossings over inverted arrows, so we may contract them.
In Figure \ref{fig:laces2}, we see the corresponding $Q$-lacing diagram for this orbit and its extension.

\bibliographystyle{alpha}
\bibliography{pipequiver}

\begin{thebibliography}{KMS06}

\bibitem[ADF85]{AdF}
S.~Abeasis and A.~Del~Fra.
\newblock Degenerations for the representations of a quiver of type {${A}_m$}.
\newblock {\em J. Algebra}, 93(2):376--412, 1985.

\bibitem[AJS94]{AJS}
H.~H. Andersen, J.~C. Jantzen, and W.~Soergel.
\newblock Representations of quantum groups at a {$p$}th root of unity and of
  semisimple groups in characteristic {$p$}: independence of {$p$}.
\newblock {\em Ast\'erisque}, (220):321, 1994.

\bibitem[All14]{allman}
Justin Allman.
\newblock Grothendieck classes of quiver cycles as iterated residues.
\newblock {\em Michigan Math. J.}, 63(4):865--888, 2014.

\bibitem[ASS06]{assemetal}
Ibrahim Assem, Daniel Simson, and Andrzej Skowro{\'n}ski.
\newblock {\em Elements of the representation theory of associative algebras.
  {V}ol. 1}, volume~65 of {\em London Mathematical Society Student Texts}.
\newblock Cambridge University Press, Cambridge, 2006.
\newblock Techniques of representation theory.

\bibitem[BB93]{BB}
Nantel Bergeron and Sara Billey.
\newblock R{C}-graphs and {S}chubert polynomials.
\newblock {\em Experiment. Math.}, 2(4):257--269, 1993.

\bibitem[BF99]{BFchernclass}
Anders~Skovsted Buch and William Fulton.
\newblock Chern class formulas for quiver varieties.
\newblock {\em Invent. Math.}, 135(3):665--687, 1999.

\bibitem[BFR05]{BFR}
Anders~S. Buch, L{{\'a}}szl{{\'o}}~M. Feh{{\'e}}r, and Rich{{\'a}}rd
  Rim{{\'a}}nyi.
\newblock Positivity of quiver coefficients through {T}hom polynomials.
\newblock {\em Adv. Math.}, 197(1):306--320, 2005.

\bibitem[Bil99]{billey}
Sara~C. Billey.
\newblock Kostant polynomials and the cohomology ring for {$G/B$}.
\newblock {\em Duke Math. J.}, 96(1):205--224, 1999.

\bibitem[Bon96]{MR1402728}
K.~Bongartz.
\newblock On degenerations and extensions of finite-dimensional modules.
\newblock {\em Adv. Math.}, 121(2):245--287, 1996.

\bibitem[Bon98]{Bongartzsurvey}
Klaus Bongartz.
\newblock Some geometric aspects of representation theory.
\newblock In {\em Algebras and modules, {I} ({T}rondheim, 1996)}, volume~23 of
  {\em CMS Conf. Proc.}, pages 1--27. Amer. Math. Soc., Providence, RI, 1998.

\bibitem[BR04]{BRspec}
Anders~S. Buch and Rich{\'a}rd Rim{\'a}nyi.
\newblock Specializations of {G}rothendieck polynomials.
\newblock {\em C. R. Math. Acad. Sci. Paris}, 339(1):1--4, 2004.

\bibitem[BR07]{MR2306279}
Anders~Skovsted Buch and Rich{{\'a}}rd Rim{{\'a}}nyi.
\newblock A formula for non-equioriented quiver orbits of type {$A$}.
\newblock {\em J. Algebraic Geom.}, 16(3):531--546, 2007.

\bibitem[Buc02]{MR1932326}
Anders~Skovsted Buch.
\newblock Grothendieck classes of quiver varieties.
\newblock {\em Duke Math. J.}, 115(1):75--103, 2002.

\bibitem[Buc05]{MR2114821}
Anders~Skovsted Buch.
\newblock Alternating signs of quiver coefficients.
\newblock {\em J. Amer. Math. Soc.}, 18(1):217--237 (electronic), 2005.

\bibitem[Buc08]{MR2492443}
Anders~Skovsted Buch.
\newblock Quiver coefficients of {D}ynkin type.
\newblock {\em Michigan Math. J.}, 57:93--120, 2008.
\newblock Special volume in honor of Melvin Hochster.

\bibitem[BZ02]{MR1967381}
Grzegorz Bobi{\'n}ski and Grzegorz Zwara.
\newblock Schubert varieties and representations of {D}ynkin quivers.
\newblock {\em Colloq. Math.}, 94(2):285--309, 2002.

\bibitem[Eis95]{eisenbud}
David Eisenbud.
\newblock {\em Commutative algebra, with a view toward algebraic geometry},
  volume 150 of {\em Graduate Texts in Mathematics}.
\newblock Springer-Verlag, New York, 1995.

\bibitem[FK94]{MR2307216}
Sergey Fomin and Anatol~N. Kirillov.
\newblock Grothendieck polynomials and the {Y}ang-{B}axter equation.
\newblock In {\em Formal power series and algebraic combinatorics/{S}\'eries
  formelles et combinatoire alg\'ebrique}, pages 183--189. DIMACS, Piscataway,
  NJ, 1994.

\bibitem[FK96]{FK}
Sergey Fomin and Anatol~N. Kirillov.
\newblock The {Y}ang-{B}axter equation, symmetric functions, and {S}chubert
  polynomials.
\newblock In {\em Proceedings of the 5th {C}onference on {F}ormal {P}ower
  {S}eries and {A}lgebraic {C}ombinatorics ({F}lorence, 1993)}, volume 153,
  pages 123--143, 1996.

\bibitem[FL94]{fultonlascoux}
William Fulton and Alain Lascoux.
\newblock A {P}ieri formula in the {G}rothendieck ring of a flag bundle.
\newblock {\em Duke Math. J.}, 76(3):711--729, 1994.

\bibitem[FR02]{FRdegenlocithom}
L{\'a}szl{\'o} Feh{\'e}r and Rich{\'a}rd Rim{\'a}nyi.
\newblock Classes of degeneracy loci for quivers: the {T}hom polynomial point
  of view.
\newblock {\em Duke Math. J.}, 114(2):193--213, 2002.

\bibitem[Ful92]{Fulton}
William Fulton.
\newblock Flags, {S}chubert polynomials, degeneracy loci, and determinantal
  formulas.
\newblock {\em Duke Math. J.}, 65(3):381--420, 1992.

\bibitem[Gra02]{graham}
William Graham.
\newblock Equivariant {$K$}-theory and {S}chubert varieties.
\newblock Preprint, 2002.

\bibitem[HH11]{HH}
J\"urgen Herzog and Takayuki Hibi.
\newblock {\em Monomial ideals}, volume 260 of {\em Graduate Texts in
  Mathematics}.
\newblock Springer-Verlag London, Ltd., London, 2011.

\bibitem[HZ14]{HZsurvey}
B.~Huisgen-Zimmermann.
\newblock Fine and coarse moduli spaces in the representation theory of finite
  dimensional algebras.
\newblock In {\em Expository lectures on representation theory}, volume 607 of
  {\em Contemp. Math.}, pages 1--34. Amer. Math. Soc., Providence, RI, 2014.

\bibitem[Kin18]{KinserICRA}
Ryan Kinser.
\newblock K-polynomials of type a quiver orbit closures and lacing diagrams.
\newblock In {\em Representations of Algebras}, volume 705 of {\em Contemp.
  Math.}, pages 99--114. Amer. Math. Soc., Providence, RI, 2018.

\bibitem[KM04]{MR2047852}
Allen Knutson and Ezra Miller.
\newblock Subword complexes in {C}oxeter groups.
\newblock {\em Adv. Math.}, 184(1):161--176, 2004.

\bibitem[KM05]{KM05}
Allen Knutson and Ezra Miller.
\newblock Gr\"obner geometry of {S}chubert polynomials.
\newblock {\em Ann. of Math. (2)}, 161(3):1245--1318, 2005.

\bibitem[KMS06]{KMS}
Allen Knutson, Ezra Miller, and Mark Shimozono.
\newblock Four positive formulae for type {$A$} quiver polynomials.
\newblock {\em Invent. Math.}, 166(2):229--325, 2006.

\bibitem[Knua]{Knutson:2009aa}
Allen Knutson.
\newblock {F}robenius splitting and {M}\"obius inversion.
\newblock \href{http://arxiv.org/abs/0902.1930v1}{\texttt{arxiv:0902.1930v1}}.

\bibitem[Knub]{Knutson:2009bb}
Allen Knutson.
\newblock {F}robenius splitting, point counting, and degeneration.
\newblock \href{http://arxiv.org/abs/0911.4941}{\texttt{arxiv:0911.4941}}.

\bibitem[KR15]{KR}
Ryan Kinser and Jenna Rajchgot.
\newblock Type {$A$} quiver loci and {S}chubert varieties.
\newblock {\em J. Commut. Algebra}, 7(2):265--301, 2015.

\bibitem[KS11]{KontSoib}
Maxim Kontsevich and Yan Soibelman.
\newblock Cohomological {H}all algebra, exponential {H}odge structures and
  motivic {D}onaldson-{T}homas invariants.
\newblock {\em Commun. Number Theory Phys.}, 5(2):231--352, 2011.

\bibitem[LM98]{LMdegen}
V.~Lakshmibai and Peter Magyar.
\newblock Degeneracy schemes, quiver schemes, and {S}chubert varieties.
\newblock {\em Internat. Math. Res. Notices}, (12):627--640, 1998.

\bibitem[LS82]{LS}
Alain Lascoux and Marcel-Paul Sch{\"u}tzenberger.
\newblock Structure de {H}opf de l'anneau de cohomologie et de l'anneau de
  {G}rothendieck d'une vari\'et\'e de drapeaux.
\newblock {\em C. R. Acad. Sci. Paris S\'er. I Math.}, 295(11):629--633, 1982.

\bibitem[Lus90]{MR1035415}
G.~Lusztig.
\newblock Canonical bases arising from quantized enveloping algebras.
\newblock {\em J. Amer. Math. Soc.}, 3(2):447--498, 1990.

\bibitem[Mil05]{MR2137947}
Ezra Miller.
\newblock Alternating formulas for {$K$}-theoretic quiver polynomials.
\newblock {\em Duke Math. J.}, 128(1):1--17, 2005.

\bibitem[MS05]{MS}
Ezra Miller and Bernd Sturmfels.
\newblock {\em Combinatorial commutative algebra}, volume 227 of {\em Graduate
  Texts in Mathematics}.
\newblock Springer-Verlag, New York, 2005.

\bibitem[Rim]{RimanyiCOHA}
Rich{\'a}rd Rim{\'a}nyi.
\newblock On the cohomological {H}all algebra of {D}ynkin quivers.
\newblock \href{http://arxiv.org/abs/1303.3399}{\texttt{arxiv:1303.3399}}.

\bibitem[Rim14]{MR3239295}
R.~Rim{\'a}nyi.
\newblock Quiver polynomials in iterated residue form.
\newblock {\em J. Algebraic Combin.}, 40(2):527--542, 2014.

\bibitem[Rin90]{Rhallalgebras}
Claus~Michael Ringel.
\newblock Hall algebras and quantum groups.
\newblock {\em Invent. Math.}, 101(3):583--591, 1990.

\bibitem[RZ13]{MR3008913}
Christine Riedtmann and Grzegorz Zwara.
\newblock Orbit closures and rank schemes.
\newblock {\em Comment. Math. Helv.}, 88(1):55--84, 2013.

\bibitem[Sch14]{Schiffler:2014aa}
Ralf Schiffler.
\newblock {\em Quiver representations}.
\newblock CMS Books in Mathematics/Ouvrages de Math\'ematiques de la SMC.
  Springer, Cham, 2014.

\bibitem[Wil06]{willems}
Matthieu Willems.
\newblock {$K$}-th\'eorie \'equivariante des tours de {B}ott. {A}pplication \`a
  la structure multiplicative de la {$K$}-th\'eorie \'equivariante des
  vari\'et\'es de drapeaux.
\newblock {\em Duke Math. J.}, 132(2):271--309, 2006.

\bibitem[WY08]{MR2422304}
Alexander Woo and Alexander Yong.
\newblock Governing singularities of {S}chubert varieties.
\newblock {\em J. Algebra}, 320(2):495--520, 2008.

\bibitem[WY12]{wooyong}
Alexander Woo and Alexander Yong.
\newblock A {G}r\"obner basis for {K}azhdan-{L}usztig ideals.
\newblock {\em Amer. J. Math.}, 134(4):1089--1137, 2012.

\bibitem[Yon05]{yongComponent}
Alexander Yong.
\newblock On combinatorics of quiver component formulas.
\newblock {\em J. Algebraic Combin.}, 21(3):351--371, 2005.

\bibitem[Zel85]{Zgradednilp}
A.~V. Zelevinsky.
\newblock Two remarks on graded nilpotent classes.
\newblock {\em Uspekhi Mat. Nauk}, 40(1(241)):199--200, 1985.

\bibitem[Zwa11]{Zwarasurvey}
Grzegorz Zwara.
\newblock Singularities of orbit closures in module varieties.
\newblock In {\em Representations of algebras and related topics}, EMS Ser.
  Congr. Rep., pages 661--725. Eur. Math. Soc., Z\"urich, 2011.

\end{thebibliography}

\end{document}